\documentclass[letter,11pt]{article}

\usepackage{listings}

\usepackage{tikz}
\usetikzlibrary{shapes, positioning}
\usepackage{graphicx}

\usepackage{algorithm}
\usepackage{algpseudocode}

\usepackage{natbib}
\usepackage{amssymb,amsmath,xcolor,graphicx,xspace,colortbl,rotating}
\usepackage[raggedrightboxes]{ragged2e}
\usepackage{textcomp}
\usepackage{appendix}
\usepackage{mathrsfs}
\usepackage{bm}
\usepackage{boxedminipage}
\usepackage{color}
\usepackage{endnotes}
\usepackage[margin=1in]{geometry}
\usepackage[doublespacing]{setspace}
\usepackage{tabulary}
\usepackage{varioref}
\usepackage{wrapfig}
\graphicspath{{../graphics/}{../tcache/}{../gcache/}}
\DeclareGraphicsExtensions{.pdf,.eps,.ps,.png,.jpg,.jpeg}
\usepackage{amsmath}

\newtheorem{theorem}{Theorem}[section]

\newtheorem{assumption}{Assumption}[section]

\newtheorem{corollary}{Corollary}[section]

\newtheorem{definition}{Definition}[section]
\newtheorem{example}{Example}[section]
\newtheorem{exercise}{Exercise}[section]
\newtheorem{lemma}{Lemma}[section]

\newtheorem{proposition}{Proposition}[section]
\newtheorem{remark}{Remark}[section]

\newenvironment {proof}[1][Proof]{\noindent \textbf {#1.} }{\ \rule {0.5em}{0.5em}}
\begin{document}

\title{A Course in Dynamic Optimization}
\author{Bar Light\protect\thanks{Business School and Institute of Operations Research and Analytics, National University of Singapore, Singapore. e-mail: \textsf{barlight@nus.edu.sg} }}  
\maketitle

\noindent \noindent \textsc{Abstract}:
\begin{quote}

These lecture notes are derived from a graduate-level course in dynamic optimization, offering an introduction to techniques and models extensively used in management science, economics, operations research, engineering, and computer science. The course emphasizes the theoretical underpinnings of discrete-time dynamic programming models and advanced algorithmic strategies for solving these models. Unlike typical treatments, it provides a proof for the principle of optimality for upper semi-continuous dynamic programming, a middle ground between the  simpler countable state space  case \cite{bertsekas2012dynamic}, and the involved universally measurable case \cite{bertsekas1996stochastic}. This approach is sufficiently rigorous to include important examples such as dynamic pricing, consumption-savings, and inventory management models. The course also delves into the properties of value and policy functions, leveraging classical results \cite{topkis1998supermodularity} and recent developments. Additionally, it offers an introduction to reinforcement learning, including a formal proof of the convergence of Q-learning algorithms. Furthermore, the notes delve into policy gradient methods for the average reward case, presenting a convergence result for the tabular case in this context. This result is simple and similar to the discounted case but appears to be new.

\end{quote}

\newpage
This course, based  on a graduate-level course taught at Tel Aviv University in 2024, provides an introduction to dynamic optimization theory and algorithms. 
This course places a strong emphasis on exploring the theoretical underpinnings of dynamic programming models and advanced algorithmic strategies for solving these models. The main objectives of the course are to:
\begin{itemize}
    \item Understand the fundamental principles of dynamic optimization.
    \item Apply dynamic programming to solve problems.
    \item Explore algorithms to solve dynamic optimization problems.
\end{itemize}

To fully benefit from this course, students are expected to have a solid background in  optimization, Markov chains, basic probability theory, basic real analysis, as well as some background in machine learning and Python programming. These prerequisites ensure that students are prepared to grasp the advanced concepts that will be covered.

There are a few notable differences from standard treatments of dynamic optimization, which I will describe below:

\subsection*{1. Presentation of the Principle of Optimality}

The principle of optimality for dynamic programming is generally applicable, but it is often taught for problems with finite or countable state spaces to avoid measure-theoretic complexities, as seen in the important textbook by Bertsekas (2012). This approach cannot be applied to classic models such as inventory management, consumption-savings models, and dynamic pricing models that have continuous state spaces, leaving students unaware of the possible challenges involved in studying dynamic programming models with general state spaces.

Conversely, the typical general state space treatment of universally measurable dynamic programming presented by \cite{bertsekas1996stochastic} is quite involved. In this course, we provide conditions and a self-contained simple proof that establish when the principle of optimality for discounted dynamic programming is valid, based on \cite{light2024principle}. We then focus on upper semi-continuous dynamic programming, which serves as a middle ground between the universally measurable case and the countable state space case. Importantly, the upper semi-continuous case is sufficiently rigorous to include important examples such as consumption-saving problems, inventory management, and dynamic pricing.

\subsection*{2. Focus on Properties of the Value and Policy Functions}

A portion of the course is dedicated to exploring the properties of the value and policy functions. I provide a systematic way to prove properties of the value function such as concavity, monotonicity, supermodularity and differentiability, and derive monotonicity conditions regarding the policy function. Some of these results are based on the foundational works  of \cite{topkis1978minimizing,topkis1998supermodularity}, while others are more recent developments from \cite{light2021stochastic}.  

\subsection*{3. Introduction to Reinforcement Learning}

The course also includes an introduction to reinforcement learning, focusing on  theoretical aspects. I provide the proof of the convergence of Q-learning type algorithms to the optimal Q-function.  This technical analysis follows the classic treatment of \cite{bertsekas1996neuro} but offers a simplified and more focused perspective. I also introduce Q-learning with function approximation, and with online search. Additionally, I present policy gradient methods using the average reward case, which is somewhat more natural and easier to analyze  than the discounted case. I discuss policy gradient theorem and the performance difference lemma for the average reward case and simulation methods with policy gradient using the average reward case. Furthermore, I provide the proof of convergence of the policy gradient method to the optimal policy in the tabular case. This result appears to be new in the average reward case, albeit similar (and simpler) than the well-studied discounted case.

\vspace{3mm}

* The lectures may have mistakes and typos. Please contact me if you found any.

* Exercises are given in the end of each lecture. Partial solutions are given in the end of the document. 
\newpage

\tableofcontents

\newpage

\section{Lecture 1: Introduction, Metric Spaces, Probability Spaces}

\subsection{Introduction}

Dynamic optimization plays a crucial role in various domains, offering powerful tools to tackle complex, time-dependent problems. Below, we discuss seven applications of dynamic optimization:

1. \textbf{Inventory problems}: Dynamic optimization is essential for inventory management, helping businesses determine optimal inventory levels to minimize holding, ordering, and shortage costs. Techniques like dynamic programming and stochastic models are used to make informed decisions on when and how much to order, ensuring that businesses run efficiently while satisfying demands.

2. \textbf{Dynamic pricing:} Dynamic pricing is an important application of dynamic optimization that involves adjusting prices for goods or services based on factors such as demand, supply, inventories, and time. Many businesses with fluctuating demand  such as in the airline industry, ride-sharing platforms like Uber, accommodation platforms like Airbnb, and retailers in marketplaces like Amazon use dynamic pricing methods. 

For example, in the airline industry, dynamic pricing is used to manage seat inventory and maximize revenues. Airlines consider numerous factors, such as available capacity, demand, time until departure, competition, and historical data, to optimize the prices of their tickets. As a result, ticket prices change over time, reflecting the dynamic nature of these factors. For instance, prices may increase as the departure date approaches and the demand for the remaining seats rises. Alternatively, airlines may lower prices temporarily to stimulate demand during periods of low interest. Dynamic optimization techniques and algorithms help airlines make data-driven pricing decisions that maximize their profits.

Similarly, Airbnb hosts can utilize dynamic pricing to optimize their rental income. Hosts can adjust the price of their listings based on factors such as seasonality, local events, competition, and booking patterns. For example, they might increase prices during peak travel seasons or major events when demand for accommodations is high. On the other hand, they could lower prices during off-peak periods to attract more bookings.

3. \textbf{Playing games}: Dynamic optimization has impacted applied game theory and artificial intelligence, enabling the development of highly advanced game-playing algorithms. AlphaZero, an example of such algorithms, employs a combination of deep neural networks and dynamic programming algorithms to master games like Chess, Shogi, and Go. It has not only outperformed traditional engines but also demonstrated a creative and human-like approach to gameplay.

4. \textbf{Recommendation system}: Dynamic optimization is employed to balance exploration and exploitation, leading to more effective content recommendations and personalized learning. A natural framework for this is contextual bandits where the algorithm not only learns from past actions but also takes into account additional information or context (like user preferences or item features) to make decisions. This approach is particularly effective in recommendation systems where each user-item interaction provides a unique context that can significantly influence the recommendations' effectiveness (contextual bandits algorithms are widely used in practice).

5. \textbf{Consumption-savings problems}: In economics, dynamic optimization is used to model and analyze consumption-savings problems, where individuals must decide how much to consume and save over time. Techniques such as the Bellman equation and Euler equations are employed to derive optimal consumption, offering insights into the effects of various factors like interest rates, income, and uncertainty on economic behavior.

6. \textbf{Shortest path problems}: Dynamic optimization is integral to solving shortest path problems, where the goal is to find the most efficient route between two points in a graph or network. Applications range from transportation and logistics, where optimal routes minimize travel time and cost, to computer science and telecommunications, where data packets are sent through the most efficient paths in networks. Dijkstra's and Bellman-Ford algorithms are prominent examples of dynamic optimization methods used for shortest path problems.

7. \textbf{Experimentation}: Dynamic optimization techniques can be used instead of traditional A/B testing. 
As an example, a news website can dynamically adjust the placement of articles. Each layout option is a possible action, and a dynamic optimization algorithm aims to maximize click through rate by learning from user behavior.

\subsection{A Few Key Notions in Dynamic Optimization}

\textbf{Computational.} Dynamic programming techniques break complex problems into smaller, overlapping subproblems and store their solutions for reuse. This exploits the problem's structure, often reducing the computational complexity and time required to find the optimal solution.

As an example, consider the very simple optimization problem $\max _{ \boldsymbol{x} \in A  } x_{1} \cdot x_{2} \cdot \ldots \cdot x_{n}$
where $A = \{ \boldsymbol{x} \in \mathbb{R}^{n} : \sum x_{i} = a, x_{i} \geq 0 \}$. It is easy to see that the solution is $x_{i} = a/n$. How this problem can relate to dynamic optimization? Can we decompose the problem into subproblems? 

Let $v_{n}(a) = \max _{ \boldsymbol{x} \in A  } x_{1} \cdot x_{2} \cdot \ldots \cdot x_{n}$. Note that $v_{n}(a) = \max _{0 \leq y \leq a } yv_{n-1}(a-y) $. Hence, we have reduced the computation of $v_{n}$ ``the value function" to the computation of a sequence of functions in one variable.

\textbf{Analytical properties of the optimal value function.} 
Structural properties of the optimal value function (which is the maximum expected cumulative reward that can be achieved, following the optimal sequence of actions) and of the optimal actions play a crucial role in dynamic optimization, as they provide insights into the underlying problem structure, enable the development of efficient algorithms, and facilitate the analysis of optimal solutions. Some of the key structural properties of the optimal value function are:

1. Optimal Substructure: This property implies that the optimal solution to a problem can be constructed from optimal solutions to its smaller subproblems. In the context of dynamic optimization, this means that the optimal value function for a given state can be found by combining the best possible actions for the current state and the optimal value functions for subsequent states. The optimal substructure property facilitates the use of dynamic programming techniques, such as value iteration, to solve complex optimization problems.

2. Monotonicity and/or Convexity: In some dynamic optimization problems, the optimal value function exhibits monotonicity and/or convexity properties. These properties provide theoretical insights regarding decisions and can be leveraged to simplify computations, develop more efficient algorithms, and ensure the convergence of iterative methods. For instance, if the optimal value function is convex, one can apply techniques from convex analysis to solve the problem more efficiently.

3. Uniqueness and Existence: The uniqueness and existence of the optimal value function are critical for ensuring that a well-defined optimal solution exists for the problem at hand. Establishing these properties helps guarantee that the dynamic optimization methods employed will converge to a unique and meaningful solution.

\textbf{The exploration and exploitation trade-off.} In dynamic optimization problems there is a  challenge of balancing between two conflicting objectives: acquiring new information to make better-informed decisions in the future (exploration) and using the existing knowledge to maximize immediate rewards (exploitation).

Exploration involves taking actions that may not appear optimal based on current information but can potentially be optimal. It is essential for discovering better solutions and improving decision-making over time. However, excessive exploration can lead to suboptimal short-term performance, as the algorithm may forego immediate rewards in pursuit of uncertain future gains.

Exploitation, on the other hand, focuses on leveraging the current knowledge to maximize immediate rewards. It involves selecting actions that appear to be the best based on available information, thus ensuring optimal performance in the short term. However, overemphasis on exploitation may result in suboptimal long-term outcomes, as the algorithm may miss out on better solutions that could be discovered through exploration.

\newpage 

\subsection{Metric Spaces} \label{Sec:metricspace}
We first recall the definition of a metric space and the basic concepts of open sets and Cauchy sequences. 

\begin{definition}
    A metric $d$ on a nonempty set $X$ is a function $d:X \times X \rightarrow \mathbb{R}$ satisfying three properties:

    1. Positivity: $d(x,y) \geq 0$ for all $x,y \in X$, and $d(x,y)=0$ if and only if $x=y$. 

    2. Symmetry: $d(x,y) = d(y,x)$ for all $x,y \in X$. 

    3. Triangle inequality: $d(x,y) \leq d(x,z) + d(z,y)$ for all $x,y,z \in X$. 

    The pair $(X,d)$ is called a metric space. 
\end{definition}

For $x \in X$, the open ball around $x$ with radius $r$ is given by $B(x,r) = \{y \in X: d(x,y) < r \}$. A subset $A$ of $X$ is called \textbf{open} if for every $x \in A$ there exists $r>0$ such that $B(x,r) \subseteq A$. $A$ is \textbf{closed} if the complement of $A$ given by $X \setminus A$ is open. 

A sequence  $\{x_{n}\}$ of a metric space $(X,d)$ is said to \textbf{converge}  to $x \in X$ (typically denoted by $x_{n} \rightarrow x$ or $\lim x_{n} = x$) if $ \lim d(x_{n},x) =0$. 

A function $f:(X,d_{x}) \rightarrow (Y,d_{y})$ between two metric spaces is called \textbf{continuous} if $f^{-1} (O)  := \{ x \in X : f(x) \in O \}$ is an open subset of $X$ whenever $O$ is an open subset of $Y$.\footnote{You are probably more aware of  the equivalent definitions: $f$ is continuous if $\lim x_{n} = x$ holds in $X$ implies that $\lim f(x_{n}) = f(x)$ holds in $Y$ or $f$ is continuous if for every point $a \in X$ and every $\epsilon >0$ there exists $\delta >0$ such that $d_{y} (f(x) , f(a) ) < \epsilon$ whenever $d_{x}(x,a) < \delta$.}

A family of sets $\{A_{i} \}_{i \in I}$ such that $A_{i} \subseteq X$ is said to be a \textbf{cover} of $X$ if $X \subseteq \cup _{i \in I} A_{i} $. If $A_{i}$ are open for each $i$, then the cover is called an open cover. If a subfamily of $\{A_{i} \}_{i \in I}$ also covers $X$ then it is called a subcover. 

A set $A \subseteq X$ is \textbf{compact} if every open cover of $A$ can be reduced to a finite subcover.  
Recall that $A$ is compact in a metric space if and only if every sequence in $A$ has a subsequence which converges to a point of $A$. 
In addition, if $A \subseteq \mathbb{R}^{n}$ then $A$ is compact if and only if it is closed and bounded.

For every set $A$ in a metric space, we let $\operatorname{cl} (A)$ to be the smallest closed set that includes $A$. 
A metric space $(X,d)$ is called \textbf{separable} if there exists a countable subset $S$ such that $cl(S) = X$ (e.g., $\mathbb{R}$ with $d(x,y) = 
|x-y| $ is a separable metric space as we can let $S$ to be set of rational numbers). 

Let $B(S)$ be the space of all real-valued bounded functions on $S$. We define for each $f,g \in B(S)$ the metric $d_{\infty} (f,g):= \sup _{s \in S} |f(s) - g(s)|$. 
Unless otherwise stated, we will endow $B(S)$ with the metric $d_{\infty} (f,g)$. 

A sequence $\{x_{n}\}$ of a metric space $(X,d)$ is called a \textbf{Cauchy} sequence if for all $\epsilon >0$, there exists $N$ such that $d(x_{n},x_{m}) < \epsilon$ for all $n,m > N$. 

A metric space $(X,d)$ is called \textbf{complete} if every Cauchy sequence converges in $X$ (e.g., $\mathbb{R}^{n}$). 
 $B(S)$ is a complete metric space (see exercises).

Recall that a \textbf{fixed-point} of a mapping $T : X \rightarrow X$ is a $x \in S$ such that $T(x) = x$.  We will need the following result which is sometimes called the Banach fixed-point theorem. 

\begin{proposition}\label{prop:Banach}
Let $(X,d)$ be a complete metric space. Let $T : X \rightarrow X$ be a mapping that is $L$-contraction, i.e., $d(T(x),T(y)) \leq L d(x,y)$ for some $0<L<1$ and for all $x,y \in X$. Then $T$ has a unique fixed point. 

\end{proposition}

\begin{proof} First, $T$ has at most one fixed-point. If $T$ has two fixed points $x,y$, then $d(x,y) = d(T(x),T(y)) \leq Ld(x,y)$ which implies $d(x,y)=0$, i.e., $x=y$.

It remains to establish the existence of a fixed point. 

Let $x_{0} \in X$ and consider the sequence $\{x_{n}\}$ given by $x_{n+1}=T(x_{n})$. We claim that  $\{x_{n}\}$ is a Cauchy sequence.  

We have $d(x_{m+1},x_{m}) = d(T(x_{m}),T(x_{m-1})) \leq L d(x_{m},x_{m-1}) $ so continuing recursively we conclude that $d(x_{m+1},x_{m}) \leq L^{m}d(x_{1},x_{0}) $. Hence, for $n \geq m$ we have 
$$ d(x_{n},x_{m}) \leq \sum _{i=m}^{n-1} d(x_{i+1},x_{i}) \leq \sum _{i=m}^{n-1} L^{i} d(x_{1},x_{0}) \leq \frac{L^{m}}{1-L} d(x_{1},x_{0}).$$
Because $0<L<1$, for all $\epsilon>0$ there exists an $N$ such that for all $n,m \geq N$ we have $d(x_{n},x_{m}) \leq \epsilon$.  Thus, $\{x_{n}\}$ is a Cauchy sequence. 

Because $(X,d)$ is complete the sequence $\{x_{n}\}$ converges to some $x \in X$. We have 
$$ d(x,T(x)) \leq d(x,x_{n}) + d(x_{n},T(x)) = d(x,x_{n}) + d(T(x_{n-1}),T(x)) \leq d(x,x_{n}) + Ld(x_{n-1},x). $$
Taking the limit $n \rightarrow \infty$ in the last inequality implies that $d(x,T(x)) = 0$, i.e., $T(x) = x$ so $x$ is a fixed point of $T$.
\end{proof}

\newpage 

\subsection{Probability Spaces} \label{Sec:Probspace} A probability space is defined by $(\Omega, \mathcal{F}, \mathbb{P})$ where $\Omega$ is the sample space or the space of all possible outcomes, $\mathcal{F}$ is a sigma-algebra and $\mathbb{P}$ is a probability measure.   

The sigma-algebra $\mathcal{F}$ contains the sets (sometimes called events) that can be measured. Formally, $\mathcal{F} \subseteq 2^{\Omega}$ where $2^{\Omega}$ is the power set of $\Omega$  is called a sigma-algebra if the following three properties hold: (1) $B \in \mathcal{F}$ implies $\Omega \setminus B \in \mathcal{F}$, i.e., it is closed under complements. (2) $\Omega \in \mathcal{F}$. (3) If $B_{1},B_{2}\ldots$ is a countable collection of sets in $\mathcal{F}$ then $ \cup _{i} B_{i} \in \mathcal{F}$, i.e., it is closed under countable unions.

The probability measure $\mathbb{P}$ is a function that assigns to any set in $\mathcal{F}$ a number between $0$ and $1$, and satisfies  (1) $\mathbb{P}(\Omega)=1$, $\mathbb{P}(\emptyset) =0$. (2)  If $B_{1},B_{2}\ldots$ is a countable collection of pairwise disjoint sets in $\mathcal{F}$, i.e., $B_{i} \cap B_{j} = \emptyset$ for all $i,j$, then $ \mathbb{P} (\cup _{i} B_{i} ) = \sum _{i} \mathbb{P} (B_{i})$. 

For the rest of these lectures, $\Omega$ will be endowed with a metric $d$ and $(\Omega,d)$ will be a separable metric space.

Let $\mathfrak{B}$ be a collection of subsets of $\Omega$. The smallest sigma-algebra containing all the sets of $\mathfrak{B}$ is called the sigma-algebra that is \textbf{generated} by  $\mathfrak{B}$. For any collection of sets $\mathfrak{B}$ the sigma-algebra that is generated by  $\mathfrak{B}$ exists (see exercises) and is denoted by $\sigma (\mathfrak{B}) $. 

We are now ready to define the Borel sigma-algebra.

 \begin{definition}
 Let $(\Omega,d)$ be a metric space. The \textbf{Borel sigma-algebra} on $\Omega$ is the sigma-algebra that is generated by all the open sets of $\Omega$. A set is called a \textbf{Borel measurable set} if it belongs to the Borel sigma-algebra. \\ A function $f: \Omega \rightarrow \mathbb{R}$ is called a \textbf{Borel measurable function} if $f^{-1}(V) := \{ \omega \in \Omega : f(\omega) \in V \}$ is a Borel  measurable set in $\Omega$ for every open set $V$ in $\mathbb{R}$. 
  \end{definition} 


 \begin{remark}
 When $\Omega $ is countable then we will assume it is endowed with the discrete distance, i.e., $d(x,y)=1$ if $x \neq y$ and $d(x,x)=0$. In this case every set is open (why?), and hence, the Borel sigma-algebra is $2^{\Omega}$ and every function is Borel measurable. 
 \end{remark}

 \begin{remark}
   The Borel sigma-algebra is the smallest sigma-algebra such that every open set is measurable, so it is the smallest sigma-algebra such that every continuous function is measurable. 
 \end{remark}

\begin{remark}
    Suppose that $\Omega = \mathbb{R}$ with the standard distance $d(x,y) = |x-y|$. Then the cardinality of the Borel measurable sets is the same as the cardinality of $\mathbb{R}$, not $2^{\mathbb{R}}$. Consequently, many sets are not Borel measurable..\footnote{However, to construct a non Borel measurable set we have to apply the axiom of choice. }  
\end{remark}

A \textbf{random variable} $X$ is defined as a real-valued measurable function $X: \Omega \rightarrow \mathbb{R}$. The \textbf{law} of a random variable $X$ (also called the probability distribution of $X$) denoted by $\mu$ is defined by $ \mu (B) := \mathbb{P} (\omega \in \Omega: X(\omega) \in B)$.

\textbf{Lebesgue Integral} In a probability space $(\Omega, \mathcal{F}, \mathbb{P})$, the integral of a random variable $X: \Omega \rightarrow \mathbb{R}$ with respect to the probability measure $\mathbb{P}$ is defined using the concept of the Lebesgue integral. The integral is denoted by
\[
\int_{\Omega} X(\omega) \mathbb{P}(d\omega).
\]
which is also called the expected value of the random variable $X$ denoted by $\mathbb{E}(X)$ or $\mathbb{E}_{\mathbb{P}}(X)$. 

\textbf{Simple Functions}
A \textit{simple function} $X$ that assumes finite positive values $c_1, \ldots, c_n$ can be represented as
\[
X = \sum_{i=1}^n c_i \mathbf{1}_{A_i},
\]
where $A_i = \{\omega \in \Omega : X(\omega) = c_i\}$ are measurable sets. The integral of a simple function is then defined by
\[
\int_{\Omega} X(\omega) \mathbb{P}(d\omega) = \sum_{i=1}^n c_i \mathbb{P}(A_i).
\]
For example, if $\Omega = [0,1]$, $\mathbb{P}([a,b) ) = b-a$ for any $1 \geq b>a \geq 0$, and $X (w)= 1$ for $w \in [0,0.5)$ and $0$ otherwise, then $\mathbb{E}(X) = 0.5$.

\textbf{General Random Variables.} For a general non-negative measurable function $X \geq 0$, the Lebesgue integral is defined as the limit of integrals of simple functions that approximate $X$. Specifically,
\[
\int_{\Omega} X(\omega) \mathbb{P}(d\omega) = \lim_{n \rightarrow \infty} \int_{\Omega} \phi_n(\omega) \mathbb{P}(d\omega),
\]
where $\phi_n$ is a sequence of simple function such that $\phi_{n}$ converges monotonically to $X$. For example, $\phi_n$ can be constructed as
\[
\phi_n = \sum_{i=1}^{n 2^n} 2^{-n}(i-1) \mathbf{1}_{B_n^i},
\]
with $B_n^i = \{\omega \in \Omega : (i-1)2^{-n} \leq X(\omega) < i2^{-n}\}$ (see exercises).

We note that it can be shown that the value of the integral of $X$ is independent of the sequence of step functions that approximate $X$, i.e., if $\phi_{n}$ and $\zeta_{n}$ are sequences of step functions that converges monotonically to $f$ with probability 1, then $\lim _{n \rightarrow \infty}  \int  \phi_{n}  (\omega)  \mathbb{P(d \omega) } = \lim _{n \rightarrow \infty} \int \zeta_{n} (\omega) \mathbb{P (d\omega) } $.

Any random variable $X$ can be decomposed into its positive and negative parts:
\[
X^+(\omega) = \max\{0, X(\omega)\}, \quad X^-(\omega) = -\min\{0, X(\omega)\}.
\]
The integral of $X$ exists if at least one of the integrals of $X^+$ or $X^-$ is finite. It is defined as
\[
\int_{\Omega} X(\omega) \mathbb{P}(d\omega) = \int_{\Omega} X^+(\omega) \mathbb{P}(d\omega) - \int_{\Omega} X^-(\omega) \mathbb{P}(d\omega).
\]

\textbf{Bibliography note:} The technical material in Sections \ref{Sec:metricspace} and \ref{Sec:Probspace} be found in most real analysis textbooks (e.g., \cite{aliprantis1998principles}).

\subsection{Exercise 1}

\begin{exercise}
    Consider the example from Lecture 1 of the optimization problem $\max _{ \boldsymbol{x} \in A  } x_{1} \cdot x_{2} \cdot \ldots \cdot x_{n}$
where $A = \{ \boldsymbol{x} \in \mathbb{Z}^{n} : \sum x_{i} = a, x_{i} \geq 0 \}$. Assume that $n$ and $a$ are non-negative integers. 

1. Write in Python a dynamic programming algorithm to solve this problem for general non-negative integers $n,a$.

Hint: Start with the following edge cases: 

\begin{verbatim}
    def max_product(n, a):
    # Edge case: if the total is zero, the product of any numbers must be zero
    if a == 0:
        return 0
    # Edge case: if we only have one number to adjust, it must be 'a'
    if n == 1:
        return a

    # Initialize table where v[i][j] is the max product for i elements summing to j
    v = [[0] * (a + 1) for _ in range(n + 1)]

      # Base case: product of one part is the number itself
    for j in range(a + 1):
        v[1][j] = j
\end{verbatim}

Now fill in the table given the discussion in Lecture 1. 

2.  Explain, intuitively, why this method is more efficient than a brute-force approach that would involve generating all possible ways to partition the sum 
$a$ into $n$
non-negative parts and then calculating the product for each partition. 

If you studied computer science, can you make your argument more formal (what is the (time) complexity of both methods)?
\end{exercise}

\textbf{Metric Spaces}

\begin{exercise} \label{Ex:B(S)_Complete}
  (a)  Show that $(B(S),d_{\infty})$ is indeed a metric space. 

  (b) Show that the metric space $(B(S),d_{\infty})$ is complete. 
\end{exercise}

\begin{exercise}
    Let $A \subseteq X$ be a closed subset of a metric space $(X,d)$. Then $x \in A$ if and only if there exists a sequence $\{x_{n} \}$ such that $\lim x_{n} = x$. 
\end{exercise}

\begin{exercise}
    Let $(X,d)$ be a complete metric space. Then a subset $A \subseteq X$ is closed if and only if $(A,d)$ is a complete metric space. 
\end{exercise}

\begin{exercise}
   In this exercise we will provide an alternative proof for Banach fixed-point theorem that is based on Cantor's intersection theorem.

Define the diameter of a set $A$ by $d(A) = \sup \{ d(x,y): x,y \in A\}$.

   (a) Let $(X,d)$ be a complete metric space and let $\{A_{n}\}$ be a sequence of closed, non-empty subsets of $X$ such that $A_{n+1} \subseteq A_{n}$ for each $n$ and $\lim d(A_{n}) = 0 $. Then $\cap _{n=1}^{\infty} A_{n}$ consists of precisely one element.

(b) Use part (a) to prove Banach fixed-point theorem. 

Hint: Define the function $f(x) = d(x,Tx)$ and consider the sets $A_{n} = \{x \in X: f(x) \leq 1/n \} $ where $T$ is the contraction operator defined in the statement of the Banach fixed-point theorem. 
\end{exercise}

\begin{exercise}
     Let $X$ be a non-empty set and let $T: X \rightarrow X$, $L \in (0,1)$. Then the following statements are equivalent: 

     (i) There exists a complete metric $d$ for $X $ such that $d(Tx,Ty) \leq L d(x,y)$ for all $x,y \in X$. 

     (ii) There exists a function $f:X \rightarrow \mathbb{R}_{+}$ such that $f^{-1}(0)$ is a singleton and $f(Tx) \leq L f(x)$. 

     Hint: For $(i) \rightarrow (ii)$ use Banach fixed-point theorem. 

     For $(ii) \rightarrow (i)$ consider $d(x,y) = f(x) + f(y)$ for $x \neq y$ and $d(x,x) = 0$. 
\end{exercise}

\textbf{Probability Spaces}

For the next exercises we let $\Omega \subseteq \mathbb{R}^{n}$ for simplicity (but $\Omega$ can be any separable metric space more generally). Measurable means Borel measurable.

 \begin{exercise} Show that for any collection of sets $\mathfrak{B}$ the sigma-algebra that is generated by  $\mathfrak{B}$ exists. 
 \end{exercise}

   \begin{exercise} 
 Find a Borel measurable function that is not continuous 
 
 Hint: consider the indicator function on the set of rational numbers. 
 \end{exercise}

\begin{exercise}
Show that the law of a random variable $X$ defined in Lecture 1 is indeed a probability measure. 
\end{exercise}

\begin{exercise}
    Let $f:\Omega \rightarrow \mathbb{R}$ be a measurable function such that $f(\omega) \geq 0$ for all $\omega$. A measurable function $\phi : \Omega \rightarrow \mathbb{R}$ is called simple if it assumes only finite number of values. 

Show that there exists a sequence of non-negative increasing simple $\phi_{n}$ functions that converges to $f$ pointwise and satisfies $\phi _{n} \leq f$ for all $n$.

Hint: Consider $B_{n}^{i} = \{ \omega \in \Omega: (i-1)2^{-n} \leq f(\omega) < i2^{-n} \} $ and $\phi _{n} = \sum _{i=1}^{n2^n} 2^{-n}(i-1) 1_{B_{n}^{i} } $. 

\end{exercise}

\begin{exercise} 
Suppose that $f,g$ are measurable. 

(1) Show that $ \{ \omega \in \Omega: f(\omega) > g(\omega) \}$, $ \{ \omega \in \Omega: f(\omega) \geq g(\omega) \}$ and $\{ \omega \in \Omega: f(\omega) = g(\omega) \}$ are all measurable. 

(2) Show that $f + g$ is measurable. 

(3) Show that $fg$ is measurable.

Hint: use  $fg = 0.5 ( (f+g)^2 - f^2 - g^2)$.

(4) Show that $f \circ g$ is measurable, i.e.,  the composition of Borel measurable functions is Borel measurable. 

(5)  Suppose that $\{k_{n} \}$ is a sequence of measurable functions. Show that if $k_{n}$ converges to $k$ pointwise then $k$ is measurable. 
\end{exercise}

\begin{exercise}
    Suppose that $(X,F)$ and $(Y,G)$ are two measurable spaces. The product sigma-algebra denoted by $F \otimes G$ is given by 
    $$ F \otimes G = \sigma  ( \{ A \times B: A \in F, B \in G \} ) $$

where $\sigma$ is defined in Lecture 1. 
    Show that $\mathcal{B} (\mathbb{R}^{n+m}) =\mathcal{B} (\mathbb{R}^{n}) \otimes 
 \mathcal{B} (\mathbb{R}^{m})   $ for $n , m \geq 1$ 
where   $\mathcal{B} (\mathbb{R}^{k})$ is the Borel sigma-algebra on $\mathbb{R}^{k}$. 
\end{exercise}

\newpage

\section{Lecture 2: The Principle of Optimality in Dynamic Programming}

\subsection{Discounted Dynamic Programming}

In this course, we will primarily explore discounted dynamic programs (in the exercises we will see that the key ideas of this note also apply to positive dynamic programming models). We describe a discounted dynamic programming model as a tuple \((S, A, \Gamma, p, r, \beta, T)\), where (measurable means Borel measurable throughout the lectures):
\begin{itemize}
    \item \(T\) represents the number of periods, typically assumed to be infinite (\(T = \infty\)).
    \item \(S\) is a measurable separable complete metric space, defining the possible states of the system.
    \item \(\mathcal{B}(S)\) denotes the Borel \(\sigma\)-algebra on \(S\).
    \item \(A \subseteq \mathbb{R}^k\) is a measurable set of possible actions. 
    \item $\Gamma: S \rightarrow 2^{A}$ is a correspondence where \(\Gamma(s)\) is the  measurable set of feasible actions in state \(s \in S\).
    \item \(p: S \times A \times \mathcal{B}(S) \rightarrow [0, 1]\) is a transition probability function, where \(p(s, a, \cdot)\) is a probability measure on \(S\) for each \((s, a)\), and \(p(\cdot, \cdot, B)\) is measurable for each \(B \in \mathcal{B}(S)\).
    \item \(r: S \times A \rightarrow \mathbb{R}\) is a measurable single-period payoff function.
    \item \(0 < \beta \leq 1\) is the discount factor, with \(0 < \beta < 1\) when \(T = \infty\).
\end{itemize}

The process initiates in some state \(s(1) \in S\). At time \(t\), given the state \(s(t)\), the decision maker (DM) selects an action \(a(t) \in \Gamma(s(t))\) and receives a payoff \(r(s(t), a(t))\). The probability that the state in the next period, \(s(t+1)\), lies in \(B \in \mathcal{B}(S)\) is given by \(p(s(t), a(t), B)\).

\begin{remark}
   (i) For every  measurable function $f:S \rightarrow \mathbb{R}$, $z(s,a) = \int f(s')p(s,a,ds')$ is measurable (why? use simple functions first..). 

(ii) In many applications the next period's state $s'$ is given by a transition function $s' =w(s,a,\zeta)$ where $\zeta$ is a random variable on some set $E$ and has a law $ \phi $. In this case $p(s,a,B) = \phi (\zeta \in E: w(s,a,\zeta) \in B)  $ for  \(B \in \mathcal{B}(S)\) and for every measurable function $f:S \rightarrow \mathbb{R}$ we have 
$$  \int _{S} f(s')p(s,a,ds') =  \int _{E }f (w(s,a,\zeta)) \phi (d \zeta ) $$
  (why)? 
\end{remark}

\begin{remark} \label{remark:directly}
    In some settings, the payoff function also depends directly on the transition probability function and is given by $r(s,a) = \int R(s,a,s') p(s,a,ds')$. 
\end{remark}

\textbf{Examples.} This framework is widely applicable in fields such as economics, operations research, management science, and engineering. Below are two classical examples.

\begin{example}
\textbf{Consumption-savings problems}: Consider an agent deciding how much to save and consume each period, possessing an initial wealth \(s(1)\) and receiving a periodic income \(y(t)\). The agent determines consumption \(c(t)\) and savings \(a(t)\) for future consumption in a risk-free bond. The savings grow at a rate \(R > 0\). The next period's wealth is given by
\[
s(t+1) = R a(t) + y(t).
\]
We assume no borrowing, with \(\Gamma(s) = [0, s]\), \(S = [0, \infty)\), and \(A = [0, \infty)\). The utility from consumption is \(r(s, a) = u(s-a)\), where \(u\) is strictly increasing, and \(y(t)\) are I.I.D. random variables with law \(\phi\). The transition probability function is
\[
p(s, a, B) = \phi(y : Ra + y \in B).
\]
\end{example}

\begin{example}
\textbf{Inventory management}: A retailer's state \(s(t) \in S = [0,\infty) \) indicates the current inventory level. The retailer decides the total units \(a(t) \in [s(t), \infty)\) available to sale in period $t$ (here $a(t) - s(t)$ is amount of units the retail produce or purchase). If the retailer's state at time  $t $ is $s$, the retailer takes an action $a$, then the retailer's state in the next period $s'$ is given by
\begin{equation*} s' =\max \{a  - D,0\}
\end{equation*}
where $D$ is a random variable that represents the random demand for the retailer's  product.  We assume that $D$ has the same law $\phi$ in each period for simplicity (we can easily accommodate Markovian demand or a more general demand structure in the discounted dynamic programming framework). 

The retailer faces production costs $c$ per unit, and holding costs \(h\) per unit, and \(w\) is the revenue per unit. The single period  expected  payoff is
\[
r(s, a) = w \mathbb{E}_{D}[\min(a, D)] - c(a - s) - h \mathbb{E}_{D}[(a - D)^+],
\]
where \(x^+ = \max(x, 0)\). The transition probability function is
\[
p(s, a, B) = \phi(D : (a-D)^{+} \in B).
\]
\end{example}

\subsection*{Remark: Special Cases of Interest}
We highlight two significant special cases within the framework of dynamic programming that are particularly relevant real world applications. 

\begin{enumerate}
    \item \textbf{Multi-Armed Bandit Problem:}
    In the multi-armed bandit problem, the state space consists of a single element, thereby simplifying the state dynamics. However, the rewards associated with each action (or``arm") are initially unknown. The primary objective is to identify the most rewarding action through repeated experimentation, balancing the trade-off between exploration of untested actions and exploitation of actions known to yield high rewards. 
    
    \item \textbf{Contextual Bandit Framework:}
    Popular in industrial applications, such as personalization,  recommendation systems, online advertising, clinical trials, and dynamic pricing, the contextual bandit framework extends the multi-armed bandit problem by incorporating a vector of context features that describe each decision instance. Here, the state \(s\) represents a feature vector (e.g., demographics, age of consumers) that is assumed to be independently and identically distributed over time. Hence,  the state dynamics are simple but as in the standard multi-armed bandit problem the rewards are unknown. The goal is to learn the best action for each possible context. 
\end{enumerate}

These special cases are instrumental in guiding the design of algorithms that efficiently balance learning and decision-making in dynamic environments, with widespread applications. 

\begin{example}
    \textbf{Bayesian multi-armed-bandit as a dynamic programming model}:
 Consider a multi-armed bandit problem as described above from a Bayesian point of view. We have a set of \( k \) arms, each, for simplicity, associated with a Bernoulli distribution. The payoff from each arm \( i \) is either 1 or 0, with the probability of receiving a 1 being \( p_i \), which is initially unknown. The belief about each arm's success probability \( p_i \) is modeled as a Beta distribution with parameters \(\alpha_i\) and \(\beta_i\) (the Beta distribution has a probability density function $cx^{\alpha-1}(1-x)^{\beta-1}$ on $[0,1]$ where $c$ is a normalization factor). 

The state at any decision epoch \( t \), denoted by \( s(t) \), is the set of parameters $$ s(t) = (\alpha_1(t), \beta_1(t), \ldots, \alpha_k(t), \beta_k(t)) $$ for the Beta distributions of each arm. These parameters represent the posterior beliefs about the probabilities \( p_i \) based on prior observations.
    
    The action \( a(t) \) at time \( t \) is choosing which arm to pull, where \( a(t) \in \{1, 2, \ldots, k\} \).

  The expected payoff \( r(s, a) \) after taking action \( i \) in state \( s \) is $\alpha_{i} / (\alpha_{i} + \beta_{i})$. 

   The probability function updates the belief state based on the observed outcome from pulling an arm. If arm \( a \) is pulled and (random) outcome \( x \) (0 or 1) is observed, the state updates from \( (\alpha_a, \beta_a) \) to \( (\alpha_a + x, \beta_a + 1 - x) \) for arm \( a \) according to Bayesian updating, while other arms' parameters remain unchanged.

\end{example}

\textbf{Strategies.} Let $H^{t}:= S \times (A \times S)^{t-1} $ be the set of finite histories up to time $t$. A strategy $\sigma$ is a sequence $(\sigma _{1} ,\sigma _{2} , \ldots )$ of Borel measurable functions $\sigma _{t} : H^{t} \rightarrow A$ such that $\sigma _{t}(s(1) ,a(1) ,\ldots  ,s(t)) \in \Gamma (s(t))$ for all $t \in \mathbb{N}$ and all $(s(1) ,a(1) ,\ldots  ,s(t))\in H^{t}$. For each initial state $s\left (1\right )$, a strategy $\sigma $ and a transition probability function $p$ induce a probability measure 
 over the measurable space of all infinite histories $(H^{\infty },\mathcal{B}(H^{\infty }))$.\protect\footnote{
The probability measure on the space of all infinite histories $H^{\infty }$ is uniquely defined by Kolmogorov extension theorem.   Let $\mu _{{1}, \ldots , t }$ be the probability measure on $H^{t}$ that is generated by the strategy and the transition probability function. Kolmogorov extension theorem asserts that there exists a unique probability $\mu$ on $H^{\infty}$ with marginals $\mu _{{1}, \ldots , t }$ for any finite $t$. That is, $\mu (E \times \prod _{k=t+1}^{\infty} H^{k} ) = \mu _{{1}, \ldots , t }(E)$ for every $E$ in $\mathcal{B}(H^{t})$ and all $t \geq 1$. In simpler terms, the Kolmogorov extension theorem guarantees that if we have a family of consistent finite-dimensional distributions, we can uniquely extend these to a probability measure on the infinite-dimensional product space. The proof of this theorem is beyond the scope of this course (see  the textbook \cite{dudley2018real} for a proof). 
} We denote the expectation with respect to that probability measure by $\mathbb{E}_{\sigma }$, and the associated stochastic process by $\{s(t) ,a(t)\}_{t =1}^{\infty }$.

\textbf{Objective function.} The decision maker's goal is to find a strategy that maximizes her expected discounted payoff. When the decision maker follows a strategy $\sigma$ and the initial state is $s \in S$ the expected discounted payoff is given by
\begin{equation*}V_{\sigma}(s) =\mathbb{E}_{\sigma }\sum \limits_{t =1}^{\infty }\beta^{t -1}r(s(t),a(t)).
\end{equation*}Define \begin{equation*}V(s) =\sup _{\sigma }V_{\sigma }(s).
\end{equation*}
We call $V :S \rightarrow \mathbb{R}$ the value function. A strategy is said to be optimal if $V_{\sigma} (s) = V(s)$ for all $s$.  A strategy $\sigma$ is said to be $\epsilon$-optimal if for every $\epsilon >0$, every strategy $\sigma'$, and any $s\in S$ we have $V_{\sigma'}(s) \leq V_{\sigma}(s) + \epsilon$. A strategy is said to be stationary if $\sigma (h(t)) = \lambda(s(t))$ where $h(t) = (s(1),a(1),\ldots,s(t) )$ for any history $h(t) \in H^{t}$ and some measurable function $\lambda:S \rightarrow A$ which is called a stationary policy. 

We let $B_{m}(S)$ to be the space of all bounded measurable functions from $S$ to $\mathbb{R}$

\textbf{The Bellman operator.} Define the operator (called the Bellman operator or Bellman equation) $T$ from measurable functions from $S$ to $\mathbb{R}$ to functions from  $S$ to $\mathbb{R}$ by
\begin{equation*}Tf(s) =\sup _{a \in \Gamma (s)}Q(s ,a ,f), 
\end{equation*}
where 
\begin{equation}Q(s ,a ,f) =r(s ,a) +\beta \int _{S}f(s^{ \prime })p(s ,a ,ds^{ \prime }). \label{eq:h}
\end{equation}

$Tf$ is the best an agent in state $s$ can achieve by choosing an action $a$ at the current period and the expected value of $f$ at the next period.  We denote by $U(X,Y)$ the set of Borel measurable functions from a metric space $X$ to a metric space $Y$ 
and  the set of optimal measurable solutions for the optimization problem above $\Lambda _{f}$ by
$$\Lambda _{f} = \{ \lambda \in U(S,A): \sup _{a \in \Gamma(s)} \ Q(s,a,f) = Q(s,\lambda(s),f), \  \forall s \in S \} .$$

 For a stationary policy $\lambda \in U(S,A)$ and $f \in U(S,\mathbb{R})$ we also define $T_{\lambda}$ by
\begin{equation}
    T_{\lambda}f(s) =  r(s,\lambda(s)) + \beta \int _{S} f(s') p(s,\lambda(s),ds') = Q(s,\lambda(s),f).
\end{equation}


\subsection{The Dynamic Programming  Principle}
 
We are now ready to state the dynamic programming principle (recall that we endow $B_{m}(S)$ with the supremum metric $d_{\infty}$ and it is a complete metric space).

\begin{theorem} \label{Thm:mainDP} Let $D \subseteq B_{m}(S)$ be a complete metric space with the metric $d_{\infty}$ that contains the constant functions. 

Suppose that $r$ is bounded. Assume that $f \in D$ implies $\Lambda_{f} \neq \emptyset$, and $Tf \in D$.  

Then the following holds: 

1. $V$ is the unique function in $D$ that satisfies the Bellman equation, i.e., $TV = V$.

2. There is a stationary policy $\lambda \in \Lambda_{V}$ that is $\epsilon$-optimal. 

3. If $\lambda \in \Lambda_{V}$ achieves the supremum then $\lambda$ is an optimal stationary policy. 

\end{theorem}

We first prove a few simple properties of the operators $T$ and $T_{\lambda}$. We will write $\Vert f - g \Vert = \sup _{s \in S} |f(s) - g(s) | = d_{\infty} (f,g)$.

\begin{lemma} \label{lemma:T-lambda-closed}
     For each measurable $\lambda \in U(S,A)$, $f \in U(S,\mathbb{R})$ implies $T_{\lambda}f \in U(S,\mathbb{R})$. In addition, for each strategy $\sigma$, we have $V_{\sigma} \in U(S,\mathbb{R})$.   
\end{lemma}
\begin{proof}
Let $f \in U(S,\mathbb{R})$ and $\lambda \in U(S,A)$. Then 
 $\int _{S} f(y) p(s,a,ds')$ is Borel measurable whenever $f$ is Borel measurable. Hence,  $Q(s,a,f) $ is Borel measurable as the sum of Borel measurable functions. Thus, $T_{\lambda}f = Q(s,\lambda(s),f)  $ is Borel measurable as the composition of  Borel measurable functions. 

Using a similar argument, $V_{\sigma}$ is Borel measurable for each strategy  $\sigma$ (see exercises). 
\end{proof}

\begin{lemma} \label{lemm:Tf-prop}
The operator $T$ satisfies the following two properties.\footnote{For two functions $f,g$ from $S$ to $\mathbb{R}$, $f \geq g$ means $f(s) \geq g(s)$ for each $s \in S$. }

(a) Monotone: $Tf \geq Tg$ whenever $f \geq g$. 

(b) Discounting: $T(f+c) = Tf + \beta c$ for any constant function $c$. 

In addition, $T_{\lambda}$ satisfies $(a)$ and $(b)$ for any stationary policy $\lambda \in U(S,A)$. 
\end{lemma}

\begin{proof}
    (a) We have $Q(s,a,f) \geq Q(s,a,g)$ whenever $f \geq g$. Hence, $Tf \geq Tg$.

    (b) For each $s \in S$ we have
    $$T(f+c)(s) =  \sup _{a \in \Gamma(s)} r(s ,a) +\beta \int _{S}f(s^{ \prime })p(s ,a ,ds^{ \prime }) + \beta c = Tf(s) + \beta c.$$

   The proof for the operator $T_{\lambda}$ is similar. 
\end{proof}

A key property of the operators $T$ and $T_{\lambda}$ is the contraction property. The operator $T$ on $B_{m}(S)$ is $L$-contraction if $\Vert Tf - Tg \Vert \leq L \Vert f - g \Vert$ for all $f,g \in B_{m}(S)$. 

\begin{lemma} \label{lemm:Tf-contraction}
    The operator $T$ is $\beta$-contraction on $B_{m}(S)$, and the operator $T_{\lambda}$ is $\beta$-contraction on $B_{m}(S)$ for any stationary policy $\lambda \in U(S,A)$. 
\end{lemma}

\begin{proof}
Let $f,g \in B_{m}(S) = U(S,\mathbb{R}) \cap B(S)$ and  $\lambda \in U(S,A)$. From  Lemma \ref{lemma:T-lambda-closed}, $T_{\lambda}f \in U(S,\mathbb{R})$ and it is immediate that $T_{\lambda}f \in B(S)$ so $T_{\lambda}f$ is indeed in $B_{m}(S)$. 

For any $s \in S$ we have 
$$ | T_{\lambda}f(s) - T_{\lambda} g(s) | = |\beta \int_{S}(f(s') - g(s'))p(s,\lambda (s),ds') | \leq \beta \int _{S} | f(s') - g(s')|p(s,\lambda(s),ds') \leq \beta \Vert f - g \Vert. $$
Taking the supremum over $s \in S$ yields $\Vert T_{\lambda}f - T_{\lambda} g \Vert \leq \beta \Vert f - g \Vert$ so $T_{\lambda}$ is  $\beta$-contraction on $B_{m}(S)$.

Hence, it follows that for all $s \in S$ and $a \in \Gamma(s)$ we have 
$$Q(s,a,f)  \leq Q(s,a,g) + \beta \Vert f - g \Vert  \leq Tg (s)+ \beta \Vert f - g \Vert.  $$
Taking the supremum over $a \in \Gamma(s)$ yields 
$$ Tf(s) \leq Tg(s) + \beta \Vert f - g \Vert.$$
By a symmetrical argument, we can deduce $Tg(s) \leq Tf(s) + \beta \Vert f - g \Vert$ for all $s \in S$ so $T$ is $\beta$-contraction on $B_{m}(S)$. 
\end{proof}

\begin{proof} [Proof of Theorem \ref{Thm:mainDP}]
From the Banach fixed-point theorem (see Proposition \ref{prop:Banach}) and Lemma \ref{lemm:Tf-contraction} the operator $T:D \rightarrow D$ has a unique fixed-point, say, $x$. We will show that $V=x$.  We first show that $V \leq x$. 

 Because $r$ is bounded by, say, $M$ we have $ V \leq M/(1-\beta)$. Hence, there exists a $y \in D $ such that $V \leq y$.

 For a strategy $\sigma$ let $\sigma | (s,a)$ be the strategy that is induced by $\sigma$ given that the first period's state-action pair was $(s,a)$, i.e., $\sigma | (s,a) _{t} (h_{t}) = \sigma _{t+1} (s,a,h_{t}) $ for all $t$ and any history $h_{t} \in H^{t}$.  
For any strategy $\sigma$, a function $y \in D$ with $V \leq y$, and $s \in S$ we have 
\begin{align*}
    V_{\sigma}(s) & = \mathbb{E}_{\sigma } \left ( r(s,a(1)) + \sum \limits_{t =2}^{\infty }\beta^{t -1}r(s(t),a(t)) \right )   \\
    & = r(s,\sigma_{1}(s)) + \beta \int _{S} V_{\sigma | (s,\sigma_{1}(s))} (s') p(s, \sigma_{1}(s),ds') \\
    & \leq  r(s,\sigma_{1}(s)) + \beta \int _{S} y (s') p(s, \sigma_{1}(s),ds') \\
    & \leq \sup _{a \in \Gamma(s)} \ r(s,a) + \beta \int _{S} y (s') p(s, a,ds') \\
    & = Ty(s). 
\end{align*}
   Thus, taking the supremum over all strategies $\sigma$ implies that $V \leq Ty$.  Repeating the same argument, we deduce that $V \leq T^{n}y$ for all $n \geq 1$. From the Banach fixed-point theorem $T^{n}y$ converges to the unique fixed-point of $T$ on $D$. We conclude that $V \leq x$.  

We now show that $V \geq x$. 
   Let $\epsilon >0$. Because $x \in D$, the Theorem's assumption implies that $\Lambda_{x} \neq \emptyset $. Hence, because $Tx=x$, there is a function $\lambda \in U(S,A)$ such that $T_{\lambda}x \geq x - \epsilon(1-\beta) $. 
   
   We will now show that $T_{\lambda}^{n}x \geq x - \epsilon$ for all $n$. For $n=1$ the result holds from the construction of $\lambda$. Assume it holds for $n>1$. Then 
   $$ T_{\lambda}^{n+1}x = T_{\lambda} (T_{\lambda}^{n}x) \geq T_{\lambda} (x - \epsilon) = T_{\lambda}x - \beta \epsilon \geq  x - \epsilon$$
   where we used Lemma \ref{lemm:Tf-prop} to derive the first inequality. 
Thus, $T_{\lambda}^{n}x \geq x - \epsilon$ for all $n$. From Lemma \ref{lemm:Tf-contraction}  and the Banach fixed-point theorem  
 $T_{\lambda}^{n}$ converges to the unique fixed-point of $T_{\lambda}$ in $B_{m}(S)$ which we denote by $W$ so $W \geq x - \epsilon$. 
 
 Hence, $W:S \rightarrow \mathbb{R}$ is Borel measurable and we have 
$$W (s) = r(s,\lambda(s)) + \beta \int_{S}W(s')p(s,\lambda(s),ds') = \mathbb{E}_{\sigma }\sum \limits_{t =1}^{\infty }\beta^{t -1}r(s(t),a(t)) $$
with $s=s(1)$ and $\sigma$ is the stationary plan that plays according to $\lambda$, i.e., $\sigma_{n}(h_{n}) = \lambda(s(n))$ for all $h_{n}(s(1),a(1),\ldots,s(n))$. Hence, $V \geq W \geq x - \epsilon$. Thus, $V \geq x$.  We conclude that $V=x$ so $V$ is the unique fixed point of $T$ on $D$ and $\lambda \in \Lambda_{V}$ is an $\epsilon$-optimal strategy.  

In addition, if $\lambda \in \Lambda_{V}$ achieves the supremum then 
$T_{\lambda} V = V$ so $\lambda$ is the optimal stationary policy.
\end{proof}

Theorem \ref{Thm:mainDP} implies the following Corollary which is important for proving properties of the value function such as concavity, monotonicity, etc. 

\begin{corollary} \label{Cor:prop}
    Assume that the conditions of Theorem \ref{Thm:mainDP} hold. Suppose that $D'$ is a closed set in $B_{m}(S)$ such that $D' \subseteq D$ and $f \in D'$ implies that $Tf \in D'$. Then $V \in D'$.   
\end{corollary}

\begin{proof}
   Suppose that $f \in D'$. Then $T^{n}f \in D'$ for every $n$. From Theorem \ref{Thm:mainDP}, $T^{n}f$ converges to the value function $V$ (recall that $V$ is the fixed point of the $\beta$-contraction operator $T$). Because $D'$ is closed, we have $V \in D'$. 
\end{proof}

Another immediate Corollary from Theorem  \ref{Thm:mainDP} provides an algorithm to find the value function which is called value function iteration. 

\begin{corollary}
    Assume that the conditions of Theorem \ref{Thm:mainDP} hold. Then for every $f \in D$, the sequence of functions $T^{n}f$ converges uniformly to the value function $V$. 
\end{corollary}

\begin{proof}
    This follows immediately from the fact that $V$ is the unique fixed point of the $\beta$-contraction operator $T$. 
\end{proof}

In particular, we can take the constant function $f=0$ and the sequence $T^{n}f$ converges to the value function. 

The conditions of Theorem \ref{Thm:mainDP} holds immediately for the important case that $S$ is countable. We summarize this in the following Corollary. 

\begin{corollary}
   (i) Suppose that the sets $S$ and $A$ are finite.  We can choose $D=B_{m}(S)=B(S)$ and the conditions of Theorem \ref{Thm:mainDP} hold. Hence, 1 2 3 of Theorem \ref{Thm:mainDP} hold and the supremum is achived. 

   (ii) Suppose that $S$ is countable, $A$ is a set in $\mathbb{R}^{n}$,\footnote{A prominent choice is to have $A$ be the set all probability measures  on some finite action set $A'$.} $r$ is bounded. We can choose $D=B_{m}(S)=B(S)$ and the conditions of Theorem \ref{Thm:mainDP} hold. Hence, 1 2 3 of Theorem \ref{Thm:mainDP} hold. 
\end{corollary}

\subsection{Upper Semi-Continuous Dynamic Programming}

For the general case, to verify the conditions of Theorem \ref{Thm:mainDP}, as a starting point it is natural to assume that the dynamic programming problem's primitives  are Borel measurable and  choose $D$ to be the sets of Borel measurable functions. But unfortunately this choice does not work. 
Let $S=A=[0,1]$. Then there is a Borel measurable set $B \subseteq S \times A$ such that its projection $E$ to $S$ is not Borel measurable.  Consider $f$ to be the zero function and $r = 1_{B}(s,a)$. Then $f$ and $r$ are Borel measurable. But $Tf (s) = \sup_{a \in A} 1_{B}(s,a) = 1_{E}(s)$ is not Borel measurable. Hence, we cannot apply Theorem \ref{Thm:mainDP} for this case. 

Thus, we will need other assumptions on the problem's primitives. We will focus on upper semi-continuous dynamic programming. 

\begin{definition}
Let $X$ be a metric space. 

    (i) A function $f:X \rightarrow \mathbb{R}$ is upper semi-continuous (u.s.c) if for all $\alpha \in \mathbb{R}$ the set $\{x: f(x) < \alpha \}$ is open. Equivalently, $f$ is u.s.c if and only if $\limsup_{x \rightarrow x_{0} } f(x) \leq f(x_{0}) $ for all $x_{0}$. 

    (ii) A transition kernel $p$ is called upper semi-continuous  if $k(s,a) = \int _{S} f(s')p(s,a,ds')$ is u.s.c and bounded whenever $f$ is u.s.c and bounded.
\end{definition}

\begin{remark}
     A more familiar notion in the literature is that $p$ is continuous (sometimes called Feller or Feller continuous) in the sense that $k(s,a) = \int _{S} f(s')p(s,a,ds')$ is continuous and bounded whenever $f$ is continuous. It is not hard to see that if $p$ is continuous then $k(s,a) = \int _{S} f(s')p(s,a,ds')$ is u.s.c and bounded whenever $f$ is. Hence, if $p$ is continuous it is also u.s.c. 
\end{remark}

As an example, $f (x) = -2 1_{ \{x < 0 \} } + 2 1_{ \{x \geq 0 \} }$ is u.s.c at $0$ but not continuous. 

We will assume in the rest of this section that $D$ is the set of all u.s.c and bounded functions. 
Clearly, from the definition of u.s.c functions, every u.s.c function is Borel measurable so $D \subseteq B_{m}(S)$. To use Theorem \ref{Thm:mainDP} we first need the following Lemma. 

\begin{lemma} \label{Lem:complete}
    The metric space $D$ is a complete metric space (endowed with $d_{\infty}$). 
\end{lemma}

\begin{proof}
    It is enough to show that $D$ is closed under uniform convergence (see exercises). 

    Consider a sequence $f_{k}$ of bounded u.s.c functions that converges uniformly to $f$. We need to show that $f$ is u.s.c.   
    For $\epsilon >0$, consider $K_{\epsilon} $ such that $|f_{k}(s) - f(s)| < \epsilon$ for all $k \geq K_{\epsilon}$ and all $s \in S$. 

    Let $s_{n} \rightarrow s_{0}$. We have $f(s_{n}) < f_{K_{\epsilon} }(s_{n}) + \epsilon$ for all $n$ and $ f_{K_{\epsilon} } (s_{0} ) < f(s_{0}) + \epsilon$. Hence, 
    $$ \limsup f(s_{n}) \leq  \limsup f_{K_{\epsilon} } (s_{n}) + \epsilon \leq  f_{K_{\epsilon} } (s_{0}) + \epsilon < f(s_{0}) + 2\epsilon. $$
Because $\epsilon$ is arbitrary, this proves that $f$ is u.s.c 
\end{proof}

The next Lemma shows that every upper semi-continuous function has a maximum on a compact set. This is useful to show that the policy function attains its maximum in Theorem \ref{Thm:mainDP} for upper semi-continuous dynamic programming. 

\begin{lemma} \label{Lemma:USC_MAX}
     Let $K$ be a compact subset of some metric space. Let $f:K \rightarrow \mathbb{R}$ be u.s.c. 
     Then $f$ has a maximum on $K$. 
\end{lemma}
\begin{proof}
     Let $m = \sup _{x \in K} f(x)$. Consider a sequence $x_{n}$ such that $f(x_{n}) \geq m - \epsilon_{n}$ for some $\epsilon_{n} $ that converges to $0$ when $n \rightarrow \infty$. Because $K$ is a compact, every sequence has a convergent subsequence. Thus, there is a subsequence of $x_{n}$, say $x_{n_{j}}$ that converges to some $x^{*}$. Hence, 
$$ m \geq f(x^{*}) \geq \limsup f(x_{n_{j}}) \geq m $$
which proves that $m = f(x^{*})$, i.e., $f$ attains its maximum on $K$. 
\end{proof}

We will also need a continuity assumption on the set of feasible actions. 

A neighborhood of a set $X$ is any set $U$ such that there exists an open set $O$ satisfying $X \subseteq O \subseteq U$. An open set $O$ that satisfies $X \subseteq O$ is called an open neighborhood of $X$. 

\begin{definition}
    The correspondence $\Gamma:S \rightarrow 2^{A}$ is upper semi-continuous (u.s.c) if $\{s \in S: \Gamma(s) \subseteq U \}$ is an open set whenever $U$ is an open set in $A$. 
\end{definition}

The following simple example demonstrates the definition. 

\begin{example}
     Suppose that $A=S=[0,1]$. Then $G(s) = [0,s]$ is u.s.c. But $G(s) = [0,1]$ for $s<1$ and $G(1) = \{0\}$ is not u.s.c.  
On the other hand, 
     $G(1) = [0,1]$ and $G(s) = \{0\}$ for $s<1$ is u.s.c.  
\end{example}

Recall that  $\operatorname{Gr} \Gamma = \{ (s,a) \in S \times A : a \in \Gamma(s) \}$ is called the graph of  a correspondence $\Gamma:S \rightarrow 2^{A}$. 
The key lemma to show that $f \in D$ implies $Tf \in D$ is the following:

\begin{lemma} \label{Lemma:max}
    Let $\Gamma:S \rightarrow 2^{A}$ be a u.s.c correspondence such that $\Gamma(s)$ is compact and non-empty for each $s \in S$.  Let $h: \operatorname{Gr} \Gamma \rightarrow \mathbb {R}$ be a u.s.c function. Define $ m(s) = \max _{a \in \Gamma(s)} h(s,a) $. Then $m$ is u.s.c 
\end{lemma}

\begin{proof}
    Note that $h$ has a maximum and $m$ is well-defined (see Lemma \ref{Lemma:USC_MAX}). Let $c \in \mathbb{R}$ and we want to show that $C = \{s \in S: m(s) < c\}$ is an open set. 
 Let $s_{0} \in C$ (if $C$ is the empty set it is open) and let $W = \{ (s,a) \in \operatorname{Gr} \Gamma: h(s,a) < c \}$. Then $W$ is open since $h$ is u.s.c, and hence, for each $a \in \Gamma (s_{0})$, by noting that $(s_{0},a) \in W$,  there are open neighborhoods $U_{a}$ of $s_{0}$ and $V_{a}$ of $a$, such that $(U_{a} \times V_{a} ) \cap \operatorname{Gr} \Gamma \subseteq W$. 
 
 The family $\{V_{a}:a \in \Gamma(s_{0}) \}$ is an open cover of $\Gamma(s_{0})$ so by compactness, there is a finite subcover $\{V_{a_{1}}, \ldots , V_{a_{n} } \}$  of $\Gamma(s_{0})$. Let $U = \cap _{i=1}^{n} U_{a_{i}} $ and $V = \cup _{i=1}^{n} V_{a_{i}} $, so $\Gamma(s_{0}) \subseteq V$. Note that  $(U \times V ) \cap \operatorname{Gr} \Gamma \subseteq W$. 

    Using the fact that $\Gamma$ is u.s.c implies that $\{ s \in S : \Gamma(s) \subseteq V \}$ is open so $E = U \cap \{ s \in S : \Gamma(s) \subseteq V \}$ is an open neighborhood of $s_{0}$. 
    
    For each $s \in E$, if $a \in \Gamma(s)$, then $h(s,a) < c$ as $(U \times V ) \cap \operatorname{Gr} \Gamma \subseteq W$ so $m(s) < c$. We conclude that $E \subseteq \{s \in S: m(s) < c \}= C$ so $C$ is an open set. 
\end{proof}

We also need the following Measurable Selection theorem to ensure that $\Lambda_{f}$ is non-empty when $f \in D$. The proof is beyond the scope of this course. 

\begin{proposition} (Measurable Selection Theorem) \label{prop:selector}
     Let $\Gamma:S \rightarrow 2^{A}$ be a u.s.c correspondence such that $\Gamma(s)$ is compact and non-empty for each $s \in S$.  Let $h: \operatorname{Gr} \Gamma \rightarrow \mathbb {R}$ be a u.s.c function. Then there is a Borel measurable function $\phi: S \rightarrow A$ such that $\phi(s) \in \Gamma(s)$ and 
    $$ h(s, \phi(s)) = \max _{a \in \Gamma(s) } h(s,a) $$
    for each $s \in S$. 
\end{proposition}

With this result, we can prove the main result of this section. 

\begin{theorem}
    Suppose that $r$ is bounded and u.s.c, $p$ is u.s.c, and $\Gamma$ is u.s.c such that $\Gamma(s)$ is non-empty and compact for each $s \in S$. 

    Then, $V$ is the unique bounded and u.s.c function that satisfies the Bellman equation and there exists a Borel measurable stationary policy function that is optimal. 
\end{theorem}

\begin{proof}
Let $D$ be the set of all bounded upper semi-continuous functions.  From Lemma \ref{Lem:complete}, $D$ is a complete metric space. 

Let $f \in D$. Then $ \int f(s') p(s,a,ds')$ is in $D$ because $p$ is u.s.c. Hence $h(s,a) = r(s,a) + \beta  \int f(s') p(s,a,ds')$ is in $D$ as the sum two u.s.c bounded functions. Now from Lemma \ref{Lemma:max}, $Tf \in D$ and from Proposition \ref{prop:selector}, $\Lambda_{f} \neq \emptyset$. Thus, the conditions of Theorem \ref{Thm:mainDP} hold and the result follows. 
\end{proof}

\textbf{Bibliography note}: The principle of optimality in  discounted dynamic programming  was studied in \cite{blackwell1965discounted}. The proof of Theorem \ref{Thm:mainDP} is adopted from \cite{light2024principle}. Upper semi-continuous dynamic programming is studied in \cite{maitra1968discounted}. The proof of Lemma \ref{Lemma:max} and the example before the lemma  is adopted from Chapter 17 in \cite{aliprantis2006infinite}. A proof of Proposition \ref{prop:selector} can be found in   \cite{savage1965gamble}.

\subsection{Exercise 2}

\begin{exercise}
    What is the Bellman Equation for the Bayesian multi-armed bandit example we introduced in Lecture 2? 

      Does the value function satisfy the Bellman Equation given the results in Lecture 2?
\end{exercise}

\begin{exercise}
    Consider the problem of deciding when to sell your house. Assume that the decision process can be modeled as a discounted dynamic programming model where at each period, you receive an offer to buy the house. The process is defined on a finite set \( S \), with a transition probability matrix \( P_{ij} \), where \( P_{ij} \) represents the probability of moving from state \( i \) to state \( j \). Each state in \( S \) could represent different market conditions or other factors influencing the offer prices.

At each period, you can either accept the offer and receive a payoff of \( R(i) \), or reject the offer, receive a payoff of \( 0 \), and wait for the next period's offer. There is a discount factor \( \beta \) (where \( 0 < \beta < 1 \)) applied to future payoffs, reflecting the time preference of money.

\begin{enumerate}
  \item Formulate this as a discounted dynamic programming model:
  Define the state space, action space, transition probabilities, reward function, and discount factor. 

  \item Write down the Bellman equation for this problem. Explain why the Bellman Equation holds (use the results in Lecture 2). 

  \item Assume that $S=\{ 1, \ldots, 10 \}$, $R(i) = i$. For each $i=2,\ldots,9$ suppose that $P_{ij} = 1/3$ for $j=i-1,i,i+1$ and $P_{ii}=1/3$ for $i=1,10$ and  $P_{10,9} = 2/3$, $P_{1,2}=2/3$. 

  Write a code in Python to find the value function using  the value function iteration algorithm. Plot the value function and the optimal policy function as a function of the state for $\beta = 0.9$ and $\beta = 0.99$. Explain the differences in the results. 

  \item Assume that $S=\{ 1, \ldots, 10 \}$, $R(i) = i$ as before. Unfortunately, you don't know $P_{ij}$. However, you have following the market for quite some time before deciding to sell and saw a previous string of states $(i(1),\ldots,i(M))$ with $i(j) \in S$ for $j=1,\ldots,M$ and a large $M$. 

  How would you solve the problem? 
  (this is quite open but maybe try an estimate then optimize approach).  
\end{enumerate}

\end{exercise}

\begin{exercise}
    Consider a finite discounted dynamic programming model ($S$ and $A$ are finite). Prove that the value function $(V(s_{1}),\ldots))$ solves the following linear programming problem (with variables $x_{1},\ldots)$

(i)    $$ \min \sum _{i \in S} x_{i} \text{  s.t  } x_{i} \geq r(s_{i},a) + \beta \sum_{s_{j} \in S}p(s_{i},a,s_{j}) x_{j}  \text{ for all } a \in \Gamma(s_{i}), \text { } s_{i} \in S. $$

Hint: show and use the fact that if $f \geq Tf$ then $f \geq V$. 

(ii) Thus, we can find the value function by solving a linear program. How many variables and constraints this linear programming has?

(iii) Unfortunately, the number of variables can be very large which makes the linear program approach not practical. One way to mitigate the problem is to consider an approximation to the value function $\hat{V} (s,w) = \sum _{i=1}^{n} w_{k} \phi _{k} (s)$ where $\phi_{k}$ are known and $w_{k}$ is a vector of parameters. This is a liner approximation of the value function with the idea that $n$ is much smaller than $|S|$. 

Show how we can slightly modify part (i) to provide a linear program to determine the weights $w_{1},\ldots,w_{n}$ with $n$ variables and the same number of constraints as part (i).

Note that while the number of constraints may be very large, there are methods (e.g., cutting plane methods) that may solve a  linear program with many constraints.

\end{exercise}

\begin{exercise}
    Consider the conditions of Theorem 1 in Lecture 2  but assume that there is no discounting ($\beta = 1$) and a finite time horizon $T$. Use the proof of Theorem 1 in Lecture 2 to show that the Bellman equation holds in the sense that 
    $$V_{t} (s) = \sup _{a \in \Gamma(s)} r(s,a) + \int V_{t+1}(s') p(s,a,ds') $$
    for all $t=1,\ldots,T$ where $V_{T+1} = 0$ and $V_{t}$ is value function starting from period $t$, i.e., the optimal reward the agent can achieve from period $t$ onwards (can you define $V_{t}$ formally)?

(i) Do you need to use the Banach fixed-point theorem?

(ii) Do you need to assume that $r$ is bounded?

\end{exercise}

\begin{remark}
    When payoffs are non-negative, i.e., $r (s,a) \geq 0$ for all $(s,a)$, then the sequence of functions $V_{t}$ is monotone as payoffs are always non-negative. In this case we can use a monotone convergence argument (we won't provide the exact details in the class) to show that the Bellman equation holds
$$V (s) = \sup _{a \in \Gamma(s)} r(s,a) + \int V(s') p(s,a,ds') $$ 
where $V$ is the value function with no discounting (it can be the case that $V=\infty$). 
\end{remark}

\begin{exercise}
    In class, we discussed the value function iteration algorithm. Now assume that you can only run approximate value function iteration (because the state space is infinite or the state space is very large, etc). That is, we have a sequence $f_{k}$ such that $d (f_{k+1},Tf_{k}) \leq \delta$ of functions $f_{k} \in B(S)$ for some $\delta >0$ where as usual $d = d_{\infty}$ is the sup metric and $T$ is the Bellman operator defined in Lecture  2. Note that value function iteration corresponds to  $\delta = 0$.  

     Show that 
    $$ \limsup _{k} d (f_{k},V) \leq \frac{\delta}{1-\beta}.$$
\end{exercise}

\begin{exercise}
    You are given a set of positive integers $X = \{x_{1}, x_{2}, \ldots, x_{n}\}$ and a target sum $Y$. Your task is to determine the minimum number of integers from $X$ that can be summed to exactly match $Y$. If no combination of these integers can sum to $Y$, return `-1`.

    Write a dynamic programming algorithm in Python to solve this problem (assume you have an unlimited number of each integer in $X$ available for use). 

    Hint: Let $V(y)$ as the minimum number of numbers needed to make the sum $y$. Initialize $V(0) = 0$. For each $y=1,\ldots,Y$ find $V(y) = $.... recursively. 
\end{exercise}

   \begin{exercise}
    Given two strings `word1` and `word2`, write a function in Python that computes the minimum number of single-character edits (insertions, deletions, or substitutions) required to change `word1` into `word2` using a dynamic programming algorithm.

    Example 1:
    Input: word1 = "disel"
word2 = "daniel"
    Output: 3  (Explanation: insert 'a' after 'd', insert 'n' after 'a', delete 's' after 'i'). 

   Example 2: Input: word1 = "dan"  word2 = "fang"  Output: 2  (Explanation: replace $d$ and add $g$ after $n$).

    Hint: Let $V(i,j)$ be the minimum number of operations required to convert the first $i$ characters of word1 to the first $j$ characters of word2.
For example, $V(0,j) = j$
as we need to insert all those $j$ characters to change word1 into word2. Similarly, $V(i,0)=i$. Now find $V(i,j)$ recursively. 
\end{exercise}

\begin{exercise}
 (i)   Show that $V_{\sigma}$ is Borel measurable under the assumptions of Theorem 1 in Lectures Notes 2. 

 (ii)    Justify the following equality from the proof of Theorem 1 in Lectures Notes 2. 
    \begin{align*}
    V_{\sigma}(s) & = \mathbb{E}_{\sigma } \left ( r(s,a(1)) + \sum \limits_{t =2}^{\infty }\beta^{t -1}r(s(t),a(t)) \right )   \\
    & = r(s,\sigma_{1}(s)) + \beta \int _{S} V_{\sigma | (s,\sigma_{1}(s))} (s') p(s, \sigma_{1}(s),ds'). \\
\end{align*}
\end{exercise}

\begin{exercise}
Consider the inventory management problem we studied in Lecture 2. 
Suppose that $s$, $a$ and $D$ are positive integers and $\Gamma(s) = [s,\overline{a}]$ for some positive integer $\overline{a}$. So $S = A = \{0, \ldots , \overline{a} \}$ and $D$ is a discrete random variable that can take values in $\{0, \ldots , \overline{a} \}$. 

1. Write the Bellman equation for this inventory management model and provide a justification for its validity.

2. Let $g(s)$ be the optimal policy and assume it is unique. Suppose that  $g(0) > 0$.  Show that there exists an $s^{*} > 0$ such that $g(s) = s^{*}$ for all $0 \leq s \leq s^{*}$ and $V(s_{2})-V(s_{1}) = c(s_{2} - s_{1})$ for all $s^{*} \geq s_{2} \geq s_{1} \geq 0$.

3. Suppose that $\overline{a}=7$ so the state space is $S=A=\{0,\ldots,7\}$, $\beta=0.95$, and the demand is uniform on $\{1,\ldots,7\}$. 

Write a Python code that implements the value function iteration algorithm to find the value function and the optimal policy function. Plot it as a function of the inventory level $s$ for the case that $w=2$, $h=c=1$.

Hint: you can start with the following parameters:  
\begin{verbatim}
    import numpy as np
import matplotlib.pyplot as plt

# Parameters
max_inventory = 7  # \overline{a}
c = 1  # cost per unit produced
h = 1  # holding cost per unit
w = 2  # revenue per unit
beta = 0.95  # discount factor
demand_range = np.array([1, 2, 3, 4, 5, 6, 7])
demand_prob = 1.0 / len(demand_range)  # Uniform distribution probability for  demand


\end{verbatim}

4. Assume you are the newly appointed Data Scientist at a retail company. You know the values for sale price per unit $w$, holding costs $h$, 
 and production costs $c$, but the distribution of customer demand remains unknown. However, you have access to historical data of past demands. Provide three different models (without implementing just discuss) to solve the dynamic inventory problem. 

Hint: Consider the use of empirical distribution, Bayesian updating models, or auto-regressive processes as potential approaches.

5. After further analysis, you realize that the demand is not I.I.D as it appears to be influenced by a variety of features such as market trends, promotional campaigns, and seasonal variations.

You have managed to access important data points regarding demand. This data comprises three key features: competitor behavior, market sentiment, and season, each categorized into binary states ${0,1}$. From this dataset, you have gathered 20 distinct periods of these features alongside the actual demand observed during each period.

You have decided to employ a two-step ``estimate and optimize" approach:

Estimation: Apply a machine learning technique to predict the demand based on the feature vectors. 

Optimization: Utilize value function iteration to solve the inventory management problem. 

The first step is done in the predict demand function below that employs the RandomForest algorithm to model and estimate the demand. 

Extend the code below to complete the second step.

*You will need to add the feature vectors to the state space and define the state  dynamics.

\begin{lstlisting}[language=Python]
    import numpy as np
import pandas as pd
from sklearn.ensemble import RandomForestRegressor

def predict_demand(features):
    # Seed for reproducibility
    np.random.seed(42)

    # Generate synthetic data for training
    num_samples = 20
    competitor_behavior = np.random.choice([0, 1], size=num_samples)
    market_sentiment = np.random.choice([0, 1], size=num_samples)
    season = np.random.choice([0, 1], size=num_samples)
    demand = np.random.randint(1, 8, size=num_samples)  
    
    # Stack features in a matrix for training
    training_features = np.column_stack((competitor_behavior, market_sentiment, season))

    # Initialize and train RandomForestRegressor
    rf_model = RandomForestRegressor(n_estimators=100, random_state=42)
    rf_model.fit(training_features, demand)

    # Predict on the provided input features
    predicted_demand = rf_model.predict([features])


\end{lstlisting}

\end{exercise}

\newpage 

\section{Lecture 3: Properties of the Value and Policy Functions}

Throughout Lecture 3, we will assume the following:
\begin{assumption}
Suppose that $r$ is bounded and upper semi-continuous (u.s.c), $p$ is u.s.c, $\Gamma$ is u.s.c, and $\Gamma(s)$ is nonempty and compact for each $s \in S$. 
\end{assumption}
Under this assumption, and as discussed in Lecture 2, $V$ is the unique bounded and u.s.c function that satisfies the Bellman equation, and there exists a Borel measurable stationary policy function that is optimal.

We will focus on proving structural properties of the value function and the optimal policy functions in this note. Properties of the value function can be proven by utilizing Corollary 1 from Lecture 2.

We will assume that $S$ and $A$ are subsets of $\mathbb{R}^{n}$ and $\mathbb{R}^{k}$, respectively, endowed with the standard product order on $\mathbb{R}^{n}$, i.e., $x \geq y$, $x,y \in \mathbb{R}^{n}$ if $x_{i} \geq y_{i}$ for each $i=1,\ldots,n$.

As usual, a function $f:S \rightarrow \mathbb{R}$ is called increasing if $f(s_{2}) \geq f(s_{1})$ whenever $s_{2} \geq s_{1}$.

\subsection{Stochastic Dominance}
We now introduce the notion of stochastic dominance, a probabilistic concept used to compare random variables, particularly useful in decision making under uncertainty, economics, and operations.

\begin{definition}
A random variable $X$ with a law $\mu_{X}$ is said to first-order stochastically dominate another random variable $Y$ with a law $\mu_{Y}$ if for every increasing function $f:S \rightarrow \mathbb{R}$ we have 
$\int f(s) \mu_{X} (ds) \geq \int f(s)  \mu_{Y} (ds)$
whenever the integrals exist, and we write $\mu_{X} \succeq_{FSD} \mu_{Y} $ or $X \succeq_{FSD} Y$. 
\end{definition}

We have the following characterization of stochastic dominance. A set $B \in \mathcal{B}(S)$ is called an upper set if $x \in B$ and $y \geq x$ imply $y \in B$. 

\begin{proposition} \label{prop:FSD}
   Let $X$ and $Y$ be two random variables on $S$. We have $\mu_{X} \succeq_{FSD} \mu_{Y} $ if and only if  $\mu_{X} (B) \geq \mu_{Y} (B) $ for every upper set $B $. 
\end{proposition}

\begin{proof}
    Suppose that $\mu_{X} \succeq_{FSD} \mu_{Y}$ and let $B$ be an upper set. Then the function $f = 1_{B}$ is an increasing function. Hence, $\mu_{X} (B) =  \int f(s) \mu_{X} (ds) \geq \int f(s)  \mu_{Y} (ds) = \mu_{Y} (B)$. 

    Now assume that $\mu_{X} (B) \geq \mu_{Y} (B) $ for every upper set $B \in \mathcal{B}(S)$. Let $f$ be an increasing function. Assuming first that $f$ is non-negative, consider the simple functions $\phi_n$ defined by 
\[
\phi_n = \sum_{i=1}^{n 2^n} 2^{-n}(i-1) \mathbf{1}_{B_n^i},
\]
with $B_n^i = \{s \in S : (i-1)2^{-n} \leq f(s) < i2^{-n}\}$ and $i(n) = n2^{n} + 1$, $B_{n}^{i(n)} = \{s \in S : (i(n)-1)2^{-n} \leq f(s) \}$.

Denote $\alpha_{i,n} = 2^{-n}(i-1)$, $\alpha_{0,n}=0$. We have
\begin{align*}
    \int \phi_{n} (s) \mu_{X} (ds) & = \sum _{i=1}^{i(n)} \alpha_{i,n} \mu_{X}  (B_{n}^{i}) =   \sum _{i=1}^{i(n)} (\alpha_{i,n} -\alpha_{i-1,n} ) \mu_{X}  \left (\cup _{i=j}^{i(n)} B_{n}^{j} \right) \\
   & \geq \sum _{i=1}^{i(n)} (\alpha_{i,n} -\alpha_{i-1,n} ) \mu_{Y}  \left (\cup _{i=j}^{i(n)} B_{n}^{j} \right) =  \int \phi_{n} (s) \mu_{Y} (ds),
\end{align*}
where the inequality follows because $\cup _{i=j}^{i(n)} B_{n}^{j}$ is an upper set for each $i$ because $f$ is increasing. Hence, taking the limit and using the integral definition in Lecture 1, we have 
\[
\int f(s) \mu_{X} (ds) =  \lim _{n \rightarrow \infty} \int \phi_{n} (s) \mu_{X} (ds) \geq    \lim _{n \rightarrow \infty} \int \phi_{n} (s) \mu_{Y} (ds) = \int f(s) \mu_{Y} (ds)
\]
which proves the result for non-negative functions. For a general $f$, as usual use $f = f^{+} - f^{-} $ (see Lecture 1). 
\end{proof}

\begin{remark}
    When $S \subseteq \mathbb{R}$,  Proposition \ref{prop:FSD} provides a very simple characterization for the stochastic dominance order $\succeq_{FSD}$ in terms of cumulative distribution functions (CDFs). We have $X \succeq _{FSD} Y$ if and only if $F_{X} (x) \leq F_{Y}(x) $ for every $x \in S$ where $F_{Z} (x) = \mathbb{P}(Z \leq x)$ is the CDF of a random variable $Z$. 
\end{remark}

\begin{remark}
    Note that the proof of Proposition \ref{prop:FSD} did not use the fact that $S \subseteq \mathbb{R}^{n}$ and can be easily generalized to a more general state space that is endowed with a partial order. 
\end{remark}

Intuitively, \(X\) is preferred to \(Y\) under FSD if \(X\) yields higher outcomes than \(Y\) with higher probabilities. This is equivalent to stating that every utility maximizer who has increasing preferences would prefer \(X\) over \(Y\).
We will say that $p$ is increasing if $g(s,a) = \int f(s') p(s,a,ds')$ is increasing  whenever $f$ is.

\begin{corollary} $p$ is increasing if and only if $p(s_{2},a_{2} , B ) \succeq _{FSD}  p(s_{1},a_{1} , B ) $ whenever $s_{2} \geq s_{1}$ and  $a_{2} \geq a_{1}$ and every upper set $B$. 
\end{corollary}

\begin{proof}
    Suppose that $p$ is increasing and let $B$ be an upper set. Then by letting $f= 1_{B}$, $f$ is increasing so monotonicity of $p$ implies $p(s_{2},a_{2} , B ) \geq  p(s_{1},a_{1} , B)$. 

The other direction follows from the definition of FSD and Proposition \ref{prop:FSD}. 
\end{proof}

\subsection{Value Function Properties}
As in the last section, 
we will say that $p$ is increasing (concave, convex, continuous, concave and increasing in $s$ or $a$, etc.) if $g(s,a) = \int f(s') p(s,a,ds')$ is increasing (concave, convex, continuous, concave and increasing in $s$, increasing in $s$, etc.) whenever $f$ is.

 Recall that $S$ is a convex set if for all $\lambda \in (0,1)$, $\lambda s_{1} + (1-\lambda ) s_{2} \in S$ whenever $s_{1},s_{2} \in S$. 

 When $S$ is convex, recall that $f:S \rightarrow \mathbb{R}$ is called a convex function if $$\lambda f(s_{1}) + (1-\lambda) f(s_{2}) \geq f(\lambda s_{1} + (1-\lambda)s_{2})$$ for all $s_{1},s_{2} \in S$, $\lambda \in (0,1)$. $f$ is concave if $-f$ is convex. 

\begin{example}
    Consider linear transitions 
$$p(s,a,B) = \phi (\zeta : cs + da + \zeta \in B)$$
where $\zeta$ is a random variable on $\mathbb{R}$ with law $\phi$, $S=A = \mathbb{R}$, $c>0$,$d>0$. 

Then $p$ is increasing, concave, and convex. To see this, note that 
$$ \int f(s') p(s,a,ds') = \int f(cs+da+\zeta)  \phi(d\zeta) $$
is increasing (concave) (convex) when $f$ is increasing (concave) (convex). 
\end{example}

 The correspondence $\Gamma$ is called convex if  $a_{1} \in \Gamma (s_{1})$ and $a_{2} \in \Gamma (s_{2})$ implies $a_{\lambda} \in \Gamma (s_{\lambda})$ for all $s_{1},s_{2}$ and $\lambda \in (0,1)$ where   
 $s_{\lambda} = \lambda s_{1} +(1-\lambda) s_{2}$ and $a_{\lambda} = \lambda a_{1} +(1-\lambda) a_{2}$.

 We have the following theorem regarding the monotonicity and concavity of the value function:

 \begin{theorem} \label{Thm:value_prop}
       (1) (Monotonicity). Suppose that $r$ is increasing in $s$, $p$ is increasing in $s$, and $\Gamma(s_{1}) \subseteq \Gamma(s_{2})$ whenever $s_{1} \leq s_{2}$. Then $V$ is increasing.  

     (2) (Concavity). Suppose that $p$ is concave, $r$ is concave, $S$ is convex, and the correspondence $\Gamma$ is convex. Then $V$ is concave.
     
     In addition, if $r$ is strictly concave in $a$, then there is a unique optimal policy function. 
 \end{theorem}

\begin{proof}
    (1) Let $D'$ be the set of bounded, u.s.c, and increasing functions on $S$.  Because $D'$ is closed, it is enough to prove that $Tf \in D'$  whenever  $f \in D'$ according to Corollary \ref{Cor:prop} in Lecture 2.

    Given the results of Lecture 2, we only need to show that $Tf$ is increasing to prove that $V \in D'$.  

    Let $s_{2} \geq s_{1}$. Because $r$ is increasing in $s$, $f$ is increasing and $p$ is increasing in $s$, for all $a \in \Gamma(s_{2})$, we have 
    \begin{align*}
        Tf(s_{2}) & \geq  r(s_{2},a) + \beta \int f(s')p(s_{2},a,ds') \\
       & \geq r(s_{1},a) + \beta \int f(s')p(s_{1},a,ds').
    \end{align*}
Using the fact that $\Gamma(s_{1}) \subseteq \Gamma(s_{2})$ we can choose $a$ to be the action that achieves the maximum on the right-hand-side of the last inequality to conclude that $Tf(s_{2}) \geq Tf(s_{1})$, i.e., $Tf$ is increasing which concludes the proof.

(2) Given part (1) it is enough to prove that $Tf:S \rightarrow \mathbb{R} $ is concave whenever $f:S \rightarrow \mathbb{R}$ is concave to prove the result. 

Let $$h(s,a) = r(s,a)  + \beta \int f(s')p(s,a,ds') $$ and let $f$ be concave on $S$.  Because $r$ is concave, $p$ is concave, and the sum of concave functions is concave, then $h$ is concave. 

Let $s_{1},s_{2} \in S$ and $s_{\lambda} = \lambda s_{1} +(1-\lambda) s_{2}$, $\lambda \in (0,1)$, so $s_{\lambda} \in S$. Assume that $a_{i}$ attains the maximum of $h(s_{i},a)$ for $i=1,2$ and define $a_{\lambda} = \lambda a_{1} +(1-\lambda) a_{2}$. We have 
\begin{equation*}
    Tf(s_{\lambda }) \geq h(s_{\lambda},a_{\lambda}) \geq \lambda h(s_{1},a_{1}) + (1-\lambda)h(s_{2},a_{2}) = \lambda Tf(s_{1}) + (1-\lambda)Tf(s_{2})
    \end{equation*}
so $Tf$ is concave. The first inequality follows because $\Gamma$ is convex so $a_{\lambda} \in \Gamma (s_{\lambda})$, the second inequality from convexity of $h$, and the equality follows from the definition of $Tf$. 

 In addition, if $r$ is strictly concave in $a$, then $r(s,a) + \int V(s') p(s,a,ds')$ is strictly concave in $a$ so it has a unique maximizer for each $s$.   
\end{proof}

\begin{example}
    Consider the consumption-savings problem we presented in Lecture 2. If $u$ is strictly increasing and strictly concave, then $V$ is increasing and concave and the optimal policy is unique.

    To see this, note that $r(s,a) = u(s-a)$ is increasing in $s$, concave in $(s,a)$, and strictly concave in $a$. $\Gamma (s) = [0,s]$ is convex and satisfies $\Gamma(s_{1}) \subseteq \Gamma(s_{2})$ whenever $s_{1} \leq s_{2}$. In addition $p$ is increasing and concave. 

    This leads to intuitive insights that the value function is increasing and concave in wealth (wealth has a positive decreasing marginal value). This can also simplifies the optimization step in the value function iteration algorithm.  
\end{example}

To provide conditions for the differentiability of the value function, we cannot use Corollary \ref{Cor:prop} from Lecture 2, as the set of differentiable functions is not closed (e.g., consider the sequence of differentiable functions \( f_n = \sqrt{x^2 + \frac{1}{n}} \), which converges uniformly to \(|x|\), but \(|x|\) is not differentiable at $0$). Despite this, we can leverage concavity to provide conditions for differentiability for state variables that affect only the payoff function. This property can also be achieved through a change of variables in some applications.

\begin{theorem} (Differentiability). \label{thm:value_diff}
    Suppose that $S=S_{1} \times S_{2}$ for some convex sets $S_{1} \subseteq \mathbb{R}^{n_{1}}$  and $S_{2} \subseteq \mathbb{R}^{n_{2}}$. Assume that the transition function $p$ does not depend on $s_{1}$ and the reward function $r(s_{1},s_{2},a)$ is differentiable and concave in $s_{1}$.  

Assume that $(s_{1}(0),s_{2}(0),a^{*})$ belongs to the interior of the graph of  $\Gamma$ where $a^{*}$ is an optimal action under the state $(s_{1}(0),s_{2}(0))$. Then $V$ is differentiable at $(s_{1}(0),s_{2}(0))$ and the  derivative  is given by 
    $$\frac{\partial }{\partial s_{1} } V(s_{1}(0),s_{2}(0)) = \frac{\partial }{\partial s_{1} } r(s_{1}(0),s_{2}(0),a^{*}).$$
\end{theorem}

\begin{proof}
 Recall that the set of subgradients of a concave function $f:X \rightarrow \mathbb{R}$ at a point $x_{0} \in X$ for some open set $X$ consists of all the vectors $y \in \mathbb{R}^{n}$ such that $ f(x) -f(x_{0}) \leq y \cdot (x-x_{0}) $ for all $x \in X$. This set is not empty and single-valued if and only if $f$ is differentiable at $x_{0}$ and the derivative equals the element in the set.

   Let $$h(s_{1},s_{2},a) =  r(s_{1},s_{2},a)  + \beta \int _{S} V(s_{1}',s_{2}')p(s_{2},a,d(s_{1}',s_{2}')) .$$ 
 Because $a^{*}$ is an optimal action we have $V(s_{1}(0),s_{2}(0)) = h(s_{1}(0),s_{2}(0),a^{*})$ and $V(s_{1},s_{2}) \geq  h(s_{1},s_{2},a)$ for all $a \in \Gamma(s_{1},s_{2})$.  

  In addition, $h$ is differentiable and concave in $s_{1}$ as $r$ is.
 
 Thus, if $y$ is a subgradient of $f:X \rightarrow \mathbb{R}$ given by $f(x) := V(x,s_{2}(0))$ defined on some neighborhood $X$ of $s_{1}(0)$ such that $(x,s_{2}(0),a^{*}) \in \operatorname{Gr} \Gamma$ for each $x \in X$ (there is such an $X$ by assumption) then 
$$h(x,s_{2}(0),a^{*}) - h(s_{1}(0),s_{2}(0),a^{*}) \leq f(x) -f(x_{0}) \leq y \cdot (x-x_{0})  $$
implies that $y$ is a subgradient of $h$. Because $h$ is differentiable in the first argument it has a unique subgradient at $s_{1}(0)$ which equals its derivative, i.e., 
 $$\frac{\partial }{\partial s_{1} } V(s_{1}(0),s_{2}(0)) =  \frac{\partial }{\partial s_{1} } h(s_{1}(0),s_{2}(0),a^{*}) = \frac{\partial }{\partial s_{1} } r(s_{1}(0),s_{2}(0),a^{*})$$
 which completes the proof. 
\end{proof}

Another important property is supermodularity. Recall that a partially ordered set is given by a set $X$ and a partial order on $X$. A partial order $\geq$ satisfies three properties - reflexivity: $a \geq a$ for each $a$, antisymmetric: $a \geq b$ and $b \geq a$ imply $a=b$, and transitivity: $a \leq b$ and $b \leq c$ imply $a \leq c$. A lattice $X$ is a partially ordered set such that each two elements $x,y \in X$ has a least upper bound (denoted by $x \lor y$) and a greatest lower bounded (denoted by $x \land y$). For simplicity, we will focus on the case that $X$ is in $\mathbb{R}^{k}$ and $x \lor y$ is just the pointwise maximum $x \lor y = (\max (x_{1} , y_{1}) , \ldots, \max (x_{n},y_{n}) )  $, $x \land y$ is the poinwise minimum $x \land y = (\min (x_{1} , y_{1}) , \ldots, \min (x_{n},y_{n}) )  $ . 

A function $f:X \rightarrow \mathbb{R}$ is said to be supermodular on a lattice $X$ if 
$$ f(x) + f(x') \leq f(x \lor x') + f(x \land x')$$
for all $x,x' \in X$. $f$ is submodular if $-f$ is supermodular 

It is easy to show that the sum of supermodular functions is supermodular and that supermodularity is closed under pointwise convergence. $p$ is called supermodular if $g(s,a) = \int f(s') p(s,a,ds')$ is supermodular whenever $f$ is. 

We will say that the correspondence $\Gamma$ is supermodular if $a' \in \Gamma(s')$ and $a'' \in \Gamma(s'')$ imply $a' \lor a'' \in \Gamma (s' \lor s'')$ and $a' \land a'' \in \Gamma (s' \land s'')$ for all $s',s''$. In other words, $\operatorname{Gr} \Gamma$ is a lattice in $S \times A$. 

We have the following result regarding the preservation of supermodularity by the maximum operator.
\begin{lemma} \label{lemma:maxsuper}
Suppose that $h:\operatorname{Gr} \Gamma \rightarrow \mathbb{R}$ is supermodular on the lattice  $\operatorname{Gr} \Gamma$.  Define $ m(s) = \sup _{a \in \Gamma(s)} h(s,a) $.  Then $m$ is supermodular. 
\end{lemma}

\begin{proof}
    Let $s',s'' \in S$, $a' \in \Gamma(s')$ and $a'' \in \Gamma(s'')$. Because $h$ is supermodular we have 
\begin{align*}
    h(s',a') + h(s'',a'') \leq h(s' \lor s'', a' \lor a'') + h(s' \land s'', a' \land a'') \leq m(s' \lor s'') + m(s' \land s'')
\end{align*}
where the second inequality follows because $\Gamma$ is supermodular so $a' \lor a'' \in \Gamma (s' \lor s'')$ and $a' \land a'' \in \Gamma (s' \land s'')$. 
\end{proof}

Hence, we apply again Corollary 1 from Lecture 2 to provide conditions that imply the supermodularity of the value function $V$. 

\begin{theorem} \label{Them:supermodular} (Supermodularity). 
    Suppose that $r$ is supermodular, $p$ is supermodular and $\Gamma$ is supermodular  when $A$ and $S$ are lattices in $\mathbb{R}^{k}$ and $\mathbb{R}^{n}$. Then $V$ is supermodular.
\end{theorem}

\begin{proof}
  Given the considerations of Theorem \ref{Thm:value_prop} and the fact that the set of supermodular functions is closed, it is enough to show that $Tf$ is supermodular whenever $f$ is.   Let $h(s,a) = r(s,a) + \beta \int f(s')p(s,a,ds') $ and let $f$ be supermodular. Then because $p$ is supermodular, $h$ is supermodular as the sum of two supermodular functions.  Lemma \ref{lemma:maxsuper} shows that $Tf$ is supermodular which concludes the proof.
\end{proof}

\subsection{Properties of the Optimal Policy Function}

In this section we will provide conditions that ensure that the optimal policy is increasing.

Let $E$ be a partially ordered set with an order $\succeq$. 
A function $f :S \times E \rightarrow \mathbb{R}$ is said to have increasing differences in $(s ,e)$ on $S \times E$ if for all $e_{2} ,e_{1} \in E$ and $s_{2} ,s_{1} \in S$ such that $e_{2} \succeq e_{1}$ and $s_{2} \geq s_{1}$, we have \begin{equation*}f(s_{2} ,e_{2}) -f(s_{2} ,e_{1}) \geq f(s_{1} ,e_{2}) -f(s_{1} ,e_{1}) .
\end{equation*} A function $f$ has decreasing differences if $-f$ has increasing differences. Note that $S$ and $E$ are not required to be  a lattice for the definition of increasing differences (it is clear that supermodularity on $S \times E $ implies increasing differences when $S$ and $E$ are lattices).

 Let $X_{i}$ be a lattice in $\mathbb{R}$ and consider the lattice $X=\times_{i=1}^{n} X_{i}$ in $\mathbb{R}^{n}$. Then $f$ has increasing differences on $X$ if it has increasing differences in each pair ($x_{i},x_{j}$) on $X_{i} \times X_{j}$, $i \neq j$, $1 \leq i,j \leq n$. 
Increasing differences is a natural concept in economics. Suppose that $f$ is a production function,  then  increasing differences can signify that the effectiveness of one input increases with the level of another input, i.e., the inputs are complementarities. As an example, $f(x) = x_{1}^{\alpha_{i} } \cdot .... \cdot x_{n}^{\alpha_{i}} $ has increasing differences on $\{x \in \mathbb{R}^{n}: x \geq 0 \}$ and is called the Cobb-Douglas production function.

 Increasing differences is equivalent to supermodularity under certain conditions as the following Proposition shows. 

\begin{proposition} \label{prop:supermodular}
    Let $X_{i}$ be a lattice in $\mathbb{R}$ and consider the lattice $X=\times_{i=1}^{n} X_{i}$ in $\mathbb{R}^{n}$ under the standard product order. Then $f:X \rightarrow \mathbb{R}$ has increasing differences on $X$ if and only if $f$ is supermodular on $X$. 
    
    If $f$ is twice  continuously differentiable then it has increasing differences if and only if $\partial ^{2} f / \partial x_{i} \partial x_{j} \geq 0$ for all $i \neq j$. 
\end{proposition}

Suppose that $X$ is a lattice with the order $\geq$. For two nonempty subsets of $X$, say $D$ and $E$, write $D \succeq _{set} E$ if $x \in D$ and $x' \in E$ imply $x \lor x' \in D$ and $x \land x' \in E$. A correspondence $\Lambda$ from $X$ to $E$ where $\Lambda$ is a (sub)-lattice for each $x$, is ascending if $\Lambda (x_{2}) \succeq_{set} \Lambda(x_{1})$ whenever $x_{2} \geq x_{1}$.  Note that if $\Lambda$ is supermodular then it is ascending.  

The following is a key result in comparing the optimal solutions of a parameterized optimization problem. 

\begin{proposition} \label{Prop:topkis}
    Suppose that $X  
    \subseteq \mathbb{R}^{n}$ is a lattice, $E$ is a partially ordered set, $\Lambda(e) \subseteq X$ is non-empty, compact for each $e \in E$, and ascending. $f:X \times E \rightarrow \mathbb{R}$ is supermodular in $x$, has increasing differences, and is u.s.c in $x$.  Let 
    $$G(e) = \operatorname{argmax} _{x \in \Lambda(e)} f(x,e)$$ 
    
    Then  $G(e)$ is non-empty and ascending.

\end{proposition} 
\begin{proof}
The fact that $G$ is non-empty is immediate from Lecture 2.  

Note that $G(e)$ is a sublattice in $X$. 
To see this, let $x , y \in G(e)$ and observe that 
$$ 0 \leq f(x,e) - f(y \land x,e) \leq f(y \lor x , e) - f(y,e) \leq 0 $$
so $y \land x$ and $y \lor x$ belong to $G(e)$.

Now let $e'' \succeq e'$, $x' \in G(e')$ and $x'' \in G(e'')$. Then because $\Lambda$ is ascending, we have $x' \lor x'' \in \Lambda(e'')$ and $x' \land x'' \in \Lambda(e')$. We have 
\begin{align*}
   0 & \leq f(x',e') - f(x' \land x'',e')  \leq f(x' \lor x'',e') - f(x'',e')   \leq f(x' \lor x'',e'') - f(x'',e'') \leq 0.
\end{align*}
The first and last inequalities follow from the optimality $x'$ and $x''$. The second inequality follows because $f$ is supermodular in $x$. The third inequality follows because $f$ has increasing differences.  Thus, $x' \lor x'' \in G(e'')$ and $x' \land x'' \in G(e')$ which prove the result.

\end{proof}

\begin{example}
    Suppose that $f$ is a continuous increasing production function (see above) from $\mathbb{R}^{n}_{+}$ to $\mathbb{R}$. We can choose a vector of inputs $(x_{1},\ldots,x_{n})$ from some compact set $X$ and we are trying to maximize profits $pf(x_{1},\ldots,x_{n})-\sum _{i=1}^{n} c_{i} x_{i}$ where $c_{i} \geq 0$ is the cost of a unit of input $x_{i}$ and $p>0$ the price of the output we produce. From Proposition \ref{Prop:topkis}, if $f$ is supermodular, an increase in the price leads to an increase of all inputs. More precisely, $$G(p) = \operatorname{argmax}_{x \in X}pf(x_{1},\ldots,x_{n})-\sum _{i=1}^{n} c_{i} x_{i}$$     
is ascending (why)? 
\end{example}

Importantly, Proposition \ref{Prop:topkis} leads to the following Corollary regarding the set of optimal solutions for the dynamic programming problem.  

\begin{theorem} \label{Thm:Optimal Solutions}
    Suppose that $A$ and $S$ are lattices in $\mathbb{R}^{k}$ and $\mathbb{R}^{n}$, respectively. Suppose that $r$, $p$, and $\Gamma$   are supermodular. 
    Then the set of optimal solutions of the discounted dynamic programming problem 
    $$G(s) = \operatorname{argmax} _{a \in \Gamma(s)} r(s,a) + \beta \int V(s')p(s,a,ds')$$
    is ascending. 
\end{theorem}

  \begin{proof}
      From Theorem \ref{Them:supermodular} the value function $V$ is supermodular. Hence, 
      $r(s,a) + \beta \int V(s')p(s,a,ds') $ is supermodular because $p$ is supermodular and  the sum of supermodular functions is supermodular. This implies that  $r(s,a) + \beta \int V(s')p(s,a,ds') $ has increasing differences in $(s,a)$ from Proposition \ref{prop:supermodular}, so the result follows from Proposition \ref{Prop:topkis}.
      \end{proof}

\begin{example}
    Consider again the consumption-savings problem from Lecture 1 when the utility function $u$ is  increasing and concave. Then we can use Theorem \ref{Thm:Optimal Solutions} to show that savings are increasing with wealth (why?). 
\end{example}

\subsection{Parameterized Dynamic Programming}

In this section, we study a broader scope of dynamic programming models by incorporating parameters that influence the primitives of the model. We introduce a parameter \( e \) from a set \( E \) and to highlight the dependence on this parameter, we denote, with slight abuse of notations, \( \Gamma(s,e) \), \( r(s,a,e) \), and \( p(s,a,e,B) \), the set of feasible actions, the payoff function, and the transition function, respectively. 

The set \( E \) is assumed to be a partially ordered set, equipped with an order denoted by \( \succeq \). This ordering can represent  varying discount factors, a parameter that impacts the payoff function, or different transition functions ordered by first-order stochastic dominance. Such parameterization allows us to derive comparative statics results with respect to the primitives of the dynamic programming model.

We will need the following Lemma that provides conditions for the preservation of increasing differences for the maximum operator. 
\begin{lemma} \label{Lemma:Increasing_d}
   Let $S,A$ be lattices and $E$ a partially ordered set.
   
   Assume that the correspondence  $\Gamma:S \rightarrow 2^{A}$ is supermodular, $h:S \times A \times E \rightarrow \mathbb{R}$ is bounded, supermodular in $(s,a)$, and has increasing differences. Then 
   $$ m(s,e ) = \max _{a \in \Gamma(s)} h(s,a,e)$$
   has increasing differences. 
\end{lemma}

\begin{proof}
  Suppose that $s_{2} \geq s_{1} $,  $e_{2} \succeq e_{1}$, $s_{2},s_{1} \in S$, $e_{2},e_{1} \in E$, and $a_{2} \in \Gamma(s_{2})$, $a_{1} \in \Gamma(s_{1})$. We have
    \begin{align*}
        m(s_{2},e_{2}) + m(s_{1},e_{1}) & \geq h(s_{2},a_{2} \lor a_{1}, e_{2}) + h(s_{1},a_{2} \land  a_{1}, e_{1}) \\
        & = h(s_{2},a_{2} \lor a_{1}, e_{2}) + h(s_{1},a_{2} \land  a_{1}, e_{1}) \\
        & + h(s_{1},a_{2} \land a_{1}, e_{2}) - h(s_{1},a_{2} \land  a_{1}, e_{2}) \\
        & \geq  h(s_{2},a_{2} , e_{2}) + h(s_{1},a_{2} \land  a_{1}, e_{1}) \\
        & + h(s_{1}, a_{1}, e_{2}) - h(s_{1},a_{2} \land  a_{1}, e_{2}) \\
        & =  h(s_{2},a_{2} , e_{2}) + h(s_{1},a_{2} \land  a_{1}, e_{1}) + h(s_{2},a_{2},e_{1}) \\
        & + h(s_{1}, a_{1}, e_{2}) - h(s_{1},a_{2} \land  a_{1}, e_{2}) - h(s_{2},a_{2},e_{1}) \\ 
        & \geq  h(s_{1},a_{1},e_{2}) + h(s_{2},a_{2},e_{1}). 
    \end{align*}
    Taking the maximum on the right-hand-side yields that $m$ has increasing differences. The first inequality follows from the supermodularity of $\Gamma$. The second inequality follows from supermodularity of $h$. The third inequality follows from increasing differences of $h$. 
\end{proof}

\begin{remark}
    (i) Lemma \ref{Lemma:Increasing_d} holds also when $\Gamma(s)$ is not compact. The proof is the same where instead of maximum we use supremum.

    (ii) Lemma \ref{Lemma:Increasing_d} also holds when $\Gamma:S \times E \rightarrow 2^{A}$ as long as $a_{2} \lor a_{1} \in \Gamma(s_{2},e_{2})$ and $a_{2} \land a_{1} \in \Gamma(s_{1},e_{1})$ whenever $s_{2} \geq s_{1}$, $e_{2} \succeq e_{1}$, $a_{1} \in \Gamma(s_{1},e_{2})$, and $a_{2} \in \Gamma (s_{2},e_{1})$. The proof is the same but the first inequality and the last step require the assumption above. 

    (iii) The assumption that $S$ and $A$ belong to $\mathbb{R}^{n}$ and $\mathbb{R}^{k}$ is not required for the proof of Lemma \ref{Lemma:Increasing_d}. 
\end{remark}

Let \( G(s, e) \) be the optimal policy correspondence for the discounted dynamic programming model parameterized by \( e \), as described in Theorem \ref{Thm:Optimal Solutions}. The following theorem provides comparative statics results regarding changes in the discount factor and the payoff function.

\begin{theorem} \label{Thm:discount}
    Suppose that  $A$ and $S$ are lattices in $\mathbb{R}^{k}$ and $\mathbb{R}^{n}$, respectively.

     (i) (A change in the discount factor):    Suppose that $r$ is supermodular $(s,a)$ and is increasing in $s$, $p$ is increasing and supermodular in $(s,a)$, $\Gamma$ is supermodular, and $\Gamma(s_{1}) \subseteq \Gamma(s_{2})$ whenever $s_{1} \leq s_{2}$. 
    
   Then $G(s,\beta)$ is ascending. 

(ii) (A parameter that influences the payoff function): Let $c \in E$ be a parameter that influences the payoff function and we write $r(s,a,c)$ instead of $r(s,a)$ where $E$ is a partially ordered set. 
Suppose that $r$ in supermodular in $(s,a)$, has increasing differences in $(s,a,c)$, $p$ is  increasing and supermodular in $(s,a)$, and $\Gamma$ is supermodular. 

Then $G(s,c)$ is ascending. 
\end{theorem}

\begin{proof}
(i)  Let $E =(0 ,1)$ be the set of all possible discount factors, endowed with the standard order: $\beta _{2} \geq \beta _{1}$ if $\beta _{2}$ is greater than or equal to $\beta _{1}$. Assume that $\beta _{1} \leq \beta _{2}$. Let $f \in B(S \times E)$ and assume that $f$ has increasing differences in $(s ,\beta )$, is supermodular in $s$, and is increasing in $s$. Given the discussion in Theorem \ref{Thm:value_prop} and Theorem \ref{Them:supermodular}, and the results in Lecture 2, it is enough to show that $Tf$ has increasing differences in order to prove that $V$ has increasing differences in $(s,\beta)$, and is supermodular and increasing in $s$. 

Let $a_{2} \geq a_{1}$. Since $f$ has increasing differences, the function $f(s ,\beta _{2}) -f(s ,\beta _{1})$ is increasing in $s$. Since $p$ is monotone we have\begin{equation*}\int _{S}(f(s^{ \prime } ,\beta _{2}) -f(s^{ \prime } ,\beta _{1}))p(s ,a_{2} ,ds^{ \prime }) \geq \int _{S}(f(s^{ \prime } ,\beta _{2}) -f(s^{ \prime } ,\beta _{1}))p(s ,a_{1} ,ds^{ \prime }) .
\end{equation*}Rearranging the last inequality yields \begin{equation*}\int _{S}f(s^{ \prime } ,\beta _{2})p(s ,a_{2} ,ds^{ \prime }) -\int _{S}f(s^{ \prime } ,\beta _{2})p(s ,a_{1} ,ds^{ \prime }) \geq \int _{S}f(s^{ \prime } ,\beta _{1})p(s ,a_{2} ,ds^{ \prime }) -\int _{S}f(s^{ \prime } ,\beta _{1})p(s ,a_{1} ,ds^{ \prime }) .
\end{equation*}Since $f$ is increasing in $s$ and $p$ is increasing, the right-hand-side and the left-hand-side of the last inequality are non-negative. Thus, multiplying the left-hand-side of the last inequality by $\beta _{2}$ and the right-hand-side of the last inequality by $\beta _{1}$ preserves the inequality. Adding to each side of the last inequality $r(s,a_{2}) -r(s,a_{1})$ yields \begin{equation*}h(s ,a_{2} ,\beta _{2} ,f) -h(s ,a_{1} ,\beta _{2} ,f) \geq h(s ,a_{2} ,\beta _{1} ,f) -h(s ,a_{1} ,\beta _{1} ,f) .
\end{equation*}
where 
$$h(s,a,\beta,f) = r(s,a) + \beta \int f(s,\beta) p(s,a,ds').  $$
That is, $h$ has increasing differences in $(a ,\beta )$. An analogous argument shows that $h$ has increasing differences in $(s ,\beta )$. The proof of Theorem \ref{Thm:Optimal Solutions} shows that $h$ is supermodular in $(s,a)$, and hence, has increasing differences in $(s ,a)$. Thus, $Tf$ has increasing differences from Lemma \ref{Lemma:Increasing_d}.

We conclude that $V$ has increasing differences in $(s,\beta)$. This implies that $h(s,a,\beta,V)$ has increasing differences in $(s,a,\beta)$ as the sum of functions with increasing differences by the argument above. In addition, $V$ is supermodular from Theorem \ref{Them:supermodular}, and hence $h$ is supermodular in $(s,a)$ as the sum of supermodular functions.  Now we can apply Proposition \ref{Prop:topkis} to conclude the result. 

(ii) The proof of part (ii) is similar to the proof of part (i) and is therefore omitted. 
\end{proof}

\textbf{Bibliography note}. Theorems \ref{Thm:value_prop} and \ref{thm:value_diff} are well-known but the precise formulation above seems to be new. 
\cite{topkis1998supermodularity} provides a comprehensive study of supermodularity and related properties and results we present above. The proof of Lemma \ref{Lemma:Increasing_d} is adopted from \cite{lovejoy1987ordered} and \cite{hopenhayn1992}. The proof of Theorem \ref{Thm:discount} is a adopted from \cite{light2021stochastic}.

\subsection{Exercise 3}

\begin{exercise}
    Let $p$ be a transition kernel where $A$ and $S$ are lattices in $\mathbb{R}^{k}$ and $\mathbb{R}^{n}$, respectively.  For some of the comparative statics results we proved in class we needed the assumption that $p$ is increasing  and supermodular.
    
    Show that if $p(s,a,B)$ is supermodular and increasing in $(s,a)$ for every upper set $B \in \mathcal{B}(S)$  then $p$ is increasing and  supermodular. 

    Hint: Take a close look on the proof of Proposition 1 in Lecture 3.
\end{exercise}

\begin{exercise}
    Consider a portfolio optimization problem involving $n$ financial assets, similar to what we have discussed in class. Each asset offers a random, non-negative return, and the returns have a support $[0,\overline{R}]$ for some finite $\overline{R} > 0$ and are independent over time. 

    In each period, the agent's wealth is denoted by $x \in X = [0, \infty)$, where the agent chooses a non-negative investment $a_i$ for each asset $i$, such that the total investment does not exceed the available wealth: $\sum_{i=1}^n a_i \leq x$ (no borrowing).

    The agent also receives a stochastic income $Y(t)$ in each period that is independent of the returns and is modeled as a finite Markov chain with transition probabilities $Q_{ij}$. This means if the current income is $y_i$, then the next period's income will be $y_j$ with probability $Q_{ij}$ and $\mathcal{Y} = \{y_{1},\ldots,y_{m} \}$ are the possible income ordered $y_{i} < y_{j}$ iff $i<j$. 

    The wealth of the agent in the next period is given by:
    \[
    x(t+1) = \sum_{i=1}^n a_i R_i(t) + Y(t+1),
    \]
    where $R_i(t)$ represents the return of the $i$-th asset in period $t$.

    The agent's utility from consumption is represented by a bounded, strictly concave, continuous, and increasing real-valued utility function $u$. The payoff in each period is given by 
    $u(x - \sum_{i=1}^n a_i)$, 
    where $a = (a_1, \dots, a_n)$ represents the vector of investments in each asset. The agent's goal is to maximize discounted payoffs.

   (i) Write the state space and the Bellman Equation for this problem. Explain why the Bellman equation holds given the result in Lecture 2. 

   Hint: it may be beneficial to define $S = X \times \mathcal{Y}$.

   (ii) Show that the value function is concave in $x$ and the optimal policy is unique. 
   
   (iii) Suppose that $\sum _{j=i}^{m} Q_{kj}$ is increasing in $k$ for each $i$ (justify why this assumption makes sense in this model). Show that the value function is increasing. 

   (iv)  Let $g$ be the (unique) optimal policy and $g_{i}$ be the optimal investment in asset $a_{i}$.
   
   Explain why we can't use directly the results from Lecture 3 to show that $g$ is increasing in $x$? (is the correspondence $\Gamma$ ascending)?
   
 Despite this, show that $g_{i}$ is increasing in the wealth $x$.

   Hint: Consider the change of variables $\sum _{i=1}^{k} a_{i}  = z_{k}$ and try to use the results of Lecture 3.

   (v) Now suppose that there only two assets. One yield 5 percent for sure (a ``bond") and one yields $0$ percent with probability $1/2$ and $16$ percent with probability $1/2$ (a ``stock"). In addition for simplicity there are only two possible incomes $1$ or $1.5$ and $Q_{jj} = 0.7$ for $j=1,2$. Assume the utility function is $u(c) = c^{0.5}$. 

   Write a Python code and solve the problem using value function iteration and print or plot the policy functions. Note that the state space is continuous so you will need a grid. 

 Start with the following code and then apply value function iteration (consider using np.searchsorted or any other way to make sure that the next state remains on the grid in your value function iteration algorithm). 

\begin{verbatim}
    import numpy as np

# Initialize parameters
num_wealth_states = 30
max_wealth = 3  # maximum wealth considered
wealth_grid = np.linspace(0, max_wealth, num_wealth_states)

num_actions = 10  # Discretize the action space into 10 possible investments per asset
action_grid = np.linspace(0, max_wealth, num_actions)

num_income_states = 2
income_states = np.array([1, 1.5])
transition_probabilities = np.array([[0.7, 0.3], [0.3, 0.7]])  # Transition probabilities 

discount_factor = 0.95
return1 = 1.05
return2 = np.array([1.0, 1.16])  # Two possible returns for the second asset
\end{verbatim}

\end{exercise}

\begin{exercise}
    
Consider the problem of Bayesian sequential hypothesis testing, where you are deciding between two hypotheses based on sequentially observed data. 

You have two hypothesis $H_{0},H_{1}$. Hypothesis $H_{i}$ is that the data generated from a probability density function $f_{i}$, $i=0,1$ where $f_{0},f_{1}$ are density functions with support on $Z \subseteq \mathbb{R}$.   
The state of the system is the probability that \(H_0\) is true given all observations so far. We denote this probability by $p$.

In each period, you can 
   continue testing, receive a new observation and update your belief using Bayes' formula
   or stop testing and choose one of the hypothesis.  The reward  for choosing \(H_0\) is \(p\) and for choosing \(H_1\) is \(1-p\).

As usual, the objective is to maximize the expected discounted reward where 
the discount factor $\beta < 1$ reflects that continuing testing is costly.

(1) Write down the Bellman equation for this problem. Explain why it holds.

(2) Explain intuitively why the value function \( V \) is convex, using the following somewhat informal argument:

Suppose that powerful nature plays a game. 
Nature flips an unbiased coin. With probability \( \lambda \), the coin lands heads and  \( p_{0} \) is the probability that \( H_0 \) is correct and $1-p_{0}$ is the probability that $H_{1}$ is correct. With probability \( 1-\lambda \) the coin lands tails and  \( p_{1} \) is the probability that \( H_0 \) is correct and $1-p_{1}$ is the probability that $H_{1}$ is correct.

Hint: the more information the better... 

\end{exercise}

\begin{exercise}
Consider a firm   that has a production function \( f:\mathbb{R}^{n}_{+} \rightarrow \mathbb{R}_{+} \). It takes a non-negative input vector \((s_{1}, \ldots, s_{n})\) in $\mathbb{R}^{n}_{+}$ and generates a non-negative output \( f(s_{1}, \ldots, s_{n}) \) in $\mathbb{R}_{+}$. The state is represented by the vector of inputs \( s= (s_{1}, \ldots, s_{n})\). For example, this vector could represent different resources or factors of production such as labor, capital, raw materials, machinery, or technology.

At each period a firm can invest in increasing the inputs in order to produce more.  The action taken at each time period is a vector of non-negative investments \(a=(a_{1}, \ldots, a_{n})\).  Each \( a_{i} \) represents the amount invested in the \( i \)-th input.
The cost of investing in each input \( a_{i} \) is given by a  cost function \( c_{i}(a_{i}) \) from $\mathbb{R}_{+}$ to $\mathbb{R}_{+}$.

 The per-period payoff function $r$ is given by the current production output multiplied by the price \( p \), minus the total cost of investments, i.e., 
   \[
   r(s,a) = p f(s_{1}, \ldots, s_{n}) - \sum_{i=1}^{n} c_{i}(a_{i})
   \]
   Here, \( p f(s_{1}, \ldots, s_{n}) \) is the total revenue generated from the production, and \( \sum_{i=1}^{n} c_{i}(a_{i}) \) is the total cost of investments.

If the current state of input $i$ is $s_{i}$ and the firm invests $a_{i}$ then the next period's input is 
\begin{equation} \label{Eq:dynamicsExam}
     s_{i}' = ((1-\delta_{i}) s_{i} + d_{i}a_{i}) \epsilon_{i}
     \end{equation}
where $d_{i} \geq 0$ is a parameter that describes the effect of investment, $\delta_{i} \in (0,1]$ is the depreciation rate of input $i$, and $\epsilon_{i}$ is a non-negative random variable for each $i=1,\ldots,n$.   

The goal of the firm is to maximize the expected discounted payoffs with a discount factor $\beta$ (as in Lecture 2). 

Suppose that $f$ and $-c_{i}$ are upper semi-continuous for each $i$. Assume that $f$ is bounded, $f$ and $c_{i}$ are strictly increasing, and  $c_{i}$ is strictly convex for each $i$.

(i) Show that the value function $V$ for this problem satisfies the Bellman equation and provide the Bellman equation. 

(ii) Show that $V$ is increasing.

(iii) Suppose that $f$ is concave. Show that $V$ is concave and that the optimal policy is unique. 

(iv) Assume that $f$ is concave and differentiable in $s_{i}$  and $\delta_{i}=1$ for some $i$. Find the derivative of $V$ with respect to $s_{i}$ and provide an interpretation and intuition for this derivative.  

(v) Suppose that $f$ is supermodular. Show that when the price $p$ increases, the optimal investment correspondence is ascending in $p$ (in the sense of Lecture 3). 

(vi) Suppose that $f$ is supermodular and concave. By part (iii) there is a unique optimal policy $\lambda(s)=(\lambda_{1}(s),\ldots,\lambda_{n}(s))$ where $\lambda_{i}$ is the optimal investment in input $i$.

Assume $d_{i}>0$ and show that 
$$ \frac{ (1-\delta_{i})(s'_{i} - s_{i})}{d_{i}} \geq \lambda_{i}(s) - \lambda_{i}(s') $$
for  $i=1,\ldots,n$ whenever $s' \geq s$. 
\end{exercise}

\newpage 

\section{Lecture 4: Introduction to Reinforcement Learning}

Throughout the notes we consider a finite state space $S$ with a random payoff function $R(s,a,s')$ (see also Remark 2 in Lecture 2). So in the notation of Lecture 2, the payoff function is given by $r(s,a) = \sum _{s
 \in S} p(s,a,s')R(s,a,s')$. 

In many practical applications, the state space can be extremely large—sometimes exponentially large in the number of state variables. This complexity makes it impossible to compute and store the value function for every possible state, due to both computational constraints. The computational burden of solving the Bellman equation increases significantly with the size of the state space, making traditional methods like value function iteration impractical. This phenomenon is called ``the curse of dimensionality".  

In addition, in many applications the payoff and transition functions may be unknown. We will discuss some basic ideas to overcome these issues in this Lecture. In Section \ref{Section:StateAgg} we will present a state aggregation method to overcome large state spaces. In Sections \ref{Section:Martingales} and \ref{Section:Convergence} we will cover  material from probability theory and optimization such as conditional expectations, and stochastic gradient descent algorithms that are needed to analyze the Q-learning algorithm. In Section \ref{Section:Q-learning} we will present and analyze the Q-learning algorithm. In Section \ref{Section:Average} we will present average reward dynamic programming and in Section \ref{Section:Policy} we will discuss policy gradient methods.

\subsection{State Aggregation} \label{Section:StateAgg}

State aggregation is a method to reduce the dimensionality of the state space by grouping ``similar" states into aggregates or clusters. Each cluster is treated as a single state in the aggregated model, significantly reducing the size of the state space.

The value function is then defined over these aggregated states rather than individual states. This reduces the number of calculations and the amount of memory required to implement 
 dynamic programming algorithms.

Hence, aggregation decreases the number of states and makes it possible to solve large dynamic programming problems. However, it can have some serious limitations depending on the specific dynamic optimization problem. Sometimes state are too different and aggregation can lead to highly sub-optimal policies. 

There are various ways to approach state aggregation. Let's introduce a finite set $X$ of aggregate states that partitions the state space $S$.

For each aggregate state define the weight $d_{xi}$ as the degree to which $x$ is represented by $i$. We will assume that the support of $d_{xi}$ is on the states in the  corresponding aggregate state $x$ and $\sum_{i \in N(x)} d_{xi} = 1$ where $N(x)$ contains the original states clustered to the aggregate state $x$. For example, $N(x_{1}) = \{s_{1},s_{2},s_{3} \}$ in Figure \ref{fig:state-aggregation}. A possible choice is $d_{xi} = 1 / |N(x)|$ for all $i$. 

We also define the weight $e_{jx}$ which forms the weights of a basis function that may represent an approximation to the value function. We have $\sum _{x \in X} e_{jx} = 1  $
for $j \in S$.

\begin{figure}[ht]
\centering
\begin{tikzpicture}
    \tikzset{
        state/.style={rectangle, draw, minimum height=3cm, minimum width=3cm, fill=gray!30, text centered, font=\footnotesize},
        substate/.style={circle, draw, minimum size=0.5cm, inner sep=0pt, font=\tiny, anchor=center}
    }

    \node[state] (x1) { \( x_1 \) };
    \node[state, right=1cm of x1] (x2) { \( x_2 \) };
    \node[state, below=1cm of x1] (x3) { \( x_3 \) };
    \node[state, right=1cm of x3] (x4) { \( x_4 \) };

    \begin{scope}[every node/.append style={substate}]
        \node at ([xshift=-0.8cm, yshift=0.8cm]x1.center) (s1) {\( s_1 \)};
        \node at ([xshift=0cm, yshift=0.8cm]x1.center) (s2) {\( s_2 \)};
        \node at ([xshift=0.8cm, yshift=0.8cm]x1.center) (s3) {\( s_3 \)};

        \node at ([xshift=-0.8cm, yshift=0.8cm]x2.center) (s4) {\( s_4 \)};
        \node at ([xshift=0cm, yshift=0.8cm]x2.center) (s5) {\( s_5 \)};
        \node at ([xshift=-0.8cm, yshift=0cm]x2.center) (s6) {\( s_6 \)};

        \node at ([xshift=-0.8cm, yshift=0.8cm]x3.center) (s7) {\( s_7 \)};
        \node at ([xshift=0cm, yshift=0.8cm]x3.center) (s8) {\( s_8 \)};
        \node at ([xshift=0.8cm, yshift=0.8cm]x3.center) (s9) {\( s_9 \)};

        \node at ([xshift=-0.8cm, yshift=0.8cm]x4.center) (s10) {\( s_{10} \)};
        \node at ([xshift=0cm, yshift=0.8cm]x4.center) (s11) {\( s_{11} \)};
        \node at ([xshift=-0.8cm, yshift=0cm]x4.center) (s12) {\( s_{12} \)};
        \node at ([xshift=0.8cm, yshift=0cm]x4.center) (s13) {\( s_{13} \)};
    \end{scope}
\end{tikzpicture}
\caption{Aggregated states diagram illustrating the partitioning of 13 original states into four aggregated states \(x_1\), \(x_2\), \(x_3\), and \(x_4\). Each square represents an aggregated state containing various original states.}
\label{fig:state-aggregation}
\end{figure}
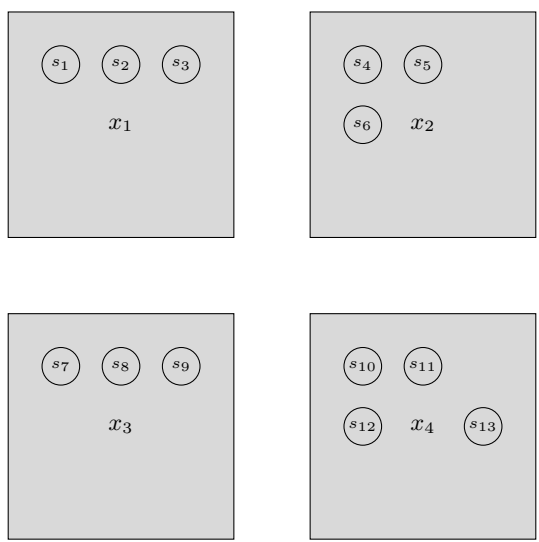

In the state aggregation method we solve the  ``aggregate" Bellman equation which needs to be solved on the state space $X$ instead of the large state space $S$. The aggregated Bellman equation is given by the fixed point of the operator $W$ defined by
$$Wf(x) = \sum _{s_{i} \in N(x) } d_{xs_{i}} \max _{ a \in \Gamma (s_{i})} \sum _{s_{j} \in S} p(s_{i},a,s_{j}) \left( R(s_{i},a,s_{j}) + \beta \sum _{y \in X} e_{s_{j}y} f(y) \right) .$$

Given the results in Lecture 2, it is easy to see that $W$ is a $\beta$-contraction so it has a unique fixed point. Once a fixed point is found $f^{*}$ (it can be found using value function iteration), we approximate the value function by $ \hat{V} = Ef^{*}$ where $E$ is the matrix generated by $e_{jy}$ and then we can find the (sub) optimal policy using this approximated value function. 

As in the previous lectures, let $V$ be the value function of the original problem (see Lecture 2). We have the following guarantee for state aggregation that, informally, shows that $f^{*}$ approximates $V$ well if the aggregation is done for states that are quite similar to each other in terms of their values. 

\begin{theorem} \label{thm:state-aggregation}
    Suppose that $e_{s_{j}y} = 1$ if $s_{j} \in N(y)$.  Let $f^{*}$ be the unique fixed point of $W$. Then we have 
    $$ |V(s_{i}) - f^{*}(x)|  \leq  \frac{\epsilon}{1-\beta}  $$
for all $x \in X$ and $s_{i} \in N(x)$ where 
$$ \epsilon = \max _{x \in X} \max _{s_{i},s_{j} \in N(x) } |V(s_{i}) - V(s_{j})|.$$
\end{theorem}

\begin{proof}
    Let $f(x)= \min _{s_{i} \in N(x)} V(s_{i}) + \epsilon/(1-\beta)$ for $x \in X$. 
    
    Let $x \in X$. We have
\begin{align*}
    Wf(x) & = \sum _{s_{i} \in N(x) } d_{xs_{i}} \max _{ a \in \Gamma (s_{i})} \sum _{s_{j} \in S} p(s_{i},a,s_{j}) \left( R(s_{i},a,s_{j}) + \beta \sum _{y \in X} 1_{ \{s_{j} \in N(y) \} } f(y) \right) \\
    & \leq  \sum _{s_{i} \in N(x) } d_{xs_{i}} \max _{ a \in \Gamma (s_{i})} \sum _{s_{j} \in S} p(s_{i},a,s_{j}) \left( R(s_{i},a,s_{j}) + \beta \left ( V(s_{j}) + \frac{\epsilon } { 1- \beta} \right) \right) \\
    & =  \sum _{s_{i} \in N(x) } d_{xs_{i}} \left ( V(s_{i}) +\frac{\beta \epsilon } { 1- \beta} \right) \\ 
    & \leq \min_{s_{i} \in N(x)} V(s_{i}) + \epsilon + \frac{\beta \epsilon } { 1- \beta} = f(x)
    \end{align*}

    Thus, $Wf \leq f$. Hence, using the Banach fixed point theorem and using the monotonicity of $W$ we have $W^2 f \leq Wf \leq f$ so $f^{*} = \lim _{n \rightarrow \infty } W^{n}f \leq f$. Hence,
    $$f^{*} (x) \leq \min _{s_{i} \in N(x)} V(s_{i}) + \frac{ \epsilon } { 1- \beta} \leq V(s_{i}) + \frac{ \epsilon } { 1- \beta}$$
    for all $x \in X$ and $s_{i} \in N(x)$. The second inequality follows similarly. 
\end{proof}

Alternatively, we can explore identifying the optimal aggregate value function when only aggregate state information is available. This approach proves advantageous when integrating state aggregation with model-free methods such as \(Q\)-learning. The primary benefit of this method is that, with a well-constructed state aggregation and a manageable action space, the tabular \(Q\)-learning algorithm (described below) can be effectively utilized to determine the aggregate value function. 

Assume that $\Gamma (s) = A$ for each $s \in S$. We can define 
$$\overline{p}(x,a,y) = \sum _{s_{i} \in N(x)} d_{x s_{i}} \sum _{s_{j} \in S} p(s_{i},a,s_{j})e_{s_{j}y} \text { and } \overline{R}(x,a) =  \sum _{s_{i} \in N(x)} d_{x s_{i}} \sum _{s_{j} \in S} p(s_{i},a,s_{j}) R(s_{i},a,s_{j}) $$
where $\overline{p}$ and $\overline{R}$ are possibly unknown. We can now solve the Bellman equation in the aggregate space
$$ f^{*}(x) = \max _{a \in A} \overline{R}(x,a) + \beta \sum _{y \in X} \overline{p}(x,a,y) f^{*}(y) $$
to receive an approximation for the value function $V$ as before by $ \hat{V} = Ef^{*}$.

Before delving into the $Q$-learning algorithm we will need to discuss stochastic iterative methods that that underpin algorithms such as stochastic gradient descent for solving optimization problems.

\subsection{Conditional Expectation and Martingales} \label{Section:Martingales}

We will study algorithms of the form 
\begin{equation*}
x_{t+1} = x_{t} + \gamma_{t} Y_{t}
\end{equation*}
where $x_{t}$ is in $\mathbb{R}^{n}$, $\gamma_{t}$ is a non-negative stepsize in $\mathbb{R}$ and $Y_{t}$ is a random variable on $\mathbb{R}^{n}$. This iterative method includes powerful algorithms such as stochastic gradient descent. We will present conditions for the convergence of the algorithm above. We first need to discuss a few tools from probability theory. 
\subsubsection{Conditional Expectations}

Let \( (\Omega, \mathcal{F}, \mathbb{P}) \) be a probability space and let \( X \) be an integrable random variable, i.e., \( \mathbb{E}[|X|] < \infty \). Let \( \mathcal{G} \) be a sub-sigma-algebra of \( \mathcal{F} \). The conditional expectation of \( X \) given \( \mathcal{G} \), denoted \( \mathbb{E}[X | \mathcal{G}] \), is a \( \mathcal{G} \)-measurable random variable satisfying:
\[
\int_G \mathbb{E}[X | \mathcal{G}] \, d\mathbb{P} = \int_G X \, d\mathbb{P}, \quad \forall G \in \mathcal{G}.
\]

It can be shown that the conditional expectation exists\footnote{We will omit throughout the lectures the phrase ``with probability $1$" for most of the analysis.}  and is unique with probability 1.

To get some intuition regarding this definition, denote the inner product space $\langle X,Y \rangle = \mathbb{E}(XY)$ for two random variables $X,Y$ with a finite second moment. 
In this context, two random variables \( X \) and \( Y \) are orthogonal if their inner product is zero, i.e., \( \mathbb{E}[XY] = 0 \).

If \( 1_H \) is an indicator function of a set \( H \in \mathcal{G} \), then by the definition of conditional expectation  \( \mathbb{E}[X \mid \mathcal{G}] \) is chosen such that the ``error" term \( X - \mathbb{E}[X \mid \mathcal{G}] \) is orthogonal to \( 1_H \). That is, 
\[
\mathbb{E}[(X - \mathbb{E}[X \mid \mathcal{G}])1_H] = 0.
\]

This orthogonality property implies that the ``error" \( X - \mathbb{E}[X \mid \mathcal{G}] \) does not correlate with any information contained in \( \mathcal{G} \). This error cannot be further ``explained" by any more information contained solely in \( \mathcal{G} \). In other words, \( \mathbb{E}[X \mid \mathcal{G}] \) has extracted all the information about \( X \) that \( \mathcal{G} \) has to offer.

The conditional expectation of \( X \) given a random variable \( Y \), denoted \( \mathbb{E}[X | Y] \), is defined as the conditional expectation of \( X \) given the sigma-algebra generated by \( Y \), which we denote with abuse of notation by \( \sigma(Y) \). It is not hard to verify that $$\sigma(Y) = \{Y^{-1}(B) \in \mathcal{F}: B \text{ is measurable}\}.$$ 
Thus, we define
\[
\mathbb{E}[X | Y] := \mathbb{E}[X | \sigma(Y)].
\]

For a discrete random variable \( X \) taking values \( x_1, x_2, \ldots, x_n \) and a random variable \( Y \) taking values \( y_1, y_2, \ldots, y_m \), the conditional expectation defined in basic probability theory is given by the function \( h \) defined as:
\[
h(y_{j}) = \sum_{i=1}^n x_i \mathbb{P}(X = x_i | Y = y_j).
\]

This definition agrees with the earlier general definition. The sigma-algebra \( \sigma(Y) \) is generated by the events \( \{Y = y_j\} \). Since \( Y \) takes on a finite number of values, any \( \sigma(Y) \)-measurable function must be constant on each of these sets (why?). Thus, the conditional expectation \( \mathbb{E}[X | \sigma(Y)] \) is constant on each \( \{Y = y_j\} \).

For \( G = \{Y = y_j\} \in \sigma(Y) \),
\[
\int_{\{Y = y_j\}} X \, d\mathbb{P} = \sum_{i=1}^n x_i \mathbb{P}(X = x_i, Y = y_j) = \sum_{i=1}^n x_i \mathbb{P}(X = x_i | Y = y_j) \mathbb{P}(Y = y_j),
\]
where the second equality follows from the basic conditional probability definition for discrete random variables.

Now we can define \( h(y_{j}) = \sum_{i=1}^n x_i \mathbb{P}(X = x_i | Y = y_j) \), and therefore,
\[
\mathbb{E}[X | Y] = \mathbb{E}[X | \sigma(Y)] =  \sum_{j=1}^m h(y_{j}) \mathbf{1}_{\{Y = y_j\}}.
\]
which shows that the conditional expectation \( \mathbb{E}[X | Y] \) matches the definition from basic probability theory for discrete random variables.

Here are some basic properties of conditional expectations (proofs can be found in most advanced probability theory textbooks): 

1. Linearity: \( \mathbb{E}[aX + bY | \mathcal{G}] = a\mathbb{E}[X | \mathcal{G}] + b\mathbb{E}[Y | \mathcal{G}] \) for any constants \( a, b \). 

In addition, $\mathbb{E}(a|\mathcal{G}) = a$ for each scalar $a$. 

2. Independence: If $X$ is independent of $\mathcal{G}$ then $\mathbb{E} (X | \mathcal{G}) = \mathbb{E}(X)$ ($\mathcal{G}$ does not provide any information on $X$).\footnote{In this context independence means that $P(A)P(B) = P(A \cap B)$ for any $A \in \sigma(X)$ and $B \in \mathcal{G}$.} 

3. Taking out what is known: If \( Z \) is \( \mathcal{G} \)-measurable, then \( \mathbb{E}[XZ | \mathcal{G}] = Z\mathbb{E}[X | \mathcal{G}] \). 

In particular, \( \mathbb{E}[Z | \mathcal{G}] = Z \). 

4. Law of total expectation: If \( \mathcal{H} \subseteq \mathcal{G} \subseteq \mathcal{F} \), then \( \mathbb{E}[\mathbb{E}[X | \mathcal{G}] | \mathcal{H}] = \mathbb{E}[X | \mathcal{H}] \).

In particular, $\mathbb{E} [\mathbb{E}[X | Y ] = \mathbb{E}[X]$

\subsubsection{Filtration}
A filtration is a sequence of sigma-algebras that represent the accumulation of information over time. Formally, a filtration \( \{\mathcal{F}_n\}_{n \geq 0} \) is a collection of sigma-algebras such that:
\[
\mathcal{F}_0 \subseteq \mathcal{F}_1 \subseteq \mathcal{F}_2 \subseteq \cdots \subseteq \mathcal{F}.
\]
Each \( \mathcal{F}_n \) represents the information available up to time \( n \).

A filtration can be defined by random variables as the sigma-algebra generated from them. For a sequence of random variables \( \{X_n\}_{n \geq 0} \), the natural filtration \( \{\mathcal{F}_n\} \) is given by:
\[
\mathcal{F}_n = \sigma(X_0, X_1, \ldots, X_n),
\]
which is the smallest sigma-algebra with respect to which \( X_0, X_1, \ldots, X_n \) are measurable. We say that the sequence \( \{X_n\}_{n \geq 0} \) is measurable with respect to a filtration  \( \{\mathcal{F}_n\} \) if $X_{n}$ is $\mathcal{F}_{n}$-measurable for each $n$.  

For example, Suppose that \(X_n\) are Bernoulli coin flips, where \(X_n = 1\) if the coin flip results in heads (with some probability $p$), and \(X_n = 0\) if it results in tails and define 
$ \mathcal{F}_n = \sigma(X_0, X_1, \ldots, X_n)$.

  Then \(\mathcal{F}_0\) is generated by \(X_0\) and contains all possible information outcomes of the first coin flip. It consists of the sets \(\{\emptyset, \Omega, \{X_0 = 0\}, \{X_0 = 1\}\}\), where \(\Omega\) is the sample space of all possible outcomes.  \(\mathcal{F}_n\) includes information from all \(X_k\) for \(k \leq n\).  

For example, a random variable \(Y_n = X_0 + X_1 + \ldots + X_n\) (the total number of heads up to time \(n\)) is measurable with respect to \(\mathcal{F}_n\), as it is a function of the outcomes \(X_0, \ldots, X_n\) but is not measurable with respect to \(\mathcal{F}_{n-1}\).

\subsubsection{Martingales, Sub-Martingales, and Super-Martingales}

Martingales are a powerful tool in probability theory due to their convergence properties and wide range of applications. They provide a framework for modeling many different probabilistic settings and are instrumental in the analysis of stochastic processes. Martingales are extensively used in financial mathematics for modeling asset prices and in the development of pricing algorithms for derivatives. They are also crucial in stochastic approximation methods, where they help in the convergence analysis of iterative algorithms such as stochastic gradient descent. Moreover, martingales play a significant role in areas such as the theory of random walks, queuing theory, Bayesian learning, and other topics of interest. 

\begin{definition}
Let \( \{X_n\}_{n \geq 0} \) be a sequence of random variables that is measurable with respect to a filtration \( \{\mathcal{F}_n\}_{n \geq 0} \). The process \( \{X_n\}_{n \geq 0} \) is called:
\begin{enumerate}
    \item A \textbf{martingale} if \( X_n \) is \( \mathcal{F}_n \)-measurable, \( \mathbb{E}[|X_n|] < \infty \) for each \( n \), and
    \[
    \mathbb{E}[X_{n+1} | \mathcal{F}_n] = X_n \quad \text{for all } n \geq 0.
    \]
    \item A \textbf{sub-martingale} if \( X_n \) is \( \mathcal{F}_n \)-measurable, \( \mathbb{E}[|X_n|] < \infty \) for each \( n \), and
    \[
    \mathbb{E}[X_{n+1} | \mathcal{F}_n] \geq X_n \quad \text{for all } n \geq 0.
    \]
    \item A \textbf{super-martingale} if \( X_n \) is \( \mathcal{F}_n \)-measurable, \( \mathbb{E}[|X_n|] < \infty \) for each \( n \), and
    \[
    \mathbb{E}[X_{n+1} | \mathcal{F}_n] \leq X_n \quad \text{for all } n \geq 0.
    \]
\end{enumerate}
\end{definition}

Martingales have the property that their future expected value, given the present, is equal to their current value, capturing the idea of a ``fair game."

Sub-martingales represent processes where the future expected value, given the present, is at least as large as the current value, indicating a tendency to increase over time.

Super-martingales represent processes where the future expected value, given the present, is at most as large as the current value, indicating a tendency to decrease over time.

\begin{example} \label{Ex:martingale}
Let \( \{S_n\}_{n \geq 0} \) be the sum \( S_n = \sum_{i=0}^n X_i \) and \( X_i \) are independent random variables with $\mathbb{E}[X_{i}] =0$ and $\mathbb{E}[ | X_{i} | ] \leq M$ for some constant $M$ and each $i$. Then \( \{S_n\} \) is a martingale with respect to the natural filtration \( \{\mathcal{F}_n\} \) where \( \mathcal{F}_n = \sigma(X_{0},X_1, \ldots, X_n) \). This follows because:
\[
\mathbb{E}[S_{n+1} | \mathcal{F}_n] = \mathbb{E}[S_n + X_{n+1} | \mathcal{F}_n] = S_n + \mathbb{E}[X_{n+1} | \mathcal{F}_n] = S_n + \mathbb{E}[X_{n+1} ] = S_{n}.
\]
\end{example}

Martingales are extremely useful in proving many results in probability theory, stochastic approximation, and stochastic iterative methods. They enjoy the following convergence results. 

\begin{theorem}[Martingale Convergence Theorem] \label{Thm:Mart}
Let \( \{X_n\}_{n \geq 0} \) be a martingale with respect to a filtration \( \{\mathcal{F}_n\}_{n \geq 0} \) such that \( \sup_{n \geq 0} \mathbb{E}[|X_n|] \leq M \) for some constant $M$. Then there exists random variable \( X_\infty \) with finite expectation such that \( X_n \) converges with probability $1$ to \( X_\infty \). 
\end{theorem}

\begin{theorem}[Supermartingale Convergence Theorem] \label{Thm:SuperMart}
Let \( \{X_n\}_{n \geq 0}, \{Y_n\}_{n \geq 0}, \{Z_n\}_{n \geq 0} \) be a non-negative sequences of random variables that are measurable with respect to a filtration \( \{\mathcal{F}_n\}_{n \geq 0} \).

Suppose that $\sum_{t=0}^{\infty} Z_{t}  < \infty$ and $\mathbb{E}[Y_{t+1}|\mathcal{F}_{t}] \leq Y_{t} - X_{t} + Z_{t}$ for each $t$. Then $\sum _{t=0}^{\infty} X_{t} < \infty$ and the sequence $Y_{t}$ converges with probability $1$ to a non-negative random variable $Y$ with finite expectation. 

\end{theorem}

As a brief demonstration of the convergence results above consider the following examples.

\begin{example} (Strong Law of Large Numbers). \label{Example:LLN}
    Consider the setting of Example \ref{Ex:martingale}. We want to prove that $\lim _{n \rightarrow \infty} S_{n}/n = 0$ with probability $1$  (the strong law of large numbers). From Example \ref{Ex:martingale}, $S_{n}$ is a martingale. From Theorem \ref{Thm:Mart}, $S_{n}$ converges to a finite random variable on some set $B$ that has probability $1$. Hence, $ \lim _{n \rightarrow \infty} S_{n} (\omega) /n = 0$  for each $\omega \in B$, i.e.,  $\lim _{n \rightarrow \infty} S_{n}/n = 0$ with probability $1$.
\end{example}

\begin{example} \label{Ex:product_martingale}
Let \( \{S_n\}_{n \geq 1} \) be the product \( S_n = \prod_{i=1}^n X_i \) where \( X_i \) are independent random variables with \(\mathbb{E}[X_i] = 1\). Then \( \{S_n\} \) is a martingale with respect to the natural filtration \( \{\mathcal{F}_n\} \) where \( \mathcal{F}_n = \sigma(X_1, X_2, \ldots, X_n) \). This follows because:
\[
\mathbb{E}[S_{n+1} | \mathcal{F}_n] = \mathbb{E}\left[\left(\prod_{i=1}^n X_i\right) X_{n+1} \Bigg| \mathcal{F}_n\right] = \left(\prod_{i=1}^n X_i\right) \mathbb{E}[X_{n+1} | \mathcal{F}_n] = S_n \cdot \mathbb{E}[X_{n+1}] = S_n.
\]
Thus, we can use the Martingale Convergence Theorem to show that $S_{n}$ converges. 
\end{example}

\subsection{
Convergence Result for Stochastic Iterative Methods} \label{Section:Convergence}

We are now ready to state the main result regarding stochastic iterative methods.

Consider the sequence 
\begin{equation}\label{eq:Iterative}
x_{t+1} = x_{t} + \gamma_{t} Y_{t},
\end{equation}
where $x_{t}$ is in $\mathbb{R}^{n}$, $\gamma_{t}$ is a non-negative stepsize in $\mathbb{R}$ and $Y_{t}$ is a random variable on $\mathbb{R}^{n}$.  We assume that \( x_{0} \) is known and we denote the natural filtration \( \mathcal{F}_{t} = \sigma (Y_{0}, \ldots, Y_{t}) \). Informally, \( \mathcal{F}_{t} \) contains the values \( x_{0}, \ldots, x_{t+1} \) of the sequence.

For the following result, we let \( \Vert \cdot \Vert \) denote the Euclidean norm defined by \( \Vert x \Vert ^{2} = \sum x_{i}^2 = x'x \), where \( x'x \) is the usual inner product on \( \mathbb{R}^{n} \) and \( x' \) is the transpose of \( x \).  The goal is to provide conditions that imply that the sequence defined in Equation (\ref{eq:Iterative}) converge to a stationary point of $f$.

\begin{theorem} \label{thm:convergence}
Consider the sequence defined in Equation (\ref{eq:Iterative}). Let $f: \mathbb{R}^{n} \rightarrow \mathbb{R}$ be a function that is non-negative, i.e. $f(x) \geq 0$ for all $x \in \mathbb{R}^{n}$. Assume the following four conditions hold: 

(i) Lipschitz:  $f$ is continuously differentiable and $\nabla f$ is Lipschitz continuous, i.e., 
$$ \Vert \nabla f(x) - \nabla f(y) \Vert \leq L \Vert x - y\Vert $$
    for some constant $L > 0$ and all $x,y \in \mathbb{R}^{n}$. 

        (ii) Mean Square: There exist $K_{1},K_{2} > 0$ such that  for all $t$,
    $$\mathbb{E}[\Vert Y_{t} \Vert ^{2} | \mathcal{F}_{t-1} ]\leq K_{1} + K_{2} \Vert \nabla f(x_{t}) \Vert ^{2}.  $$

    (iii) Gradient Direction: There exists $c >0$ such that for all $t$, 
    $$ c\Vert \nabla f(x_{t}) \Vert ^{2} \leq  - \nabla f(x_{t}) '  \mathbb{E}[Y_{t}|\mathcal{F}_{t-1}].  $$

(iv) Stepsize: the stepsizes are non-negative and satisfy 
$$ \sum_{t=0}^{\infty} \gamma_{t} = \infty \text{ and }  \sum _{t=0}^{\infty} \gamma_{t}^{2} < \infty$$

Then the following holds: 

(a) The sequence $f(x_{t})$ converges and we have $\lim _{t \rightarrow \infty } \nabla f(x_{t}) = 0$. 

(b) Every limit point of $x_{t}$, say $x^{*}$  is a stationary point of $f$, i.e., $\nabla f(x^{*}) = 0$. 

In particular, if $f$ has a unique stationary point $x^{*}$ and $x_{t}$ is bounded, then $x_{t}$ converges to $x^{*}$. 

\end{theorem}

The Gradient Direction condition 
ensures that on average, the updates \( Y_{t} \) are making progress in reducing the function \( f \). Intuitively, the term \( \nabla f(x_{t})' \mathbb{E}[Y_{t} | \mathcal{F}_{t-1}] \) represents the slope of $f$ at $x_{t}$ along the expected update direction $\mathbb{E}[Y_{t} | \mathcal{F}_{t-1}]$. When this slope is negative, it indicates that the expected update direction is aligned with the negative gradient, thus pointing towards a direction of descent. The condition ensures that the magnitude of this descent is proportional to the square of the gradient norm, meaning that the larger the gradient norm, the larger the expected decrease in the function \( f \). 

To get further intuition, suppose that $Y_{t} = w_{t}$ is deterministic in each period. Then from Taylor's expansion theorem we have  
$$f(x_{t+1}) = f(x_{t} + \gamma_{t}w_{t}) = f(x_{t}) + \gamma_{t}\nabla f(x_{t})'w_{t} +\overline{R}(x_{t},w_{t},\gamma_{t})$$
where $\overline{R}$ is the remainder which is, intuitively, negligible for small $\gamma$. Hence, intuitively, substantial changes in the slope  of $f$ along the direction $w_{t}$ are critical for the convergence of the iterative method. 

Importantly, the result do not assume convexity, which is typically used to ensure that any local minimum is also a global minimum. Without convexity, \( f \) may have multiple local minima, saddle points, etc. Nevertheless, the conditions ensure that the iterative method will find points that can't be improved locally. 

We will first need the following lemma: 

\begin{lemma} \label{Lemma:Dominate}
    Suppose that condition (i) of Theorem \ref{thm:convergence} holds. Then for all $x,y \in \mathbb{R}^{n}$ we have  
    \begin{equation*} \label{Eq:Lip1}
         f(y) \leq f(x) +\nabla f(x) ' (y-x) + \frac{L}{2} \Vert y-x \Vert^{2}
          \end{equation*}
          \end{lemma}

       \begin{proof}
          Let $x,z$ be two elements in $\mathbb{R}^{n}$.  Define $g(t) = f(x+tz)$ where $t$ is a scalar. We have 
$$ f(x+z) - f(x) = g(1) - g(0) = \int_{0}^{1} \frac{d g(t) }{dt} dt = \int _{0}^{1} z' \nabla f(x+tz) dt$$
    where the last equality follows from the chain rule.  Now note that 
    \begin{align*}
        \int _{0}^{1} z' \nabla f(x+tz) dt  & =    \int _{0}^{1} z' \nabla f(x) dt +   \int _{0}^{1} z' \nabla f(x+tz) dt -    \int _{0}^{1} z' \nabla f(x) dt   \\
        & \leq  z' \nabla f(x)  + \int _{0}^{1} \Vert z \Vert \cdot \Vert \nabla f(x+tz)  - \nabla f(x) \Vert dt \\
        & \leq  z' \nabla f(x) +\Vert z \Vert  \int _{0}^{1} Lt\Vert z \Vert dt = z' \nabla f(x) +\frac{L}{2} \Vert z \Vert ^{2}
         \end{align*}
where in the first inequality we used the Cauchy–Schwarz inequality. Now let $z=y-x$ to conclude the proof the Lemma. 
       \end{proof}

\begin{proof}[Proof of Theorem \ref{thm:convergence}]
    We  only prove that $f(x_{t})$ converges and $\liminf_{t \rightarrow \infty} \Vert \nabla f(x_{t} ) \Vert = 0$. The rest of the proof can be found, for example, in \cite{bertsekas1996neuro}. We have
 \begin{align*}
     \mathbb{E}[f(x_{t+1}) | \mathcal{F}_{t-1}  ] & \leq \mathbb{E}[f(x_{t}) | \mathcal{F}_{t-1}  ] + \gamma_{t} \mathbb{E}[ \nabla f(x_{t}) ' Y_{t}  | \mathcal{F}_{t-1} ]  + \frac{L}{2}  \gamma_{t}^{2} \mathbb{E}[\Vert Y_{t} \Vert^{2} | \mathcal{F}_{t-1}  ]  \\
     & = f(x_{t}) + \gamma_{t} \nabla f(x_{t}) ' \mathbb{E}[ Y_{t}  | \mathcal{F}_{t-1} ]  + \frac{L}{2}  \gamma_{t}^{2} \mathbb{E}[\Vert Y_{t} \Vert^{2} | \mathcal{F}_{t-1}  ] \\
     & \leq f(x_{t}) + \gamma_{t} \nabla f(x_{t}) ' \mathbb{E}[ Y_{t}  | \mathcal{F}_{t-1} ]  + \frac{L}{2}  \gamma_{t}^{2} \left (K_{1} + K_{2} \Vert \nabla f(x_{t}) \Vert ^{2} \right ) \\
     & \leq  f(x_{t}) - \gamma_{t} \left ( c - \frac{\gamma_{t}LK_{2}}{2} \right ) \Vert \nabla f(x_{t}) \Vert ^{2} + \frac{LK_{1}\gamma_{t}^{2}}{2}.
 \end{align*}
     The first inequality follows from Lemma \ref{Lemma:Dominate} with $y=x_{t+1}$ and $x = x_{t}$ and conditioning on $\mathcal{F}_{t-1}$. The equality follows because $x_{t}$ is measurable with respect to $\mathcal{F}_{t-1}$. The second and third inequalities follows from conditions (ii) and (iii) of the theorem. 

     Now define for each $t$,
     $$ X_{t-1} =  \gamma_{t} \left ( c - \frac{\gamma_{t}LK_{2}}{2} \right ) \Vert \nabla f(x_{t}) \Vert ^{2}1_{ \{ \gamma_{t}LK_{2} \leq 2c \} } \text{,    } W_{t-1} = f(x_{t})$$
     and
     $$Z_{t-1} = \frac{LK_{1}\gamma_{t}^{2}}{2}1_{ \{ \gamma_{t}LK_{2} \leq 2c \} } + \left (  - \gamma_{t} \left ( c - \frac{\gamma_{t}LK_{2}}{2} \right ) \Vert \nabla f(x_{t}) \Vert ^{2} + \frac{LK_{1}\gamma_{t}^{2}}{2} \right )   1_{ \{ \gamma_{t}LK_{2} > 2c \} }  $$
     and note that $W_{t-1},X_{t-1},Z_{t-1}$ are measurable with respect to $\mathcal{F}_{t-1}$ because $x_{t}$ is measurable with respect to $\mathcal{F}_{t-1}$. In addition, $W_{t}, Z_{t}$ and $X_{t}$ are non-negative and in the inequalities above we have shown that 
     $$\mathbb{E}[W_{t}|\mathcal{F}_{t-1}] \leq W_{t-1} - X_{t-1} + Z_{t-1}$$ 
     for each $t$ and  $\sum_{t=0}^{\infty} Z_{t}  < \infty$ follows easily from the assumption that $\sum _{t=0}^{\infty} \gamma_{t}^{2} < \infty $. 

Thus, we can use Theorem \ref{Thm:SuperMart} to conclude that $W_{t-1}=f(x_{t})$ converges with probability $1$ to a non-negative random variable and $\sum_{t=0}^{\infty} X_{t} < \infty $. 
Because $\gamma_{t}$ converges to $0$ (as $\sum _{t=0} ^{\infty} \gamma_{t}^{2} < \infty $) we have $LK_{2} \gamma_{t} \leq c$ for any $t \geq t_{0}$ for some finite $t_{0}$. Hence, 
\begin{equation} \label{Eq:SumXt}
    \infty > \sum _{t=0}^{\infty}  \gamma_{t} \Vert \nabla f(x_{t}) \Vert ^{2}.  
    \end{equation} 

Now, if there exists some time $t_{0}$ and $\delta > 0$ such that $\Vert \nabla f(x_{t}) \Vert ^{2} \geq \delta $ for all $t \geq t_{0}$ then the fact that $\sum _{t=0}^{\infty} \gamma_{t} = \infty$ leads to a contradiction given Inequality (\ref{Eq:SumXt}). 

We conclude that for every $t_{0}$, and $\delta > 0$, there exists $t(t_{0}) \geq t_{0} $ such that $\Vert \nabla f(x_{t(t_{0})}) \Vert  < \delta $, i.e., $\liminf _{t \rightarrow \infty} \Vert \nabla f(x_{t}) \Vert =0  $ 
\end{proof}

\begin{remark}
    The proof of the last theorem also follows in the case that $K_{1}$ is a bounded (with probability $1$) random variable that is measurable with respect to $\mathcal{F}_{t-1}$. 
\end{remark}

\subsubsection{Stochastic Gradient Descent Variants} \label{Sec:SGD}

We present a few important variants of stochastic gradient descent algorithms that are a special case of Equation (\ref{eq:Iterative}). 

We will assume for the rest of this section that the stepsizes are non-negative and satisfy 
$$ \sum_{t=0}^{\infty} \gamma_{t} = \infty \text{ and }  \sum _{t=0}^{\infty} \gamma_{t}^{2} < \infty.$$

\textbf{Noisy Gradient Descent:} Consider the following noisy gradient descent algorithm 
\begin{equation} \label{Eq:NoisyGrad}
    x_{t+1} = x_{t} - \gamma_{t} ( \nabla f(x_{t}) + Z_{t}). 
\end{equation}
which corresponds to Equation (\ref{eq:Iterative}) with $Y_{t} = - \nabla f(x_{t}) - Z_{t}$. 
We assume that $f$ is non-negative, has a Lipschitz continuous gradient (see Assumption (i) in Theorem \ref{thm:convergence}), 
$$\mathbb{E} [Z_{t} | F_{t-1} ] = 0 \text { and } \mathbb{E} [ \Vert Z_{t}  \Vert ^{2}| F_{t-1} ] \leq K_{1}' + K_{2}' \Vert \nabla f(x_{t})  \Vert ^{2}$$
for some constants $K_{1}', K_{2}'$.

Note that 
$$ - \nabla f(x_{t})' \mathbb{E} [Y_{t} | F_{t-1} ] = \nabla f(x_{t})' (\nabla f(x_{t}) +  \mathbb{E} [Z_{t} | F_{t-1} ] ) = \Vert \nabla f(x_{t})  \Vert ^{2} $$
so Assumption (iii) of Theorem \ref{thm:convergence} holds with $c=1$. 

In addition, given the assumption on $Z_{t}$ we have
  $$\mathbb{E}[\Vert Y_{t} \Vert ^{2} | \mathcal{F}_{t-1} ] \leq  K_{1}' + (K_{2}' +1)\Vert \nabla f(x_{t}) \Vert ^{2} $$
so Assumption (ii) of Theorem \ref{thm:convergence} holds.

Thus, we can apply Theorem \ref{thm:convergence} to the noisy gradient problem to conclude that $f(x_{t})$ converges and $\nabla f(x_{t})$ converges to $0$. 

\textbf{Stochastic Approximation:} \label{SA}
Now assume that we have a sequence of I.I.D random variables $V_{t}$ with a law $V$ and a finite mean $\mu$ that is unknown. We consider the stochastic approximation algorithm 
$$ x_{t+1} = (1-\gamma_{t})x_{t} + \gamma_{t} V_{t}$$
so we use a single sample at each time from the law $V$ and we hope that $x_{t}$ converges to the unknown mean $\mu$. 

Define $-Z_{t} = V_{t} - \mu$ and note that we get 
$$ x_{t+1} = x_{t} - \gamma_{t} (x_{t} - \mu) - \gamma _{t} Z_{t}$$
so stochastic approximation is the noisy gradient descent with $f(x) = 0.5 \Vert x-\mu \Vert ^{2} $. In particular, Theorem \ref{thm:convergence} applies under the conditions on $Z_{t}$ assumed in the noisy gradient part.  

As a special case that will be used in the analysis of the Q-learning algorithm, consider $f(x)=x^{2}/2$ on  $\mathbb{R}$ so the derivative of $f$ is exactly $x$. So if $\mathbb{E} [Z_{t} | F_{t-1} ] = 0 \text { and } \mathbb{E} [  Z_{t}  ^{2}| F_{t-1} ] \leq K_{1}' + K_{2}' x_{t}^{2}$ for some $K_{1}' , K_{2}'  >0$, the sequence $x_{t}$ converges to $0$ from the noisy gradient algorithm analysis.

\textbf{Stochastic Gradient Descent (SGD):} 
Stochastic Gradient Descent (SGD) algorithm has become a fundamental algorithm for training large-scale models. Unlike traditional gradient descent, which computes the gradient using the entire dataset to update the model parameters once per iteration, SGD uses only a single or a few training examples at a time. This approach significantly reduces the computational burden, making it feasible to train on large datasets and can also introduce a beneficial level of noise into the gradient estimates, which helps to avoid local minima. 

The popularity of SGD has surged in recent years due to its effectiveness in handling massive datasets. It is particularly favored in the training of deep neural networks, where the size of the datasets and the complexity of the models can make traditional optimization methods computationally impractical.

The typical setting consists of a differentiable in $\theta$ loss function \( L(k_\theta(x_{i}), y_i) \) where \( \{(x_i, y_i)\}_{i=1}^N \) represent training data instances and $\theta$ are parameters (e.g., weights of a neural network). Our objective is to minimize the function \( f \) defined by

\[ 
f(\theta) = \frac{1}{N} \sum_{i=1}^{N} L(k_\theta(x_i), y_i).
\]

A common approach would be to use deterministic gradient descent, updating the parameters \( \theta \) according to the formula:

\[ 
\theta_{t+1} = \theta_t - \gamma_t \frac{1}{N} \sum_{i=1}^{N} \nabla_\theta L(k_\theta(x_i), y_i).
\]

However, this approach can be computationally expensive because it requires processing the entire dataset at each iteration. As discussed above, an alternative is to update the parameters based on a single randomly selected data point per iteration, significantly reducing the computational burden.

Formally, let \( U(t) \) be a sequence of i.i.d. random variables uniformly distributed over \( \{1, \ldots, N\} \), representing the indices of the data points. The SGD update at each step is then given by:

\[
\theta_{t+1} = \theta_t - \gamma_t \nabla_\theta L(k_\theta(x_{U(t)}), y_{U(t)}),
\]

where \( \gamma_t \) is the learning rate at step \( t \). This stochastic update reduces computational costs by approximating the true gradient of \( f \) based on a single observation per iteration.

We note that we can write:

\[
\theta_{t+1} = \theta_t - \gamma_t \nabla f(\theta_{t})  - \gamma_t Z_{t},
\]

with 

\[
Z_{t} = \nabla_\theta L(k_\theta(x_{U(t)}), y_{U(t)}) -  \nabla f(\theta_{t}),
\]

so we can use the noisy gradient descent analysis to study SGD. Note that $\mathbb{E}[Z_{t} | \mathcal{F}_{t-1} ] = 0$ because $U(t)$ is uniformly distributed  
and it can be shown that the mean squared condition is satisfied if
\[
\max _{i} \Vert \nabla_\theta L(k_\theta(x_i), y_i) \Vert ^{2} \leq K_{1}' + K_{2}' \Vert \nabla f(\theta) \Vert ^{2}
\]

for all \(\theta\) and some constants \( K_{1}',K_{2}' \) (see exercises).

\subsection{Q-Learning} \label{Section:Q-learning}

 For simplicity we assume that $\mathcal{A} = \Gamma(s)$ for all state $s \in S$ where $\mathcal{A}$ is the set of probability measures on some finite set of actions $A = \{a_{1},\ldots,a_{k} \}$ which we denote by $A$ with some abuse of notation (all the results in this section will also hold for the case that the feasible actions at state $s$ are given by $\Gamma(s)$ as in Lecture 2). 

The general idea of Q-learning is to mimic value function iteration using simulation. 

Q-learning is a ``model-free" algorithm in the sense that it can be used whenever there is explicit model of the system or the payoffs. The general idea of Q-learning is to mimic value function iteration algorithm using simulation. The algorithm operates through direct interactions with the environment and learns from the experience derived from the enviorment.

The core of Q-learning is the Q-function or Q-values, denoted as \( Q(s, a) \). The objective of the Q-learning algorithm is to determine the Q-values. 

In this section we will introduce the $Q$-learning algorithm and provide a convergence result for the tabular case (without function approximation for the Q-function).

 Consider the operator $H:B(S \times A) \rightarrow B(S \times A)$ by
$$HQ(s,a) = \sum _{s' \in S} p(s,a,s') \left ( R(s,a,s') + \beta \max_{a' \in A} Q(s' , a') \right )$$

\begin{lemma} \label{Lem:H is contraction}
The operator $H$ is a $\beta$-contraction when $B(S \times A)$ is endowed with maximum metric $d$. 
\end{lemma}

\begin{proof}
    Let $Q$ and $W$ be two functions in $B(S \times A)$. Then 
    \begin{align*}
        | HQ(s,a) - HW(s,a) | & =    \beta \left | \sum _{s' \in S}p(s,a,s')\left ( \max_{a' \in A} Q(s' , a') - \max_{a' \in A} W(s' , a')\right ) \right | \\
        & \leq  \beta  \sum _{s' \in S}p(s,a,s') \left | \max_{a' \in A} Q(s' , a') - \max_{a' \in A} W(s' , a') \right | \\
          & \leq  \beta  \sum _{s' \in S}p(s,a,s') \max_{a' \in A} \left |  Q(s' , a') -  W(s' , a') \right | \\
           & \leq  \beta  \sum _{s' \in S}p(s,a,s') d(Q,W)  =  \beta d(Q,W). 
    \end{align*}
Taking the maximum over $(s,a)$ on the left-hand-side yields the result.

In the second inequality we use $| \max x_{i} - \max y_{i} | \leq \max_{i} |x_{i} - y_{i}| $ for two vectors $x,y$. 
\end{proof}

Hence, from the Banach fixed point theorem $H$ has a unique fixed point $Q^{*}$ (see Lecture 1).

It is easy to see that the value function is given by $V(s) = \max _{a \in A} Q^{*}(s,a)$ by noticing that $Tf (s) =  \max _{a \in A} HQ (s,a)$ where $f(s) = \max _{a \in A} Q(s,a)$. 
Thus, we can apply value function iteration for $H$ to obtain the value function. Working directly with $H$  function rather than $T$ increases the state space but simplifies the implementation of value function iteration. Specifically, the advantage of using $Q$ functions lies in the encapsulation of the action selection within the expectation calculation, i.e., directly choosing the action that maximizes the future reward as part of the expectation calculation itself.

A more general version of this value function iteration algorithm is to consider 
$$ Q(s,a) \leftarrow (1-\gamma) Q(s,a) +  \gamma \sum _{s' \in S} p(s,a,s') \left ( R(s,a,s') + \beta \max_{\Tilde{a} \in A} Q(s' , \Tilde{a}) \right )  $$
for some step size parameter $\gamma \in (0,1]$ that may change overtime. The $Q$-learning algorithm is an approximation of this method using simulation where we replace the unknown transition probability function by a single observed sample, i.e., 
$$ Q(s,a) \leftarrow (1-\gamma) Q(s,a) +  \gamma  \left ( R(s,a,s') + \beta \max_{\Tilde{a} \in A} Q(s' , \Tilde{a}) \right )  $$
where the next period's state $s'$ and the payoff $R$ are generated from $(s,a)$ by simulating the enviorment. That is, to enable $Q$-learning it is assumed that we have an access to a simulator where for each policy and each state, we can obtain the next period's state using $p$ and the payoff $R$ at that period (without knowing $p$ and $R$). 

More precisely, let $t$ be an index where we update our $Q$-functions. Then the $Q$-learning updates according to 
\begin{equation} \label{Eq:Q}
     Q_{t+1} (s,a) = (1- \gamma_{t}(s,a) ) Q_{t}(s,a) + \gamma_{t}(s,a) \left (R(s,a,s')  + \beta \max_{\Tilde{a} \in A} Q_{t}(s' , \Tilde{a})\right )
     \end{equation}
where $s'$ is sampled randomly  according to $p(s,a,s')$ and $\gamma_{t}(s,a) =0$ for $(s,a)$ that are not updated. For simplicity, we assume that we update only $(s_{t},a_{t})$ (the state-action pair that was observed in period $t$) and do not update the $Q$ values of the other state-action pairs (this is called asynchronous $Q$-learning). 

\textbf{Remark:} In Q-learning, the selection of actions \(a_t\) and the updating of Q-values represent two distinct aspects of the learning process that are designed to balance exploration with exploitation.

\begin{itemize}
    \item \textbf{Action Selection Policy (exploration):} Actions in Q-learning are selected according to a policy that allows for exploration of the state-action space. Note $a_{t}$ is not chosen by the $Q$-learning algorithm so we will need to provide a policy ($Q$-learning is called an off-policy algorithm). \\ 
    The policy is crucial as actions need to be explored sufficiently over time. This is critical for the convergence of Q-learning algorithm as we will show in this section. 
 \item \textbf{Q-value Update Rule (exploitation):} The update rule for Q-values in Q-learning is inherently greedy. 
       \(\max_{\Tilde{a}} Q(s_{t+1}, \Tilde{a})\) represents the highest Q-value achievable in the next state \(s_{t+1}\), reflecting the greedy aspect of the update.

    \end{itemize}

The main result of this section shows that if the algorithm visits every state-action pair infinitely often, then the $Q$-learning converges to the optimal $Q^{*}$ function. 

\begin{theorem}
    Suppose that 

    $$ \sum _{t=0}^{\infty} \gamma_{t} (s,a) = \infty \text { and } \sum _{t=0}^{\infty} \gamma_{t} ^{2} (s,a) < \infty$$
    then $Q_{t}$ derived from the Q-learning described in Equation (\ref{Eq:Q}) converges with probability $1$ to $Q^{*}$.  
\end{theorem}

\begin{proof} For the proof we will use $\Vert \Vert$ for the maximum norm.  Without loss of generality we assume that $Q^{*} =0$. We omit the phrase with probability 1 throughout the proof.

    Let $\mathcal{F}_{t}$ be the sigma algebra generated from the history up to time $t$  (see Lecture 2). 
 Let us write 
 \begin{equation} 
     Q_{t+1} (s,a) = (1- \gamma_{t}(s,a) ) Q_{t}(s,a) + \gamma_{t}(s,a) \left (HQ_{t}(s,a)  + Z_{t+1} (s,a) \right )
     \end{equation}
     where 
     $$ Z_{t+1} (s,a) = R(s,a,s') + \beta \max_{\Tilde{a} \in A} Q_{t}(s' ,\Tilde{a}) -  HQ_{t}(s,a) $$

     Note that $\mathbb{E} [Z_{t+1} (s,a) | \mathcal{F}_{t} ] =0$ and $\mathbb{E} [Z_{t+1}^{2} (s,a) | \mathcal{F}_{t} ] \leq K_{1} + K_{2} \Vert Q_{t} \Vert^{2} $  for some constants $K_{1},K_{2}$.

     We are interested in proving that 
     $$Q_{t+1}(i) = (1- \gamma_{t}(i) )Q_{t}(i) + \gamma_{t}(i) (HQ_{t}(i) + Z_{t+1} (i) )$$
converges to the fixed point of $H$ given by $Q^{*} = 0$ with probability $1$ where we denote $(s,a) = i$ for simplicity.. 

Let $\epsilon$ be such that $\beta (1+\epsilon) =1 $. Let $G_{0} = \max \{ \Vert Q_{0} \Vert  , 1 \}$ and define recursively 
\begin{equation*}
    G_{t+1} =
    \begin{cases}
        G_{t} & \text{if } \Vert Q_{t+1} \Vert \leq (1+\epsilon) G_{t} \\
        G_{0}(1+\epsilon)^{k} & \text{if } \Vert Q_{t+1} \Vert > (1+\epsilon) G_{t}
    \end{cases}
\end{equation*} where $k$ is chosen such that $(1+\epsilon)^{k-1} G_{0} <\Vert Q_{t+1} \Vert   \leq (1+\epsilon)^{k} G_{0} = G_{t+1} $. 

Note that $G_{t}$ is a sequence of non-decreasing random variables that are measurable with respect to $\mathcal{F}_{t}$ such that  $\Vert Q_{t} \Vert   \leq (1+\epsilon) G_{t} $ and $\Vert Q_{t} \Vert   \leq  G_{t} $ if $G_{t-1} < G_{t}$.  

We have
\begin{align} \label{Eq:Q_proof_Gt}
  \Vert HQ_{t} \Vert   \leq \beta \Vert Q_{t} \Vert \leq  \beta (1+\epsilon) G_{t} = G_{t}
\end{align}
where the first inequality follows because $H$ is contraction (see Lemma \ref{Lem:H is contraction}) and using $HQ^{*}=Q^{*} = 0$ and the equality from the definition of $\epsilon$.

We are now rescaling $Z_{t}(i)$ be defining $Y_{t+1}(i) = Z_{t+1}(i)/G_{t}$. This ensures that $\mathbb{E} [Y_{t+1} (i) | F_{t} ] =0$ and $\mathbb{E} [Y_{t+1}^{2} (i) | F_{t} ] \leq K $ for some constant $K$ so we can apply Theorem \ref{thm:convergence} in Step 1 below.  

     We proceed by proving the following two steps:

     \textbf{Step 1}. (i) For all $i$, $t_{0} \geq 0$, define  $W_{t_{0}; t_{0} } = 0$ and
     $$ W_{t+1;t_{0}} (i) = (1-\gamma_{t}(i)) W_{t;t_{0}} (i) + \gamma_{t}(i) Y_{t+1} (i) , \text{ } t \geq t_{0} $$

     Then $\lim _{t\rightarrow \infty} W_{t; t_{0}}(i) = 0$.

(ii) For every $\delta >0 $ there exists some $t_{0}$ such that $|W_{t;t_{0}}| \leq \delta$  for all $t \geq t_{0}$
 
     \textbf{Proof of Step 1.}
(i) From the SGD analysis of 
(see the stochastic approximation paragraph in Section \ref{Sec:SGD}) $W_{t;0}$ is a noisy gradient descent on $\mathbb{R}$ with the function $k(x) = x^{2} /2$ and the conditions of the convergence result (see Theorem \ref{thm:convergence}) hold which implies that  $\lim _{t\rightarrow \infty} W_{t;0}(i) = 0$. 
The case $\lim _{t\rightarrow \infty} W_{t;t_{0}}(i) = 0$ follows from a similar argument.

(ii) Note that for all $t \geq t_{0}$ we have 
$W_{t;0}(i) - \prod _{\tau = t_{0} } ^{t-1} (1-\gamma_{\tau}(i))W_{t_{0};0}(i)= W_{t;t_{0} }$. Now use part (i) to conclude the result. 

\textbf{Step 2.} The sequence $Q_{t}$ is bounded. 

\textbf{Proof of Step 2.} Assume in contradiction that the sequence $Q_{t}$ is unbounded. This means that $G_{t}$ is unbounded. The fact that  $\Vert Q_{t} \Vert   \leq  G_{t} $ if $G_{t-1} < G_{t}$ implies that $\Vert Q_{t} \Vert   \leq  G_{t} $ holds for infinitely many values of $t$. 

Using Step 1, and the fact that $\gamma_{t}(i)$ converges to $0$ we conclude that there exists $t_{0}$ such that $\Vert Q_{t_{0} } \Vert   \leq  G_{t_{0}}$, $|W_{t;t_{0} } | \leq \epsilon$ and $\gamma_{t}(i) \leq 1$ for all $t \geq t_{0}$. 

We now claim that $G_{t} = G_{t_{0}} $ and $$-G_{t_{0}} + G_{t_{0}}W_{t;t_{0} }(i) \leq Q_{t } (i)  \leq G_{t_{0}} + G_{t_{0}}W_{t;t_{0} }(i)$$ for all $t \geq t_{0}$ which is a contradiction to  the assumption that $G_{t}$ is unbounded, and hence, it shows that $Q_{t}$ is bounded.

The proof is by induction. The case $t=t_{0}$ is immediate from $W_{t_{0} ; t_{0} } =0$ and $\Vert Q_{t_{0} } \Vert   \leq  G_{t_{0}}$. Assume that it holds for some $t \geq t_{0}$. 

We have 
\begin{align*}
    Q_{t+1}(i) & = (1- \gamma_{t}(i) )Q_{t}(i) + \gamma_{t}(i) (HQ_{t}(i) + Z_{t+1} (i) ) \\
   & \leq (1- \gamma_{t}(i) ) ( G_{t_{0}} + G_{t_{0}}W_{t;t_{0} }(i) ) + \gamma_{t}(i) HQ_{t}(i) +  \gamma_{t}(i) G_{t_{0}} Y_{t+1} (i) \\
   & \leq (1- \gamma_{t}(i) )  ( G_{t_{0}} + G_{t_{0}}W_{t;t_{0} }(i) ) +  \gamma_{t}(i)G_{t_{0}}  + \gamma_{t}(i) G_{t_{0}} Y_{t+1} (i) \\
   & = G_{t_{0}} +W_{t+1;t_{0} }(i) G_{t_{0}}.  
\end{align*}
The first inequality follows from the induction hypothesis. The second inequality follows from inequality (\ref{Eq:Q_proof_Gt}). A symmetrical argument yields $-G_{t_{0}} +W_{t+1;t_{0} }(i) G_{t_{0}} \leq Q_{t+1}(i)$. Using $|W_{t;t_{0} } | \leq \epsilon$ we have $|Q_{t+1}(i)| \leq G_{t_{0}} (1+\epsilon)$ which implies that $G_{t+1}=G_{t} = G_{t_{0}}$ by the definition of the sequence $G_{t}$ and the induction hypothesis. This completes the induction and proof of Step 2.

\textbf{Step 3}. The sequence $Q_{t}$ converges to $0$ with probability $1$. 

\textbf{Proof of Step 3.} Step 2 shows that $Q_{t}$ is bounded, say by $D_{0}$ (that can be random). Thus $\Vert Q_{t} \Vert \leq D_{0}$ for each $t$. 

Now suppose that $\Vert Q_{t} \Vert \leq D_{k} $ for some random $D_{k}$ and all $t \geq t_{k}$ such that $\gamma_{t}(i) \leq 1$. We want to show that $\Vert Q_{t} \Vert \leq \beta D_{k+1} $ for all $t \geq t_{k+1} $ for some $t_{k+1} \geq t_{k}$.  

Define $V_{t_{k}} (i) = D_{k} $ and consider the sequence 
$$V_{t+1}(i) = (1-\gamma_{t}(i)) V_{t}(i) + \gamma_{t} (i) \beta D_{k} \text{ for all } t\geq t_{k}. $$

 We now claim that $$-V_{t}(i) + \overline{W}_{t;t_{k}} (i) \leq Q_{t} (i) \leq V_{t}(i) + \overline{W}_{t;t_{k}} (i) $$ where as in Step 1 we define
  for all $i$, $t_{0} \geq 0$, define  $\overline{W}_{t_{0}; t_{0} } = 0$ and
     $$\overline{W}_{t+1;t_{0}} (i) = (1-\gamma_{t}(i)) \overline{W}_{t;t_{0}} (i) + \gamma_{t}(i) Z_{t+1} (i) , \text{ } t \geq t_{0}. $$
 The proof is by induction and is similar to Step 2. For $t= t_{k}$ it is immediate by the fact that $\overline{W}_{t_{k} ; t_{k} } = 0$. Now assume it holds for some $t \geq t_{k}$. Then from inequality (\ref{Eq:Q_proof_Gt}) we have $HQ_{t}(i) \leq \beta \Vert Q_{t} \Vert \leq \beta D_{k}$. Thus,
\begin{align*}
    Q_{t+1}(i) & = (1- \gamma_{t}(i) )Q_{t}(i) + \gamma_{t}(i) (HQ_{t}(i) + Z_{t+1} (i) ) \\
   & \leq (1- \gamma_{t}(i) ) ( V_{t}(i) + \overline{W}_{t;t_{k}} (i)  ) + \gamma_{t}(i) \beta D_{k} +  \gamma_{t}(i)  Z_{t+1} (i) \\
   & = V_{t+1}(i) + \overline{W}_{t+1;t_{k}} (i)
\end{align*}
and the other inequality follows from a symmetrical argument. 

 $V_{t}$ converges to $\beta D_{k} $ (this follows from the analysis of Step 1 part (i) or can be shown directly). From Step 2, $Q_{t}$ is bounded, and hence, $\lim _{t \rightarrow \infty} \overline{W}_{t;t_{k}} = 0$ as the mean square condition holds (see the proof of Step 1 and recall that $\mathbb{E} [Z_{t+1}^{2} (s,a) | \mathcal{F}_{t} ] \leq K_{1} + K_{2} \Vert Q_{t} \Vert^{2} $). Hence, using  $-V_{t}(i) + \overline{W}_{t;t_{k}} (i) \leq Q_{t} (i) \leq V_{t}(i) + \overline{W}_{t;t_{k}} (i) $ we conclude that $\limsup_{t\rightarrow \infty} \Vert Q_{t} \Vert \leq \beta D_{k}$. Therefore, there is $t_{k+1}$ such that $\Vert Q_{t} \Vert \leq \beta D_{k}$ for all $t \geq t_{k+1}$. 

Hence, we can generate an increasing sequence, $t_{k},\ldots,t_{k+m},\ldots$ such that $\Vert Q_{t} \Vert \leq \beta ^{m} D_{k}$ for all $t \geq t_{k+m}$ which means that $Q_{t}$ converges to $0$ as claimed. 
\end{proof}

\subsubsection{Exploration}
There are various policies for balancing exploration and exploitation, which are crucial for the effectiveness of Q-learning. We now present several basic popular exploration policies. 

\textbf{Epsilon-Greedy Policy}.
The \(\epsilon\)-greedy strategy is straightforward and widely used. In this policy the agent chooses the best known action with probability \(1 - \epsilon\) and a random action (typically uniformly distributed on the action set) with probability \(\epsilon\), facilitating both exploration and exploitation, i.e., 
\[
a_t = 
\begin{cases} 
\arg\max_a Q(s_t, a) & \text{with probability } 1 - \epsilon_{t}, \\
\text{random action} & \text{with probability } \epsilon_{t}.
\end{cases}
\]

\textbf{Softmax Policy}.
The Softmax or Boltzmann exploration policy selects actions based on the relative value of their action-value functions, using a Boltzmann distribution. That is, 
\[
\text{Pr}(a_t = a | s_t = s) = \frac{\exp \left ( Q(s, a)/\tau  \right ) }{\sum_{a' \in A} \exp \left ( Q(s, a')/\tau \right ) },
\]
where \(\tau\) denotes the temperature parameter that moderates the exploration level.

\textbf{Upper Confidence Bound (UCB)}. 
The UCB policy selects actions based on both their average reward and the uncertainty. One option is to choose actions according to 
\[ a_t = \arg\max_a \left( Q(s_t, a) + c \sqrt{\frac{\log t}{N_t(s_t, a)}} \right),
\]
where \(N_t(s_t, a)\) is the number of times action \(a\) has been selected in state \(s_t\) up to time \(t\), and \(c\) is a tunable parameter that controls the level of exploration.

In practice, ensuring sufficient exploration is challenging. It's essential for the $Q$-learning algorithm to explore well to succeed in practical problems. There are many approaches for enhanced exploration and this is a topic of interest.

\subsubsection{SARSA: On-Policy Q-learning Style Algorithm}

There are many variants of Q-learning. A popular one is SARSA (State-Action-Reward-State-Action). 
In contrast to Q-learning, which is an off-policy method, SARSA is an on-policy algorithm that updates its Q-values using the actions taken by the policy it is currently learning, rather than the greedy policy. This approach integrates the exploration strategy directly into the policy's evaluation and is inherently more conservative, particularly in environments where certain actions might lead to significantly negative consequences. Thus, if we care for some reasons about the rewards during the simulation, then SARSA looks like a more suitable option than Q-learning. 

SARSA updates its Q-values based on the equation:
\[
Q_{t+1}(s, a) = (1 - \gamma_t(s, a)) Q_t(s, a) + \gamma_t(s, a) \left( R(s, a, s') + \beta Q_t(s', a') \right),
\]
so it uses \(Q_t(s', a')\) where \(a'\) is the actual action taken as opposed to the Q-learning, which includes the term \(\max_{a'} Q(s', a')\) in its update rule. 

It can be shown that SARSA converges to the optimal $Q$-functions as we shown above for $Q$-learning if the policy converges to a greedy policy.

\subsubsection{Q-Learning with Function Approximation} \label{Sec:Qapprox}

When dealing with large state spaces or continuous state spaces, directly learning a Q-value for every state-action pair becomes computationally infeasible. In this case, we need to represent the $Q$-function using a function approximation. The most basic approximation is linear function approximation. Other options are possible, including nonlinear neural networks to represent the $Q$-function which is sometimes called ``deep $Q$-learning".

More precisely, in Q-learning with function approximation, we aim to approximate the Q-values \( Q^\ast(s, a) \) using a parameterized function \( Q(s, a; \theta) \). The parameters \( \theta \) are adjusted during learning to minimize the prediction error of the Q-values. Typically the parameters \( \theta \) are adjusted during learning to minimize the prediction error of the Q-values.

The learning objective is typically  minimizing the squared error between the predicted Q-value and the target Q-value:
\[ \quad \min_\theta \; \mathbb{E} \left[ \left( y_t - Q(s_t, a_t; \theta) \right)^2 \right] \] 
where \( y_t \) is the target value for the Q-function at time \( t \), defined by:
\[ y_t = R(s_{t},a_{t},s_{t+1}) + \beta \max_{a'} Q(s_{t+1}, a'; \theta^-) \]

Here, \( \theta^- \) denotes the parameters of the Q-function used to compute the target value, often fixed or updated less frequently to stabilize learning (one way is to periodically update $\theta ^{-}$ by setting it as $\theta$ every $N$ periods). 
The parameters \( \theta \) are typically updated using gradient descent. The gradient of the objective with respect to \( \theta \) is:
\[ \nabla_\theta \left( \left( y_t - Q(s_t, a_t; \theta) \right)^2 \right) = -2 \left( y_t - Q(s_t, a_t; \theta) \right) \nabla_\theta Q(s_t, a_t; \theta) \]

Using this gradient, the update rule for \( \theta \) becomes:
\begin{equation} \label{eq:updatin_thetha}
     \theta_{t+1} = \theta_t + \gamma_t \left( y_t - Q(s_t, a_t; \theta_t) \right) \nabla_\theta Q(s_t, a_t; \theta_t) 
     \end{equation}
where \( \gamma_t \) is the learning rate at time \( t \).

Through this systematic update of \( \theta \), the Q-learning algorithm aims to converge to the optimal policy, provided the function approximator is capable of representing the Q-function accurately and ``good" exploration is maintained.

There are two important cases of function approximation.  

1. \textbf{Linear function approximation.} In linear function approximation we assume that the Q-values are a linear combination of features derived from the states and actions. Mathematically, the Q-function is approximated by:
\[
Q(s,a; \theta) = \theta^T \phi(s,a)
\]
where \( \theta \) is a parameter vector and \( \phi(s,a) \) is a known  feature vector derived from the state \( s \) and action \( a \). The effectiveness of function approximation depends largely on how well these features represent the important characteristics of the state and action that are relevant to predicting rewards and making decisions. There are various ways to determine $\phi(s,a)$ in practice (domain knowledge, auto-encoders, etc) and it depends on the specific application. 

The Q-learning update rule with linear function approximation is:
\[
\theta_{t+1} = \theta_t + \gamma_{t} \delta_t \phi(s_t, a_t)
\]
where \( \delta_t \) is the temporal difference error, defined as:
\[
\delta_t = R(s_{t},a_{t},s_{t+1}) + \beta \max_{a'} Q(s_{t+1}, a'; \theta^{-}) - Q(s_t, a_t; \theta_{t})
\]

2. \textbf{Neural networks approximation (deep $Q$-learning)}. In contrast to linear function approximation, deep Q-learning employs neural networks to approximate the Q-function. This approach is particularly useful in complex environments where the relationship between state, action, and reward cannot be adequately captured by simple linear models. The Q-function in deep Q-learning is represented as:
\[
Q(s, a; \theta) = \text{NN}(s, a; \theta)
\]
where \( \theta \) represents the weights of a neural network, and \( \text{NN}(s, a; \theta) \) denotes the neural network's output for state \( s \) and action \( a \). The neural network can be thought of as a highly non-linear function approximator that can learn to capture complex patterns in the data.

Consider a simple multi-layer perceptron (MLP) as an instance of \( \text{NN}(s, a; \theta) \). This MLP could be structured as follows:
  \[
    \mathbf{h}^{(l)} = \sigma(\mathbf{W}^{(l)} \mathbf{h}^{(l-1)} + \mathbf{b}^{(l)})
    \]
    where \( \mathbf{W}^{(l)} \) and \( \mathbf{b}^{(l)} \) are the weights and biases of the \( l \)-th layer, \( \mathbf{h}^{(l-1)} =  \sigma(\mathbf{W}^{(l-1)} \mathbf{h}^{(l-2)} + \mathbf{b}^{(l-1)})\) is the output of the previous layer (starting with \( \mathbf{h}^{(0)} \) which is the combined input features of state and action), and \( \sigma \) is called a non-linear activation function such as the ReLU function $\sigma(x) = \max \{ x,0\}$. MLPs consist of multiple layers of ``neurons", each connected to the next with a set of weights and biases. This structure allows MLPs to build complex representations of input data by combining the outputs of previous layers in a sophisticated way. Non-linear activation functions allow MLPs to capture non-linear relationships between input and output, which is essential to approximate general functions.
    Theoretically, under certain conditions, MLPs can approximate any continuous function on compact subsets of $\mathbb{R}^n$. This means that MLPs can theoretically learn to approximate the Q-function if the weights are well chosen.  
    
In this setup, \( \theta \) consists of all the weights and biases across all layers $ \theta = (\mathbf{W}^{(1)}, \mathbf{b}^{(1)},\ldots, ) $. The final layer output of the network directly corresponds to \( Q(s, a; \theta) \).

To apply Equation (\ref{eq:updatin_thetha}) we can use the backpropagation algorithm  to compute the gradient of $ Q(s, a; \theta) $  with respect to $\theta$  efficiently 
 using the chain rule (in a smart way).

To conclude, $Q$-learning with function approximation reduces the dimension of the problem, allowing Q-learning to be applied to practical problems. It also allows the agent to generalize from seen to unseen state-action pairs through the features.

\subsubsection{Online Search with Function Approximation }

In function approximation using, say, neural networks, 
using only the neural network function approximation to make decisions can be seen as an ``offline player" as it corresponds to leveraging a pre-trained model to make decisions or predictions in real-time (see also Section \ref{Section:Policy}). In this context, ``offline" refers to the fact that the neural network has been trained on historical data or through prior simulations and does not adapt or learn during actual online gameplay. It uses the learned policy and value functions to make decisions  relying entirely on the knowledge it acquired during the training phase.

There are ways to combine the ``offline" function approximation technique and an ``online player".

For example, a search algorithm such as Monte Carlo Tree Search 
(MCTS), can be considered an ``online solver" because it performs real-time computation to explore possible future states from the current state (say, game position in board games such as Chess). The key feature of MCTS is that it dynamically explores and evaluates potential future moves during each turn of the game, making it highly adaptive to the current state of the game. MCTS  actively explores different possibilities by simulating sequences of actions from the current state instead of full relying on the pre-trained model.

When MCTS is integrated with a function approximation such as a neural network, the approach combines the strengths of both methods:

Online Component (MCTS): MCTS explores possible moves by simulating outcomes and building a tree of possibilities. This helps in assessing the immediate decisions that need to be made based on the current game state. 
The tree search can be guided by the neural network's policy predictions (the policy network), which suggest promising avenues of exploration that are on extensive prior training involving many games or simulations.

Offline Component: After exploring a few steps ahead with MCTS, the neural network's value estimates (the value network) help evaluate positions that are too far in the future to practically reach through direct simulation using the online component.

In essence, MCTS augmented with neural networks is like having an online player that can think a few steps ahead by itself (MCTS) and then consult an expert (the neural network) for its opinion on the resulting positions. This approach has shown promising results in various games such as Chess, Go, etc.

\begin{figure}[h!]
  \centering
  \includegraphics[width=0.7\textwidth]{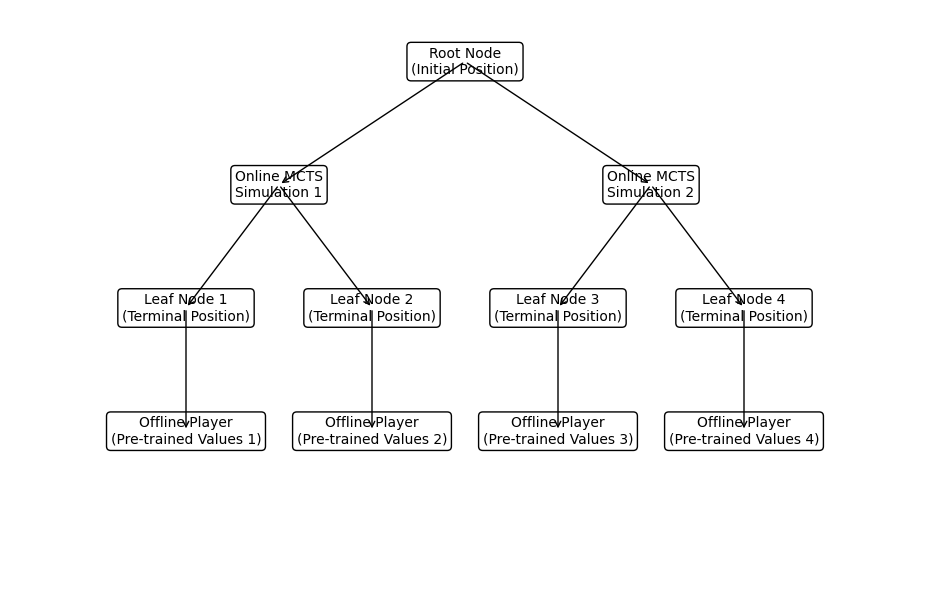}
  \caption{Online search combined with ``offline" pre-trained values.}
  \label{MTCS}
\end{figure}

\subsection{Average Reward Dynamic Programming} \label{Section:Average}

In various real-world applications, the natural objective is the average reward over the horizon instead of  maximizing the cumulative discounted rewards. This is especially relevant in domains where the system operates continually without a discernible endpoint, such as process control, robotics, and in games like \emph{Call of Duty}.

In such environments, an agent continuously interacts within a dynamic setting and the primary objective often revolves around maximizing the agent's effectiveness over time, rather than emphasizing future discounted returns. The average reward framework is apt for such scenarios because it optimizes for steady-state performance, offering an optimal average long-term strategy.

The average reward model focuses on maximizing the long-term average of the rewards received at each time step. For simplicity, we will assume in  this section that the average reward is always well defined (see Assumption \ref{Assumpt:AverageReward}). A policy $\pi(s,a)$ is a probability measure on the (finite) set $A$ for each state $s \in S$ (see Lecture 2 for a discussion on the optimality of stationary policies). We note that each policy $\pi$ generates a Markov chain on the state space $S$ where the probability moving from state $s$ to state $s'$ denoted by $K_{\pi} (s,s')$ is given by 
$$ K_{\pi}(s,s') = \sum _{a \in A} \pi(s,a) p(s,a,s').  $$

We will consider the class of ergodic policies  throughout Section \ref{Section:Average}. Recall that $K_{\pi}$ is irreducible and aperiodic if $K_{\pi}^{t}(s,s') > 0$ for some $t >0$ and all $(s,s') \in S$.  We will say that a policy $\pi$ is ergodic if the Markov chain generated by the policy $K_{\pi}$ is irreducible and aperiodic. To avoid technical complications we will assume that all policies are ergodic throughout this section. In particular, an ergodic policy  $\pi$ generates a Markov chain that has a unique invariant distribution, which we denote by $\lambda_{\pi}$. 

\begin{assumption} \label{Assumpt:AverageReward}
   All possible policies are ergodic. In particular, for any ergodic policy $\pi$, the average reward generated by $\pi$ is well defined, does not depend on the initial state, and is given by
\[
\rho(\pi) = \lim_{T \to \infty} \frac{1}{T} \mathbb{E}_{\pi} \left[\sum_{t=1}^T R(s(t), a(t),s(t+1)) \right]
\]
(see Lecture 2 for the definitions  $(s(t),a(t))$ and $\mathbb{E}_{\pi} $).  
\end{assumption} 

Let $\Pi$ be the set of ergodic policies and define the optimal policy by $\rho = \sup _{\pi \in \Pi} \rho(\pi )$.

For ergodic policies, the law of large numbers for Markov chains holds. That is, if $\lambda_{\pi}$ is the unique stationary distribution generated by the policy $\pi$, then we have  
\begin{equation} \label{Eq:LawofLargeNumbers}
 \rho(\pi) =  \mathbb{E}_{s \sim \lambda_{\pi}, a \sim \pi} \left [  \sum _{s' \in S} p(s,a,s')    R(s,a,s')  \right ] = \sum _{a \in A} \sum _{s \in S} \sum _{s' \in S} \lambda_{\pi}(s) \pi(s,a) p(s,a,s') R(s,a,s').
\end{equation}
We don't formally prove this result, but the intuition behind this result can be understood by analogy with the law of large numbers for independent random variables (see Example \ref{Example:LLN}). For each state $s \in S$, consider the stochastic process
$X_{\tau_{0,s} }, X_{\tau_{1,s} }, \ldots , X_{\tau_{t,s}}, \ldots $  where $\tau_{t}$ represents the $t$-th visit to state $s \in S$ and $X_{\tau_{t,s} }$ represents the payoff at the $t$ visit. Under the assumption of ergodicity, each state $s \in S$ is revisited infinitely often so $\tau_{t,s}$ is finite (with probability $1$). In addition $ \{ X_{\tau_{t,s} } \}_{t=0}^{\infty}$ are independent because of the Markovian structure.

\subsubsection{Bellman Equation for Average Reward DP} \label{Sec:Bellman:average}
In the context of average reward dynamic programming, the structure of the Bellman equation differs from that of the traditional discounted reward framework we studied in Lecture 2. This divergence stems from the objective of optimizing long-term average returns rather than focusing on cumulative discounted rewards.

Although we won't provide a full derivation of the Bellman equation for average rewards (interested readers may refer to \cite{bertsekas2012dynamic} for a comprehensive  formal treatment), we will provide an intuitive explanation. In Lecture 2, we extensively covered the derivation of the Bellman equation for the discounted reward scenario and established that the value function 
$V$ uniquely satisfies this equation.

Moreover, an important link between discounted and average reward formulations exists, articulated through Tauberian theorems in dynamic programming. These theorems show that under specific conditions, the solutions to discounted problems converge to those of average reward problems as the discount factor approaches one. 

More precisely, the Tauberian theorems show that under certain conditions we have
\[
\rho(\pi) (s) = \lim_{\beta \to 1} (1-\beta) \cdot V^\beta_\pi(s)
\]
for each $s \in S$, where $V^\beta_\pi(s)$ is the discounted reward under policy $\pi$ with a discount factor $\beta$ (see Lecture 2). Let $V^{\beta} $ be the value function given a discount factor $\beta$ and define the differential value function $$ h^{\beta} (s) = V^{\beta} (s) - V^{\beta} (s_0)$$ 
where \( s_0 \in S \) is some reference state. From Lecture 2, $V^{\beta}$ satisfies the Bellman equation, so 
$$ V^{\beta} (s)  = \max _{a \in \Gamma(s) }\sum _{s' \in S} p(s,a,s') \left (   R(s,a,s') + \beta  V^{\beta} (s') \right).  $$
Subtracting $V^{\beta} (s_{0})$ from each side of the last equality yields 
$$ (1-\beta) V^{\beta} (s_{0})  + h^{\beta} (s) = \max _{a \in \Gamma(s) }\sum _{s' \in S} p(s,a,s') \left (   R(s,a,s') + \beta  h^{\beta} (s') \right).$$

Under the ergodicity assumption, the average reward satisfies $\rho(s) = \rho$ for each $s \in S$. If we take the limit as \(\beta \to 1\), and it can be argued that the value function \(h^{\beta}\) converges to some function \(h\), (actually it can be argued that the optimal policy remains the same for \(\beta\) sufficiently close to 1, a characteristic referred to as the ``Blackwell optimality") we can use the Tauberian theorem to derive the Bellman equation for the average reward setting. By employing the concept of a ``differential" value function \(h\), the Bellman equation in the average reward context is given by: 
\[
\rho + h(s)  = \max _{a \in \Gamma(s) }\sum _{s' \in S} p(s,a,s') \left (   R(s,a,s') +  h (s') \right).
\]
That is, the Bellman equation for the average reward is given by
\begin{equation} \label{Bellman:Average}
 h(s) = \max _{a \in \Gamma(s) }\sum _{s' \in S} p(s,a,s') \left (   R(s,a,s') - \rho +  h (s') \right)  
\end{equation}
where $h$ is a function from $S$ to $\mathbb{R}$ and $\rho$ is the maximal average reward. A strategy that achieves the maximum in Equation (\ref{Bellman:Average}) is optimal.

We can define as in the previous section the $Q$-function by 
\begin{equation} \label{Eq:Q-functionaverage}
    Q(s,a) = \sum _{s' \in S} p(s,a,s') \left (   R(s,a,s') - \rho +  h (s') \right)
\end{equation}
We also let $h_{\pi }$ be the $h$ function that plays always  according to $\pi$ instead of choosing the optimal action in Equation (\ref{Bellman:Average}) so 
$$ h_{\pi} (s)  = \sum _{a \in A} \pi(s,a) \sum _{s' \in S} p(s,a,s') \left (   R(s,a,s') - \rho(\pi) +  h_{\pi } (s') \right) = \sum _{a \in A} \pi(s,a)Q_{\pi}(s,a)  $$
where
$$Q_{\pi}(s,a) :=  \sum _{s' \in S} p(s,a,s') \left (   R(s,a,s') - \rho (\pi) +  h_{\pi } (s') \right). $$

\subsection{Policy Gradient Methods} \label{Section:Policy}
As we previously discussed, when faced with large or continuous state spaces, traditional dynamic programming techniques often fall short. Policy gradient methods, which adjust policy parameters through gradient ascent to maximize the expected average reward, provide an interesting alternative. These methods are suitable for environments where the policy can be parametrically represented and differentiated with respect to its parameters.

We will focus on average reward dynamic programming with ergodic policies throughout this section. 

Let $\pi_{\theta}$ be a policy that is parameterized by some vector $\theta$ so $\pi_{\theta}(s,a)$ is the probability playing $a \in A$ when the state is $s \in S$ and parameters vector is $\theta$. We assume that $\pi_{\theta}$ is continuously differentiable.

The idea of policy gradient methods within the average reward framework is to identify policy parameters $\theta$ that maximize the average reward. The gradient of the average reward concerning the policy parameters, $\nabla_\theta \rho(\pi_\theta)$, is used to guide the updates to these parameters. Typically, this gradient is estimated via simulation, making it feasible for implementation in real-time (such as in gaming environments). 

The policy gradient update rule in this context is expressed as:
\[
\theta_{k+1} = \theta_k + \gamma_k \nabla_\theta \rho(\pi_{\theta_k})
\]
where $\gamma_k$ denotes the learning rate at iteration $k$. The assumption is that $\pi_{\theta}$ is a probability measure for each $\theta$. A typical choice is the softmax policy  
$$ \pi_{\theta} (s,a ) = \frac{\exp (\theta \cdot \phi(s,a) ) }{\sum_{a' \in A} \exp (\theta \cdot \phi(s,a') )  }$$
where $\phi(s,a)$ is a known feature vector. 

We can also consider projected gradient update  where $\theta$ is projected to the space of probability measures (we will discuss it in the next section). 

Importantly, there is a simple formula to compute $\nabla_\theta \rho(\pi_{\theta_k})$ that can be used to sample from the gradient (see Section \ref{Sec:PolicySimulation}). 
To determine how the optimal value $\rho$ changes with respect to $\theta$, we provide the following Lemma which holds for ergodic policies and is simple to prove by directly differentiating. This result is sometimes called Policy Gradient Formula or Policy Gradient Theorem.  
\begin{lemma} \label{Lemma:PolicyGrad} (Policy Gradient Lemma): Let $\pi_{\theta}$ be an ergodic policy with (unique) stationary distribution $\lambda _{ \pi_{\theta}}$. Assume that $\pi_{\theta}$, $\rho(\pi_{\theta} )$, and $Q_{\pi_{\theta}}(s)$ are differentiable with respect to the parameters $\theta$.  
We have 
\begin{equation} \label{Eq:PolicyGrad}
    \frac{\partial }{\partial \theta} \rho (\pi_{\theta} ) = \sum_{s \in S } \lambda_{{\pi _{ _{\theta}} } } (s) \sum_{a \in A}  \frac{\partial \pi_\theta(s, a)}{\partial \theta} Q_{\pi_\theta}(s, a).
\end{equation}
    
\end{lemma}

\begin{proof}
    Applying the product rule,
 we have
 \begin{align*}
  \frac{ \partial h_{\pi_{\theta}}(s) } {\partial \theta} & =    \frac{ \partial  } {\partial \theta}\sum _{a \in A} \pi _{\theta}(s,a) Q_{\pi_{\theta} } (s,a) \\
  & =  \sum_{a \in A} \left( \frac{\partial \pi_\theta(s, a)}{\partial \theta} Q_{\pi_\theta}(s, a) + \pi_\theta(s, a) \frac{\partial Q_{\pi_\theta}(s, a)}{\partial \theta} \right) \\
  & =  \sum_{a \in A} \left( \frac{\partial \pi_\theta(s, a)}{\partial \theta} Q_{\pi_\theta}(s, a) + \pi_\theta(s, a) \frac{\partial }{\partial \theta} \left (\sum _{s' \in S} p(s,a,s') \left (R(s,a,s') - \rho (\pi_{\theta} ) + h_{\pi _{ _{\theta}}  }(s') \right ) \right) \right) \\
   & =  \sum_{a \in A} \left( \frac{\partial \pi_\theta(s, a)}{\partial \theta} Q_{\pi_\theta}(s, a)   + \pi_\theta(s, a)   \sum _{s' \in S} p(s,a,s') \frac{\partial }{\partial \theta}h_{\pi _{ _{\theta}} }(s') \right) - \frac{\partial }{\partial \theta}\rho (\pi_{\theta} )
 \end{align*}
where in the last equality we used the fact that $\sum _{a \in A} \pi_{\theta} (s,a) = 1$ and the payoff function does not depend on $\theta$. Now because $\lambda_{{\pi _{ _{\theta}} } } (s)$ is the stationary distribution of the Markov chain generated by the policy $\pi _{\theta} $ we have
$$\lambda_{{\pi _{ _{\theta}} } } (s') = \sum _{a \in A} \sum _{s \in S} \pi _{\theta} (s,a) p(s,a,s') \lambda_{{\pi _{ _{\theta}} } } (s) . $$ Multiplying each side of the last inequality by $\lambda_{{\pi _{ _{\theta}} } } (s)$ and summing yields 
\begin{align*}
   \frac{\partial }{\partial \theta}\rho (\pi_{\theta} ) & =  \sum_{s \in S } \lambda_{{\pi _{ _{\theta}} } } (s) \sum_{a \in A}  \frac{\partial \pi_\theta(s, a)}{\partial \theta} Q_{\pi_\theta}(s, a)   +  \sum _{s' \in S}  \sum_{a \in A} \sum_{s \in S } \lambda_{{\pi _{ _{\theta}} } } (s) 
 \pi_\theta(s, a)   p(s,a,s') \frac{\partial }{\partial \theta}h_{\pi _{ _{\theta}} }(s')  \\
 & - \sum_{s \in S } \lambda_{{\pi _{ _{\theta}} } } (s) \frac{\partial }{\partial \theta}h_{\pi _{ _{\theta}} }(s) \\
&  = \sum_{s \in S } \lambda_{{\pi _{ _{\theta}} } } (s) \sum_{a \in A}  \frac{\partial \pi_\theta(s, a)}{\partial \theta} Q_{\pi_\theta}(s, a).
\end{align*}
\end{proof} 

\begin{remark}
    A similar formula to Equation (\ref{Eq:PolicyGrad}) holds for the discounted case but $\lambda$ is a suitable discounted version of the stationary distribution.   
\end{remark}

Thus, the policy gradient method uses the update 
\[
\theta_{k+1} = \theta_k + \gamma_k \nabla_\theta \rho(\pi_{\theta_{k} } ) = \theta_{k} + \gamma_k \sum_{s \in S } \lambda _{\pi _{\theta _{k}} } (s) \sum_{a \in A}  \frac{\partial \pi_{\theta_{k}} (s, a)}{\partial \theta} Q_{\pi_{\theta_{k} } }(s, a)
\]
so we need to estimate $\lambda$ and $Q$ to apply the policy gradient method. This can be done using simulation (see Section \ref{Sec:PolicySimulation}).

\subsubsection{Policy Gradient with Simulation} \label{Sec:PolicySimulation}
Note that if $\pi_{\theta}$ is positive, we can also write 
$$\frac{\partial }{\partial \theta}\rho (\pi_{\theta} )  = \sum_{s \in S } \lambda_{{\pi _{ _{\theta}} } } (s) \sum_{a \in A}  \frac{\partial \pi_\theta(s, a)}{\partial \theta} Q_{\pi_\theta}(s, a) = \sum_{s \in S } \lambda_{{\pi _{ _{\theta}} } } (s) \sum_{a \in A} \frac{\partial \log ( \pi_\theta(s, a))}{\partial \theta}  \pi_\theta(s, a) Q_{\pi_\theta}(s, a). $$

That is,
\[
\frac{\partial }{\partial \theta}\rho (\pi_{\theta} ) = \mathbb{E}_{s \sim \lambda_{\pi_{\theta}}, a \sim \pi_{\theta}} \left[ \frac{\partial \log (\pi_\theta(s, a))}{\partial \theta} Q_{\pi_\theta}(s, a) \right]
\]
which can be useful for a few reasons. First, sometimes computing the log derivative form \( \frac{\partial \log (\pi_\theta(s, a))}{\partial \theta} \) simplifies the computation of gradients, e.g., for softmax policies described in the last section. 
Second, this formulation leverages the expected value under the stationary distribution generated by the policy so it allows to apply standard simulation tools to evaluate expectation from the Markov chains literature such as Monte Carlo. 

This provides an idea of simulation with policy gradient methods (such as an algorithm known as REINFORCE).


Suppose that we can simulate the enviorment as in the Q-learning algorithm section. We can use Monte Carlo simulations to estimate the action-value function by sampling complete trajectories as per $\pi_\theta$. The return from state-action pairs sampled during these trajectories provides an  estimate of $Q_{\pi_\theta}(s_t, a_t)$, say $G_{t}$. The gradient of the expected average reward for an ergodic policy parameterized by $\theta$ can be estimated by
\[
\frac{\partial }{\partial \theta}\rho (\pi_{\theta} )  =  \sum_{t=0}^{H} \frac{\partial \log (\pi_\theta(s_t, a_t))}{\partial \theta}  G_t.   
\]

Given that \( k \) indexes the iteration of the parameter update, the update rule can be written as:

\[
\theta_{k+1} = \theta_k + \gamma_k \sum_{t=0}^{H} \frac{\partial \log (\pi_\theta(s_t, a_t))}{\partial \theta_k} G_t
\]
where $\gamma_{k}$ is the learning rate. 
Note that this formulation assumes that the gradient of the log-policy \( \log (\pi_\theta(s_t, a_t)) \) with respect to the parameters \( \theta \) is computable. Also note that in each iteration, we use a new set of trajectories to compute $G_{t}$ and we have to choose a large $H$ to have a reliable estimate.

To summarize this approach, in each instance we have parameters $\theta$ and we can compute \( \log (\pi_\theta(s_t, a_t)) \). We generate a trajectory $s_{0},a_{0}, R_{0} \ldots, R_{H}$ following the policy $\pi_{\theta }$. We obtain an estimate $G_{t}$ and then update the parameters by the gradient descent equation above.

To choose $G_{t}$ we can set $G_{H+1}=0$, $\bar{\rho} = \sum _{k=1}^{H} R_{k}/H$ and use 
\[ G_t = R_{t} + G_{t+1} - \bar{\rho}.
\]
(compare with Section \ref{Sec:Qapprox}).  
In this setting, we treat $\bar{\rho}$ as a constant that describes the long average payoff under the policy $\pi_{\theta_{k}}$ and the idea is that if  $\mathbb{E} (G_{t+1} |\mathcal{F}_{t+1} ) = Q_{\theta_{k}}(s_{t+1},a_{t+1}) $ (which holds for $t+1=H $) then 
\begin{align*}
    \mathbb{E} [ G_{t} | \mathcal{F}_{t} ] & = \sum _{s_{t+1} \in S} p(s_{t},a_{t},s_{t+1}) R(s_{t},a_{t},s_{t+1}) +  \mathbb{E} [ G_{t+1} | \mathcal{F}_{t} ]- \bar{\rho}  \\
    & =  \sum _{s_{t+1} \in S} p(s_{t},a_{t},s_{t+1})R(s_{t},a_{t},s_{t+1}) +  \mathbb{E} [\mathbb{E} [ G_{t+1} | \mathcal{F}_{t+1} ]| \mathcal{F}_{t} ]- \bar{\rho}   \\
    & =  \sum _{s_{t+1} \in S} p(s_{t},a_{t},s_{t+1})R(s_{t},a_{t},s_{t+1}) +  \mathbb{E} [Q_{\theta_{k}}(s_{t+1},a_{t+1})| \mathcal{F}_{t} ]- \bar{\rho}  \\
    & =  \sum _{s_{t+1} \in S} p(s_{t},a_{t},s_{t+1}) \left ( R(s_{t},a_{t},s_{t+1}) + \sum_{a_{t+1} \in A} \pi_{\theta_{k}} (s_{t+1},a_{t+1}) Q_{\theta_{k}}(s_{t+1}, a_{t+1})- \bar{\rho} \right ) \\ 
    & = Q_{\theta_{k}}(s_{t},a_{t}).
    \end{align*} 
where $\mathcal{F}_{t}$ is the sigma-algebra generated by the history up to time $t$  (see Lecture 2 for more details).

The use of \( G_t \) might introduce high variance in the gradient estimates. There are ways to reduce variance. For instance, introducing an unbiased baseline $b(s)$ can  reduce this variance by using  the update 
\[
\theta_{k+1} = \theta_k + \gamma_k \sum_{t=0}^{H} \frac{\partial \log (\pi_\theta(s_t, a_t))}{\partial \theta_k} (G_t-b(s)).
\]

\subsubsection{Convergence of the Policy Gradient Method}
Suppose that  $\frac{\partial }{\partial \theta} \rho (\pi_{\theta} )$ satisfies condition (i) of Theorem \ref{thm:convergence}. Hence, if the stepsizes satisfies condition (iv) of that theorem, then $\frac{\partial }{\partial \theta} \rho (\pi_{\theta_{k}} )$  converges to $0$ (see the Noisy Gradient Descent discussion in Section \ref{Sec:SGD}). Thus, the policy gradient method converges to a stationary point. 

We can say much more than that for  the tabular case where we update the policy directly $\pi(s,a)$ with $\theta = (s,a)$ as parameters, i.e., we store the policy values for all the state-action pairs at each iteration. This is called direct parameterization and, in this case, $\frac{\partial }{\partial (s,a)} \pi(s,a) =1$. Note that we need to make sure that the policy function remains in the space of probability measures when updating the policy. Hence, we will consider the projected policy gradient algorithm instead of the standard policy gradient algorithm. 
  
  Let $\Delta(X)$ be the set of probability measures over a finite set $X$. Recall that the projection on a closed set $X$ is given by $P_{X}(x)=\operatorname{argmin}_{y \in X} \Vert y-x \Vert ^{2} $. We will consider 
  the projected policy gradient update rule which is expressed as:
\begin{equation} \label{Algo:projectedgradient}
\pi_{k+1} = P_{\Delta (S \times A) } ( \pi_k + \gamma_k \nabla_\pi \rho({\pi_k}) )
\end{equation}
where $\gamma_k$ denotes the learning rate at iteration $k$  and $P_{\Delta (S \times A) }$ is the projection operator in the standard Euclidean norm (which can be computed efficiently). 

The key result of this section is that policy gradient methods in the tabular case converges to a global optimum under certain regularity conditions.  We first need the following result from optimization that has the flavor of a deterministic version of Theorem \ref{thm:convergence} but for projected gradient update rule (see Theorem 10.15 in \cite{beck2017first} for a proof).  

\begin{proposition} \label{prop:projected}
    Let $f:X \rightarrow \mathbb{R}$ be a  function for some compact and convex set $X \subseteq \mathbb{R}^{n}$. Suppose that $\nabla f$ is Lipschitz continuous on $X$ with a constant $L$ (see condition (i) of Theorem \ref{thm:convergence}). Consider the problem $\max _{x \in X} f(x)$. 
    
    Let $x_{k+1}=P_{X} (x_{k} + \alpha \nabla f(x_{k}))$ be the sequence generated by the projected gradient algorithm with $\alpha \in (0,1/L]$. Then $x_{k}$ has a limit point and any limit point $x_{\infty}$ is a stationary point of $f$, i.e., $ (x-x_{\infty})' \nabla f(x_{\infty}) \leq 0$ for all $x \in X$. 
\end{proposition}

    We also need the following simple formula for the difference of the average rewards between two different policies. This result, which is sometimes called the Performance Difference Lemma, can be useful to compare between different policies if we can compute the $Q$ functions and calculate the stationary distributions (this may be done using simulations as we discuss in Section \ref{Sec:PolicySimulation}). Let $$\mathcal{A}_{\pi} (s,a) = Q_{\pi}(s,a) - h_{\pi}(s)$$ be the difference between the $Q$ values and $h$ values which is called the advantage function.

As before, we denote by $\lambda_{\pi}$ the unique stationary distribution over $S$ generated by an ergodic policy $\pi$. 

\begin{lemma} (Performance Difference Lemma). \label{Lemma:performance}
Let $\pi$ and $\pi'$ be any two ergodic policies. Then 
$$ \rho(\pi') - \rho(\pi) =   \mathbb{E}_{s \sim \lambda_{\pi'}, a \sim \pi'} [ \mathcal{A}_{\pi} (s,a) ]. $$    
\end{lemma}

\begin{proof}
    We have 
    \begin{align*}
        \mathbb{E}_{s \sim \lambda_{\pi'}, a \sim \pi'} [ \mathcal{A}_{\pi}  (s,a)] & =  \mathbb{E}_{s \sim \lambda_{\pi'}, a \sim \pi'} \left [  \sum _{s' \in S} p(s,a,s') \left (   R(s,a,s') - \rho (\pi) +  h_{\pi } (s') \right) - h_{\pi} (s) \right ] \\
         & =  \rho(\pi') - \rho(\pi)  + \mathbb{E}_{s \sim \lambda_{\pi'}, a \sim \pi'}  \sum _{s' \in S} p(s,a,s')  h_{\pi }  (s') - \mathbb{E}_{s \sim \lambda_{\pi'}} h_{\pi}(s) 
    \end{align*}
    where the equality follows from Equation (\ref{Eq:LawofLargeNumbers}). 
Now we can use stationarity to conclude that 
    $$ \sum _{a \in A} \sum _{s \in S} \pi' (s,a)\lambda_{ \pi' } (s)  \sum _{s' \in S}  p(s,a,s') h_{\pi}(s') =  \sum _{s' \in S} \lambda_{{\pi' } }  (s')  h_{\pi}(s')   $$
    which proves the lemma.   
\end{proof}

\begin{theorem} \label{Theorem:Policyconvergence}
     Let $c=  \max _{\pi' , \pi \in \Pi^{*}} \max _{s \in S} \lambda _{\pi' }(s) /  \lambda _{\pi } (s) $ where $\Pi^{*}$ is the set of stationary points of $\rho$ and assume that $c$ is finite. Suppose that 
     $ \nabla \rho ( \pi)$ is Lipschitz continuous on $P_{\Delta (S \times A ) }$ with a constant $L$ and the stepsize $\gamma_{k} = \alpha \in (0,1/L]$ is constant. Then $\pi_{k}$ has a limit point and any limit point of the sequence $\pi_{k}$ generated from the projected gradient algorithm (Equation (\ref{Algo:projectedgradient})) is an optimal policy $\pi^{*}$. 
\end{theorem}

\begin{proof}
    The proof follows from the gradient domination lemma below. Indeed, suppose that the lemma holds. From Proposition \ref{prop:projected}, any limit point of $\pi_{k}$ converges to a stationary point $\pi$ of $\rho$. Using the lemma,
    $$\rho(\pi^{*}) - \rho ( \pi) \leq c \max _{\pi' \in \Delta(S \times A)} (\pi' - \pi) ^{T} \nabla \rho ( \pi) \leq 0.$$ That is, $\pi$ is optimal. 
    \end{proof}

  \begin{lemma} (Gradient Domination Lemma). 
   Let $\pi^{*}$ be an optimal policy. For each $\pi \in \Pi^{*}$ we have $$\rho(\pi^{*}) - \rho ( \pi) \leq c \max _{\pi' \in \Delta(S \times A)} (\pi' - \pi) ^{T} \nabla \rho ( \pi) $$
      \end{lemma}

     \begin{proof} We have
      \begin{align*}
          \rho(\pi^{*}) - \rho ( \pi) & =  \sum_{s \in S} \sum _{a \in A} \lambda _{\pi^{*} }(s) \pi^{*}(s,a)  \mathcal{A}_{\pi} (s,a)   \\
          & \leq \sum_{s \in S} \lambda _{\pi^{*} }(s) \max _{a' \in A} \mathcal{A}_{\pi} (s,a') \\
         & \leq c  \sum_{s \in S} \lambda _{\pi}(s) \max _{a' \in A} \mathcal{A}_{\pi} (s,a') \\ 
         & =  c  \sum_{s \in S} \lambda _{\pi}(s) \max _{\mu \in \Delta(A)} \sum _{a \in A} \mu(a)  \mathcal{A}_{\pi} (s,a) \\
         & = c \max _{\pi' \in \Delta(S \times A)} \sum_{s \in S}  \sum _{a \in A} \lambda _{\pi}(s)  \pi'(s,a) \mathcal{A}_{\pi} (s,a)\\
        & =  c \max _{\pi' \in \Delta(S \times A)} \sum_{s \in S}  \sum _{a \in A} \lambda _{\pi}(s)  (\pi'(s,a) - \pi(s,a) )\mathcal{A}_{\pi} (s,a)  \\
         & = c \max _{\pi' \in \Delta(S \times A)} \sum_{s \in S}  \sum _{a \in A}  \lambda _{\pi}(s)  (\pi'(s,a) - \pi(s,a) ) Q_{\pi } (s,a) \\
          & = c \max _{\pi' \in \Delta(S \times A)} \sum_{s \in S}  \sum _{a \in A}  (\pi'(s,a) - \pi(s,a) )   \frac{\partial }{\partial (s,a)}\rho (\pi)  
      \end{align*}
 The first equality follows from Lemma \ref{Lemma:performance}. The second inequality follows because $\max_{a'} \mathcal{A}_{\pi} (s,a') \geq 0$ for each policy $\pi$ and state $s$ (why?). The fourth and fifth equalities use the facts that $ \sum _{a \in A} \pi(s,a) \mathcal{A}_{\pi} =0$ and  $\sum _{a \in A}    (\pi'(s,a) - \pi(s,a) ) h_{\pi}(s) = 0$. The last equality follows from the policy gradient lemma (Lemma \ref{Lemma:PolicyGrad}).  Thus, 
 $$\rho(\pi^{*}) - \rho ( \pi) \leq c \max _{\pi' \in \Delta(S \times A)} (\pi' - \pi)^{T} \nabla \rho ( \pi) $$
 which completes the proof. 
           \end{proof}

\subsubsection{Policy Gradient with Function Approximation}

A popular approach, known as ``actor-critic methods", combines function approximation (see Section \ref{Sec:Qapprox}) and policy gradient methods. It can be preformed on the value function  directly or on the $Q$-function. We consider the $Q$-function actor critic. 
We utilize a parameterized $Q$-function \( Q(s, a; w) \), with the parameters \( w \) adjusted during learning to minimize the prediction error of the Q-values.

We initialize parameters for the policy function $\theta$ and for the value function $w$. An arbitrary state $s_0$ is chosen from the state space, and the average reward $\rho_0$ is initialized to zero. The simulation of the system at each time period $t$ involves action selection according to the policy $\pi_{\theta_t}(\cdot, s_{t})$, observation of the subsequent state $s_{t+1}$, and receipt of reward $R_{t+1}$, then selecting an action $a_{t}'$ according to the policy $\pi_{\theta_t}(\cdot, s_{t+1})$, and updating the weights of the $Q$ function and policy function $w$ and $\theta$.

As in Section \ref{Sec:Qapprox}, the learning objective aims to minimize the squared error between the predicted and target Q-values:
\[
\min_w \mathbb{E} \left[ \left( y_t - Q(s_t, a_t; w) \right)^2 \right],
\]
where \( y_t \) is the target defined as:
\[
y_t = R(s_{t},a_{t},s_{t+1}) + Q(s_{t+1}, a_{t}'; w) - \rho_t.
\]
(see the last part of previous section).  
In this setting, \( y_t \) can be seen as a bootstrap estimate as we update estimates based on sampled data in each step.

The parameters \( w \) are updated using the gradient descent rule (critic update):
\[
w_{t+1} = w_t + \gamma_t \delta_{t} \nabla_w Q(s_t, a_t; w_t),
\]
where \( \delta_t \) is the temporal difference error:
\[
\delta_t = R(s_{t},a_{t},s_{t+1}) - \rho_t + Q(s_{t+1}, a_{t}'; w_{t}) - Q(s_t, a_t; w_{t}),
\]
and the average reward is updated as:
\[
\rho_{t+1} = (1-\epsilon_t) \rho_t + \epsilon_t R_{t+1}.
\]

Policy parameters \(\theta\) are updated using gradient ascent leveraging the policy gradient lemma above (actor update):
\[
\theta_{t+1} = \theta_t + \alpha_t \nabla_\theta \log \pi_{\theta_{t}}(s_t, a_t) Q(s_{t},a_{t};w_{t}),
\]
where \(\alpha_t\) is the learning rate for policy updates. Instead of using directly $ Q(s_{t},a_{t};w_{t})$ we can use an estimate $y_{t}$ or $\delta_{t}$. See Algorithm 1 below for a summary of the method. 

This actor-critic framework facilitates continuous adaptation and improvement of both policy and value function, leveraging the latest experiences from interaction with the environment. There are many possible variants of this approach. We may also consider an episodic setting where we sample a trajectory and update $\theta$ and $w$ at the end of the episode (similar to the approach in the last section).

\begin{algorithm} \label{Algorithm:Actor}
\caption{Q-function Actor-Critic Method}
\begin{algorithmic}[1]
\State Initialize policy parameters \(\theta\), Q-function parameters \(w\), and \(\rho_0 = 0\)
\State Select initial state \(s_0\) from the state space
\State \textbf{repeat} 
    \State \quad \textbf{for each} step \(t\) \textbf{do}
        \State \quad\quad Select action \(a_t\) using policy \(\pi_{\theta_t}(s_t, \cdot)\)
        \State \quad\quad Execute action \(a_t\), observe reward \(R_{t+1}\) and next state \(s_{t+1}\)
        \State \quad\quad Select action \(a_{t}'\) for next state using policy \(\pi_{\theta_t}(s_{t+1}, \cdot)\)
        \State \quad\quad Compute temporal difference error \(\delta_t\):
        \State \quad\quad \(\delta_t = R_{t+1} - \rho_t + Q(s_{t+1}, a_{t}'; w_t) - Q(s_t, a_t; w_t)\)
        \State \quad\quad Update Q-function parameters \(w\):
        \State \quad\quad \(w_{t+1} = w_t + \gamma_t \delta_t \nabla_w Q(s_t, a_t; w_t)\)
        \State \quad\quad Update policy parameters \(\theta\) using policy gradient:
        \State \quad\quad \(\theta_{t+1} = \theta_t + \alpha_t \nabla_\theta \log \pi_{\theta_{t} }(s_t, a_t) Q(s_t, a_t; w_t)\)
        \State \quad\quad Update average reward estimate:
        \State \quad\quad \(\rho_{t+1} = (1 - \epsilon_t) \rho_t + \epsilon_t R_{t+1}\)
        \State \quad\quad Update state:
        \State \quad\quad \(s_t = s_{t+1}\)
    \State \textbf{until} convergence 
\end{algorithmic}
\end{algorithm}

\textbf{Bibliography note}.
A comprehensive introduction to reinforcement learning (RL) can be found in \cite{sutton2018reinforcement}. Detailed information on state aggregation, as discussed in Section \ref{Section:StateAgg}, is available in Section 6.5 of \cite{bertsekas2012dynamic}, Volume II. 
The material covered in Section \ref{Section:Martingales} is well-documented in advanced probability textbooks, such as Chapter 10 of \cite{dudley2018real}. The technical content in Section \ref{Section:Q-learning} is primarily derived from \cite{bertsekas1996neuro}. The proof of the convergence of $Q$-learning provided here is somewhat simpler and less general than that in \cite{bertsekas1996neuro}. 
For a thorough treatment of average reward dynamic programming, as presented in Section \ref{Section:Average}, refer to Chapter 5 of \cite{bertsekas2012dynamic}, Volume II. The convergence results for tabular policy gradient methods in the discounted case are detailed in \cite{agarwal2021theory} and \cite{bhandari2024global}. 
Note that in the operations and optimal control communities, RL is sometimes referred to as approximate dynamic programming.

\subsection{Exercise 4}

\begin{exercise} 
    Let $\Vert x \Vert^{2} = x'x$ be the Euclidean Norm and consider the stochastic iterative scheme
    $$ x_{t+1} = (1-\gamma_{t})x_{t} + \gamma_{t}(W(x_{t}) + Z_{t})$$
    where $x_{0} \in \mathbb{R}^{n}$ and  $W:\mathbb{R}^{n} \rightarrow \mathbb{R}^{n}$ satisfies $\Vert W(x) \Vert \leq \beta \Vert x \Vert$ for all $x \in \mathbb{R}^{n}$ and some $\beta \in (0,1)$. Assume that $W(0)=0$ so the vector $0$ is a fixed point of $W$. 

    $\{Z_{t}\}$ is a sequence of random variables and $\mathcal{F}_{t} = \sigma (Z_{0} , \ldots, Z_{t})$.  Assume that   $$\mathbb{E}[\Vert Z_{t} \Vert ^{2} | \mathcal{F}_{t-1} ]\leq K_{1} + K_{2} \Vert x_{t} \Vert ^{2} \text { and } \mathbb{E}[ Z_{t}  | \mathcal{F}_{t-1} ] = 0 $$
for some $K_{1}$ and $K_{2}$ and the usual stepsize condition (condition (iv) in the convergence theorem in Lecture 4). 

    Show that $x_{t} $ converges to $0$ (the fixed point of $W$). 
\end{exercise}

\begin{exercise} (State Aggregation). 
    Consider a discounted dynamic programming model with action set $A = \{1, 2, 3\}$ and state space $S = \{1, \dots, 100\}$. We explore state aggregation technique discussed in Lecture 4. Assume that the set of aggregate states is $X=\{x_{1},\ldots,x_{10} \}$ and suppose that  $e_{s_{j}y} = 1$ if $s_{j} \in N(y)$ and $d_{xs_{i}} = 1/|N(x)|$.

   (i) Initially, assume that states are aggregated into ten groups sequentially (e.g., $N(x_1) = \{1, \dots, 10\}$, $N(x_2) = \{11, \dots, 20\}$, etc.), with each aggregate state \(x_i\) corresponding to these consecutive states. The reward function is defined by $R(s, a) = sa$, and the transition probabilities are:
    \[ p(s, a, \min(s + 1, 100)) = \frac{1}{0.1s + a} \]
    \[ p(s, a, \max(1, s - 1)) = 1 - p(s, a, \min(s + 1, 100)) \]
    Discuss why this initial sequential aggregation makes sense for this setting. 
    
    Implement a Python script for aggregated value function iteration based on this model.

    You can start with the following code: 

 \begin{lstlisting}
import numpy as np

# Set parameters
num_states = 100
num_actions = 3
num_aggregated_states = 10
discount_factor = 0.95

# Define state transition probabilities and reward function
transitions = np.zeros((num_states, num_actions, num_states))
rewards = np.zeros((num_states, num_actions))

for s in range(num_states):
    for a in range(num_actions):
        action = a + 1  # Actions are 1, 2, 3
        next_state_up = min(s + 1, 99)  # Ensure state indices are within bounds
        next_state_down = max(0, s - 1)
        
        # Define transition probabilities
        prob_up = 1 / (0.1 * (s + 1) + action)  # Adjusted s + 1
        transitions[s, a, next_state_up] = prob_up
        transitions[s, a, next_state_down] = 1 - prob_up
        
        # Define reward function
        rewards[s, a] = (s + 1) * action  # Adjust s+1 for correct state value

# Define aggregation mapping from state to aggregated state
aggregation_map = np.repeat(np.arange(num_aggregated_states), 
                            num_states // num_aggregated_states)
 \end{lstlisting}

   (ii) Assume now that both the rewards and transition probabilities are unknown and must be simulated. Moreover, each state is associated with a randomly generated feature vector in $\mathbb{R}^{2}$ for simplicity. Use the K-means algorithm to determine the aggregation of states based on features, which may not result in consecutive numerical groupings.

    Develop a Python script that employs this data-driven clustering approach to compute the optimal aggregated value function. Start with the following code to cluster the states.  

    Then define the aggregate value function iteration to compute the aggregate value function. 
   \begin{lstlisting}
import numpy as np
from sklearn.cluster import KMeans
import matplotlib.pyplot as plt

# Initialize parameters
num_states = 100
num_actions = 3
num_aggregated_states = 10
discount_factor = 0.95

# Set random seed for reproducibility
np.random.seed(42)

# Generate transition probabilities and rewards
transitions = np.random.rand(num_states, num_actions, num_states)
rewards = np.random.rand(num_states, num_actions)

# Normalize transition probabilities
transitions /= np.sum(transitions, axis=2, keepdims=True)

# Generate features with only two dimensions
features = np.random.rand(num_states, 2)
kmeans = KMeans(n_clusters=num_aggregated_states, random_state=42)
clusters = kmeans.fit_predict(features)

# Plot the features colored by cluster assignment
plt.figure(figsize=(10, 6))
plt.scatter(features[:, 0], features[:, 1], c=clusters, cmap='viridis', marker='o', edgecolor='k', s=50)
for i, center in enumerate(kmeans.cluster_centers_):
    plt.text(center[0], center[1], str(i), fontdict={'weight': 'bold', 'size': 12})
plt.colorbar()
plt.title('State Features Clustered by K-Means')
plt.xlabel('Feature 1')
plt.ylabel('Feature 2')
plt.grid(True)
plt.show()
    \end{lstlisting}
    
    Describe how the clustering outcome might differ from the sequential aggregation and discuss the potential benefits of this data-driven approach.
\end{exercise}

\begin{exercise} (Stochastic Gradient Descent). 
    Consider the stochastic gradient descent algorithm presented in Lecture 4. Suppose that there are constants $K_{1}',K_{2}'$ such  that
\[
\max _{i} \Vert \nabla_\theta L(k_\theta(x_i), y_i) \Vert ^{2} \leq K_{1}' + K_{2}' \Vert \nabla f(\theta) \Vert ^{2}
\]

for all \(\theta\) and some constants \( K_{1}',K_{2}' \).

(i) Show that the convergence theorem in Lecture 4 applies, and hence, $\nabla f(x_{t}) $ converges to $0$. 

Hint: use the analysis of the noisy gradient descent algorithm. 

(ii) Consider a dataset \( \{(x_i, y_i)\}_{i=1}^N \) where \( x_i \in \mathbb{R}^d \) are the input features and \( y_i \in \mathbb{R} \) are the target values. Our goal is to find the parameter vector \( \theta \) that minimizes the least squares error in fitting a linear function \( k_\theta(x) = \theta^T x \).

The objective is to minimize the function \( f \) defined by

\[
f(\theta) = \frac{1}{N} \sum_{i=1}^{N} L(k_\theta(x_{i}), y_i) = \frac{1}{N} \sum_{i=1}^{N} \left( y_i - \theta^T x_i \right)^2 
\]

where the loss function \( L \) is given by the squared error.

Show that the sequence $\theta_{t}$ generated using SGD converges to the unique minimizer of $f$.

(iii) Write a python code implementing SGD. Simulate random data first and then run stochastic gradient descent. Show that the algorithm converges. 

 Start with the following code: 

\begin{lstlisting}
import numpy as np
import matplotlib.pyplot as plt

# Generate synthetic data for demonstration
np.random.seed(0)
N = 1000  # number of data points
d = 2    # number of features
X = np.random.randn(N, d)
true_theta = np.array([2, -3.5])
y = X @ true_theta + np.random.randn(N) 
learning_rate = 0.01
n_iterations = 500
\end{lstlisting}

(iv) Discuss a possible disadvantage of using one sample in each iteration as compared to the classical gradient descent algorithm. 

Hint: variance..

- As a fix for the issue you identified, propose a stochastic gradient descent with batching where we  sample a random subset of data points (instead of one) in each iteration and explains why we can still use the convergence theorem for this case (you don't need to code / prove this part just explain in words).

- Discuss why in practice it may be useful to start with small batches and then increase them overtime.

\end{exercise}

\begin{exercise} (Q-learning). 
        Consider a robot navigating a 8x8 grid world. The robot starts in a random position and needs to reach a target position while avoiding obstacles. The robot can perform four actions at each state: move up, down, left, or right. Actions that would move the robot off the grid leave it in its current location.
    Note that the transitions are deterministic, meaning that chosen actions lead predictably to the next state.

The grid has one designated goal state. The reward is 10 for reaching the goal, $-1$ for hitting an obstacle.
and $-0.1$ for each step taken. 

 In this exercise you will implement episodic $Q$-learning algorithm as studied in Lecture 4 with the parameters below  and an \(\epsilon\)-greedy policy for action selection. Note that the learning rate is constant 0.1. 

(i) First, check the code below. What is the number of states and what is the meaning of a state in terms of the 8x8 grid?

Run episodic $Q$-learning algorithm with 1000 episodes and 200 steps for each episode. In each episode we use the updated $Q$-function from the previous episode, randomize the initial state and apply $Q$-learning with 200 steps. That is, a new episode starts with a randomized initial state but retains the Q-values learned from previous episodes.

(ii)  What is the advantage of this episodic method over using 1000*200 steps in one episode?

(iii) Do you reach the goal? Plot the policy you find for each state in the grid (up, down, left, right, etc). 

 Start with the code below:

\begin{lstlisting}
import numpy as np
import random

# Initialize parameters
grid_size = 8
num_states = grid_size * grid_size
num_actions = 4  # up, down, left, right
discount_factor = 0.99
learning_rate = 0.1
epsilon = 0.1  # exploration rate
episodes = 1000  # number of episodes for training
max_steps = 200  # maximum steps per episode

# Action mapping
actions = {
    0: (-1, 0),  # up
    1: (1, 0),   # down
    2: (0, -1),  # left
    3: (0, 1)    # right
}

# Define rewards and obstacles
target_position = (3, 7)
obstacles = [(1, 2), (1, 3), (3, 1), (3, 2), (3, 6), (6, 6)]
reward_matrix = np.full((grid_size, grid_size), -0.1)
reward_matrix[target_position] = 10
for obs in obstacles:
    reward_matrix[obs] = -1

# Initialize Q-table
Q = np.zeros((num_states, num_actions))

def get_state(row, col):
    return row * grid_size + col

def get_position(state):
    return state // grid_size, state % grid_size

def is_valid_position(row, col):
    return 0 <= row < grid_size and 0 <= col < grid_size

\end{lstlisting}

\end{exercise}

\begin{exercise} (Q-learning for ridesharing matching example). 
   In this exercise you will solve a rideshare matching problem where there are two types of drivers and two types of requests. Each type of driver can match with each type of request, and the rewards for these matches are different. 

   The state of the system at any time \( t \) is represented by a vector \( (d_1, d_2, r_1, r_2) \), where \( d_i \) indicates whether there is a driver of type $i$ available $i=1,2$, ($d_{i}=1$ if available, $0$ if not) and \( r_i \) indicates whether there is a request of type $i$ waiting ($1$ if waiting, $0$ if not).

   The actions correspond to the possible matches between drivers and requests, as well as the option to make no match at all (note that we can match driver of type $i$ with request of type $j$ only if both $d_{i}=1$ and $r_{j}=1$).
\begin{itemize}
    \item \( a_{11} \): Match driver type 1 with request type 1.
    \item \( a_{12} \): Match driver type 1 with request type 2.
    \item \( a_{21} \): Match driver type 2 with request type 1.
    \item \( a_{22} \): Match driver type 2 with request type 2.
    \item \( a_0 \): No match (the system waits for the next time step).
    \item \( a_3 \): Match both drivers with both requests simultaneously. Note that this action is only available in the state 
(1,1,1,1). 
\end{itemize}

When a matching action is taken, the matched driver and request are removed from the system, transitioning the state vector components corresponding to the matched driver and request from 1 to 0. However, after each time step, there is a probability \(p_{id}\) that a removed driver \(d_i\) will reappear in the system (transitioning from 0 to 1) and a probability \(p_{ir}\) that a removed request \(r_i\) will reappear (transitioning from 0 to 1). If the no match action (\(a_0\)) is chosen, the state remains the same, or a driver \(d_i\) may reappear with probability \(p_{id}\) if it was absent, and a request \(r_i\) may reappear with probability \(p_{ir}\) if it was absent.  When \(a_3\) is taken, the system transitions directly to state \((0,0,0,0)\), removing all drivers and requests from the system. After this action, each driver and request has a probability \(p_{id}\) and \(p_{ir}\) of reappearing, transitioning the respective components of the state vector from 0 to 1.

The reward from action $a_{0}$ is zero, from $a_{3}$ is $x_{3}$ and from $a_{ij}$ is $x_{ij}$ for $i=1,2$, $j=1,2$. 

The discount factor is $\beta \in (0,1)$. Assume that $\beta = 0.99$. 

(i) Formulate the ridesharing matching problem above as a discounted dynamic programming problem as in Lecture 2. Define the state space, action space, etc formally and clearly. 

(ii) Write down the Bellman equation for this problem and explains why it holds. 

(iii) Solve the problem using Q-learning with epsilon-greedy and a single long episode.

 Use the parameters in the code below and print the optimal policy. 

How does the optimal policy change when the reward from action $a_{3}$ is $20$?  Explain. 

\begin{lstlisting}
    # Define the parameters
# Rewards
reward_a11 = 3  # Reward for matching driver type 1 with request type 1
reward_a12 = 2  # Reward for matching driver type 1 with request type 2
reward_a21 = 4  # Reward for matching driver type 2 with request type 1
reward_a22 = 3  # Reward for matching driver type 2 with request type 2
reward_a3 = 10  # Reward for matching both drivers with both requests simultaneously
reward_a0 = 0   # Reward for no match

# Probabilities
p_id = [0.5, 0.6]  # Probabilities for drivers d1 and d2 to reappear
p_ir = [0.45, 0.7]  # Probabilities for requests r1 and r2 to reappear

# Hyperparameters
alpha = 0.1  # Learning rate
beta = 0.99  # Discount factor
epsilon = 0.1  # Exploration rate
n_steps = 10000  # Number of steps in the single long episode

# Mapping actions to indices
action_map = {'a11': 0, 'a12': 1, 'a21': 2, 'a22': 3, 'a3': 4, 'a0': 5}


\end{lstlisting}

\end{exercise}

\begin{exercise} (Policy Gradient with State Aggregation).
     You are solving a dynamic programming problem using the average reward criterion, as discussed in Lecture 4. The action space $A$ is relatively small, but the state space $S = \{1, 2, \ldots, 1000\}$ is quite large. To manage the large state space, you decide to apply a tabular policy gradient method (as introduced in Lecture 4) on a reduced state space $S' = \{100, 200, \ldots, 1000\}$, which is a subset of $S$. In this approach, when your policy is $\pi$, for each state from $1$ to $100$, you play according to $\pi(100, \cdot)$, for each state from $101$ to $200$ you play according to $\pi(200, \cdot)$ and similarly for other states. Let $\Pi$ represent the space of such policies defined on $S' \times A$.
        
        We assume that all possible policies are ergodic.
        Show that if $\pi^{*}$ is a stationary point for maximizing the average reward over the set $\Pi$, i.e., $ (\pi-\pi^{*})' \nabla \rho (\pi^{*}) \leq 0$ for all $\pi \in \Pi$, then the following equation holds:
        \[
        \mathbb{E} \, h_{\pi^{*}}(S) = \max_{\pi \in \Pi} \mathbb{E} \left[ \sum_{a \in A} \pi(S,a) Q_{\pi^{*}}(S,a) \right]
        \]
        where $S$ is a random variable and the expectation is taken with respect to the stationary distribution $\lambda_{\pi^{*}}$ (see Lecture 4 for the definitions of $h$, $Q$, $\lambda$).

        Provide an interpretation of the equation above in the context of your approach to solve this problem. You can discuss its meaning in words and / or how it relates to the Bellman equation for the average reward criterion.

\end{exercise}

\begin{exercise} (Combinatorial Optimization as Dynamic Programming). 
An integer programming (IP) problem can be formulated as:
\begin{equation}
\begin{aligned}
& \text{minimize} & & c^T x \\
& \text{subject to} & & Ax \leq b, \text{ }  x \in \mathbb{Z}^n 
\end{aligned}
\end{equation}
where \( c \in \mathbb{R}^n \) is the cost vector, \( A \in \mathbb{R}^{m \times n} \) is the constraint matrix, \( b \in \mathbb{R}^m \) is the constraint vector, and \( x \) is the vector of decision variables constrained to take integer values.

The cutting plane method for solving integer programming begins by solving the linear relaxation of the integer programming problem. This relaxation is obtained by dropping the integrality constraints \( x \in \mathbb{Z}^n \) in the original problem:
\begin{equation} \label{Eq:2}
\begin{aligned}
& \text{minimize} & & c^T x  \\
& \text{subject to} & & Ax \leq b, \text{ } x \in \mathbb{R}^n 
\end{aligned}
\end{equation}
Let \( x^* \) be the optimal solution of this linear relaxation. If \( x^* \) is integer, then it is also optimal for the original integer program. Otherwise, a cutting plane (which is an inequality of the form $a^{T}x \leq \beta$) is generated. A cutting plane is a linear inequality that is satisfied by all feasible integer solutions of the original problem but violated by the current non-integer solution \( x^* \). The cutting plane (which can be shown to always exist) is added to the linear relaxation, creating a new linear program. This new linear program is solved to obtain a new optimal solution. The process of solving the linear relaxation and adding cutting planes is repeated until an integer solution is found or no further progress can be made. 

Assume that for each liner program, there is a finite set of cutting planes $D$ that can be added to the current linear program with a known method (e.g., Gomory’s  cuts for those studied Integer Programming). 

Formulate the cutting plane method to solve integer programming problems as a discounted dynamic programming model. 

Hint: Let $C(1)$ be the original relaxation in Equation (\ref{Eq:2}). Now define $s(t) = (C(t),x^{*}(t), D(t))$ as the state at time $t$ where $C(t)$ is the linear program at state $t$, $x^{*}(t)$ is its solution and $D(t)$ is the set of feasible cuts given the linear program $C(t)$. 

Define the action set as $D(t)$ at stage $t$, i.e., we choose feasible cuts at each stage. 

Define a reward function $R(s,a,s')$ and a discount factor. What is the interpretation of the discount factor? 
What are the transitions?
\end{exercise}

\begin{remark}
    Integer programming has a variety of applications in scheduling, production planning,  matching, etc. However, they are (provably) known to be hard to solve for large state problems. Using reinforcement learning algorithms for the dynamic programming model above, (see \cite{tang2020reinforcement}) or other possible formulations,   can be an interesting direction to solve these problems. 
\end{remark}

\bibliographystyle{ecta}
\bibliography{Lecture}

\section{Solutions to Exercises}

*Solutions may be partial or missing.

\subsection{Exercise 1 Solutions}
\begin{exercise}
    Consider the example from Lecture 1 of the optimization problem $\max _{ \boldsymbol{x} \in A  } x_{1} \cdot x_{2} \cdot \ldots \cdot x_{n}$
where $A = \{ \boldsymbol{x} \in \mathbb{R}^{n} : \sum x_{i} = a, x_{i} \geq 0 \}$. Assume that $n$ and $a$ are non-negative integers. 

1. Write in Python a dynamic programming algorithm to solve this problem for general non-negative integers $n,a$.

2.  Explain, intuitively, why this method is more efficient than a brute-force approach that would involve generating all possible ways to partition the sum 
$a$ into $n$
non-negative parts and then calculating the product for each partition. 

If you studied computer science can you make your argument more formal (what is the (time) complexity of both methods)?
\end{exercise}

1.
\begin{verbatim}
    def max_product(n, a):
    # Edge case: if the total is zero, the product of any numbers must be zero
    if a == 0:
        return 0
    # Edge case: if we only have one number to adjust, it must be 'a'
    if n == 1:
        return a

    # Initialize table where v[i][j] is the max product for i elements summing to j
    v = [[0] * (a + 1) for _ in range(n + 1)]

    # Base case: product of one part is the number itself
    for j in range(a + 1):
        v[1][j] = j

    # Fill DP table
    for i in range(2, n + 1):  # from 2 to n parts
        for j in range(1, a + 1):  # from 1 to a total sum
            # Find the best split
            for k in range(1, j + 1):  # split point, ensuring non-zero products
                product = k * v[i-1][j-k]
                if product > v[i][j]:
                    v[i][j] = product

    return v[n][a]

# Example
n = 1
a = 5
print("Maximum product for n =", n, "and a =", a, "is:", max_product(n, a))

\end{verbatim}

2.  The dynamic programming approach breaks down the problem into simpler subproblems. Each subproblem is solved just once and its solution is stored, allowing these solutions to be reused whenever the same subproblem arises again. This significantly reduces the amount of redundant calculations compared to brute-force methods.
On the other hand, the brute-force approach would generate every possible way to partition the sum $a$ into $n$
 parts and calculate the product for each partition separately. This method does not  reuse overlapping subproblems.

The number of ways to partition 
$a$ into  $n$ parts is
$\binom{a+n-1}{a}$  and for each partition calculating the product takes $O(n)$ so the brute-force approach complexity can be approximated by $O (\binom{a+n-1}{a} n)$. On the other hand, the dynamic programming algorithm has a polynomial complexity $O(na^{2})$ (outer loop runs $n$ times, now count the number of operations for the inner loops).  

\begin{exercise} 
  (a)  Show that $(B(S),d_{\infty})$ is indeed a metric space. 

  (b) Show that the metric space $(B(S),d_{\infty})$ is complete. 
\end{exercise}

We denote $d_{\infty} $ by $d$.

(a) To see that $(B(S),d)$ is a metric space we will show that the triangle inequality holds (the other two properties are immediate). To see this, note that for any $f,g,h \in B(S)$ and each $\omega \in \Omega$ we have 
$$|f(\omega) - g(\omega) | \leq |f(\omega) - h(\omega)| + |h(\omega) - g(\omega)| \leq d(f,h) + d(h,g) $$
so $d(f,g) \leq  d(f,h) + d(h,g)$. 

(b) Let $\{ f_{n} \}$ be a Cauchy sequence in $B(S)$ and let $\epsilon >0$. Then there exists $N$ such that $d(f_{n},f_{m}) < \epsilon$ for all $n,m>N$. For every $\omega \in \Omega$ the inequalities $|f_{n}(s) - f_{m}(s)| \leq d(f_{n},f_{m}) < \epsilon$ imply that $\{ f_{n}(s) \}$ is a Cauchy sequence in $\mathbb{R}$. Hence, $\{ f_{n}(s) \}$ converges to a limit $f(\omega)$. We conclude that $|f_{n}(s) - f(\omega) | \leq \epsilon$ for all $\omega \in \Omega$ and all $n > N$ which in turn implies that $f \in B(S)$ and $d(f_{n},f) \leq \epsilon$ for all $n>N$, i.e., $(B(S),d)$ is complete.

\begin{exercise} \label{Ex: lim A closed}
    Let $A \subseteq X$ be a closed subset of a metric space $(X,d)$. Then $x \in A$ if and only if there exists a sequence $\{x_{n} \}$ of $A$ such that $\lim x_{n} = x$. 
\end{exercise}

Let $x \in A$. For each $n$ pick $x_{n} =x$ so $x_{n} \in A$ and $\lim x_{n} = x$.

Now suppose that the sequence  $\{x_{n} \}$ of $A$ satisfies $\lim x_{n} = x$. Then $B(x,r) \cap A \neq \emptyset$ for each $r >0$. On the other hand, if $x$ belongs to the open set $X \setminus A$ then there exists $r >0$ such that $B(x,r) \subseteq A$ which leads to a contradiction. Hence, $x \in A$.

\begin{exercise}
    Let $(X,d)$ be a complete metric space. Then a subset $A \subseteq X$ is closed if and only if $(A,d)$ is a complete metric space. 
\end{exercise}

Let $A$ be closed. Take a Cauchy sequence $\{x_{n} \}$ of $A$. Then  $\{x_{n} \}$  is a Cauchy sequence of the complete metric space $(X,d)$, and hence, there exists $x \in X$ such that $\lim x_{n} =x$. But since $A$ is closed  Exercise \ref{Ex: lim A closed} implies that $x \in A$ so $(A,d)$ is a complete metric space. 

Now assume that $(A,d)$ is a complete metric space. Suppose that a sequence $\{x_{n} \}$ of $A$ satisfies $\lim x_{n} =x$. Then $\{x_{n} \}$ of $A$  is Cauchy sequence of $A$, so it converges to a unique element of $A$ that must be $x$. Hence, $x \in A$ and $A$ is closed.

\begin{exercise}
   In this exercise we will provide an alternative proof for Banach fixed-point theorem that is based on Cantor's intersection theorem.

Define the diameter of a set $A$ by $d(A) = \sup \{ d(x,y): x,y \in A\}$.

   (a) Let $(X,d)$ be a complete metric space and let $\{A_{n}\}$ be a sequence of closed, non-empty subsets of $X$ such that $A_{n+1} \subseteq A_{n}$ for each $n$ and $\lim d(A_{n}) = 0 $. Then $A = \cap _{n=1}^{\infty} A_{n}$ consists of precisely one element.

(b) Use part (a) to prove Banach fixed-point theorem. 

Hint: Define the function $f(x) = d(x,Tx)$ and consider the sets $A_{n} = \{x \in X: f(x) \leq 1/n \} $ where $T$ is the contraction operator defined in the statement of the Banach fixed-point theorem. 
\end{exercise}

(a) $A$ consists at most one element: If $x,y \in A$ then $d(x,y) \leq d(A_{n})$ for each $n$ so $d(x,y)=0$. 

$A$ consists at least one element: Take $x_{n} \in A_{n}$ for each $n$. It easily follows that $\{x_{n} \}$ is a Cauchy sequence in $X$. Thus, there exists $x \in X$ such that $\lim x_{n} =x$. Because  each $A_{n}$ is closed conclude that $x \in A_{n}$ for each $n$, so $x \in A$. 

(b) $f$ is continuous and $ \lim f(T^{n}x) =0$ for each $x \in X$ because $T$ is contraction. Conclude that $A_{n}$ is closed and not-empty. Now if $x,y \in A_{n}$ then 
$$ d(x,y) \leq d(x,Tx) + d(Tx,Ty) + d(Ty,y) \leq 2/n  + L d(x,y).$$
Hence, $d(A_{n}) \leq 2/n(1-L)$ so $\lim d(A_{n}) =0$. Clearly $A_{n+1} \subseteq A_{n}$. Hence, we can use part (a) to conclude that $A = \cap A_{n}$ consists of exactly one element $z$. Conclude that $z$ is the unique fixed-point of $T$.

\begin{exercise}
     Let $X$ be a non-empty set and let $T: X \rightarrow X$, $L \in (0,1)$. Then the following statements are equivalent: 

     (i) There exists a complete metric $d$ for $X $ such that $d(Tx,Ty) \leq L d(x,y)$ for all $x,y \in X$. 

     (ii) There exists a function $f:X \rightarrow \mathbb{R}_{+}$ such that $f^{-1}(0)$ is a singleton and $f(Tx) \leq L f(x)$. 

     Hint: For $(i) \rightarrow (ii)$ use Banach fixed-point theorem. 

     For $(ii) \rightarrow (i)$ consider $d(x,y) = f(x) + f(y)$ for $x \neq y$ and $d(x,x) = 0$. 
\end{exercise}

Assume that (i) holds. Then $T$ has a unique fixed point $z$. Now define $f(x) = d(x,z)$. Note that  $f^{-1}(0)$ is a singleton and $f(Tx) = d(Tx,z)  = d(Tx, Tz) \leq Ld(x,z) = Lf(x) $. 

Assume that (ii) holds. Define $d(x,y)$ as in the hint. Show that $d$ is indeed a metric on $X$. In addition for $x \neq y$, $d(Tx,Ty) = 0 $ or $d(Tx,Ty) = f(Tx) + f(Tx) \leq Lf(x) + Lf(y) = d(x,y)$ and for $x=y$, $d(Tx, Ty) = 0$. Thus, $T$ is $L$-contraction operator. It remains to show that $d$ is complete. 

Take a Cauchy sequence $\{x_{n} \}$ in $X$. If $\{x_{n}: n \geq 1 \}$ is finite then clearly $x_{n}$ converges so assume that it is infinite. Then there exists a subsequence of $\{x_{n} \}$ denoted also by $x_{n}$ of distinct elements such that $d(x_{n},x_{m}) = f(x_{n}) + f(x_{m})$ for $n \neq m$. Thus, $ \lim f(x_{n}) =0$. From (ii) there exists a $z \in X$ such that $f(z) = 0$. We conclude that $\lim d(x_{n},z) = 0$, i.e., $\lim x_{n} =z$. Hence, each Cauchy sequence converges to $z$ which completes the proof.

 \begin{exercise} Show that for any collection of sets $\mathfrak{B}$ the sigma-algebra that is generated by  $\mathfrak{B}$ exists. 
 \end{exercise}

Let $W = \{ \mathfrak{F}: \mathfrak{B} \subseteq \mathfrak{F},  \mathfrak{F}  \text{ is a sigma-algebra } \} $. 
Now $W$ is not empty ($2^{\Omega}$ belongs to $W$). And one can check that $\cap W$ is the sigma-algebra that is generated by  $\mathfrak{B}$.

   \begin{exercise} 
 Find a Borel measurable function that is not continuous 
 
 Hint: consider the indicator function on the set of rational numbers. 
 \end{exercise}

 As the hint suggests, consider $f = 1_{Q}$ where $Q$ is the set of rational numbers and $1_{Q}$ is the indicator function, i.e., $1_{Q}(x) = 1$ if $x$ is rational and $1_{Q}(x)$ if $x \in \mathbb{R} \setminus Q$. Then $f$ is obviously not continuous. 
 
 The set of rationals is Borel measurable as a countable union of single point sets. Now show that the indicator function of a Borel measurable set is a Borel measurable function. 

\begin{exercise} 
Show that the law of a random variable $X$ is indeed a probability measure. 
\end{exercise}

Let $\mu (A) = \mathbb{P} (\{ \omega \in \Omega : X(\omega) \in A  \} ) =  \mathbb{P} ( X^{-1}(A) ) $ be the law of $X$. $\mu (\Omega) = 1$ follows immediately.  
If $A_{i}$ are pairwise disjoint, justify the following 
equalities 
$$\mu (\cup A_{i} ) = \mathbb{P} ( X^{-1} (\cup  A_{i} ) ) =  \mathbb{P} (  \cup X^{-1}(A_{i}) ) = \sum \mathbb {P} ( X^{-1}(A_{i}) ) = \sum \mu (A_{i}) .$$

\begin{exercise}
    Let $f:\Omega \rightarrow \mathbb{R}$ be a measurable function such that $f(\omega) \geq 0$ for all $\omega$. A measurable function $\phi : \Omega \rightarrow \mathbb{R}$ is called simple if it assumes only finite number of values. 

Show that there exists a sequence of non-negative increasing simple $\phi_{n}$ functions that converges to $f$ pointwise and satisfies $\phi _{n} \leq f$ for all $n$.

Hint: Consider $B_{n}^{i} = \{ \omega \in \Omega: (i-1)2^{-n} \leq f(\omega) < i2^{-n} \} $ and $\phi _{n} = \sum _{i=1}^{n2^n} 2^{-n}(i-1) 1_{B_{n}^{i} } $. 
\end{exercise}

Consider $\phi_{n}$ as in the hint. Note that $\{\phi_{n} \}$ is a sequence of simple functions such that $ 0 \leq \phi_{n}(\omega) \leq \phi_{n+1}(\omega) \leq f(\omega)$ for all $\omega$ and $n$. In addition, use $ 0 \leq f(\omega) - \phi_{n}(\omega) \leq 2^{-n}$ to conclude that $\phi _{n} (\omega) $ converges to $f$ pointwise.

\begin{exercise} 
Suppose that $f,g$ are measurable. 

(1) Show that $ \{ \omega \in \Omega: f(\omega) > g(\omega) \}$, $ \{ \omega \in \Omega: f(\omega) \geq g(\omega) \}$ and $\{ \omega \in \Omega: f(\omega) = g(\omega) \}$ are all measurable. 

(2) Show that $f + g$ is measurable. 

(3) Show that $fg$ is measurable.

Hint: use  $fg = 0.5 ( (f+g)^2 - f^2 - g^2)$.

(4) Show that $f \circ g$ is measurable, i.e.,  the composition of Borel measurable functions is Borel measurable. 

(5)  Suppose that $\{k_{n} \}$ is a sequence of measurable functions. Show that if $k_{n}$ converges to $k$ pointwise then $k$ is measurable. 
\end{exercise}

(1) Let $q_{1},q_{2},\ldots$ be an enumeration of the rational numbers of $\mathbb{R}$. Then 
$$ \{\omega \in \Omega: f(\omega) > g(\omega) \} = \cup _{n=1}^{\infty} \left [ \{ \omega \in \Omega: f(\omega) > q_{n} \} \cap \{ \omega \in \Omega: g(\omega) < q_{n} \}  \right ]$$ 
(why?) is a measurable as a countable union of measurable sets. Now use this to prove the other claims.

(2) For $a \in \mathbb {R}$ we have 
$$ (f+g)^{-1}([a,\infty)) = \{s: f(\omega) +g(\omega) \geq a \} = \{s: f(\omega) \geq a -g(\omega) \} $$
which is measurable as $h(\omega) := a-g(\omega)$ is measurable.

(3) Show that $f^{2}$ and $cf$ are measurable for any $c \in \mathbb{R}$ when $f$ is measurable. Now use the hint. 

(4) Let $f,g$ be Borel measurable functions and let $O$ be an open set.  Then $g^{-1}(O)$ is Borel measurable so $(g \circ f)^{-1} (O ) = f^{-1} (g^{-1}(O))$ is Borel measurable.

(5) Note that $k^{-1}((a,\infty)) = \cup _{n=1}^{\infty} \cap _{i=n}^{\infty} k_{i}^{-1} ((a+1/n,\infty))$ and use the measurability of each $k_{i}$.

\begin{exercise}
    Suppose that $(X,F)$ and $(Y,G)$ are two measurable spaces. The product sigma-algebra denoted by $F \otimes G$ is given by 
    $$ F \otimes G = \sigma  ( \{ A \times B: A \in F, B \in G \} ) $$

where $\sigma$ is defined in Lecture 1. 
    Show that $\mathcal{B} (\mathbb{R}^{n+m}) =\mathcal{B} (\mathbb{R}^{n}) \otimes 
 \mathcal{B} (\mathbb{R}^{m})   $ for $n , m \geq 1$ 
where   $\mathcal{B} (\mathbb{R}^{k})$ is the Borel sigma-algebra on $\mathbb{R}^{k}$. 
\end{exercise}

Let $O(\mathbb{R}^{n+m})$ be the collection of  bounded open intervals in $\mathbb{R}^{n+m}$ (a bounded open interval $I$ means that $I= \prod I_{i}$ where $I_{i}$ is a bounded open interval in $\mathbb{R}$).  Since every open set can be written as a countable union of  bounded open intervals we have  $\mathcal{B} (\mathbb{R}^{n+m}) = \sigma (O(\mathbb{R}^{n+m}))$. Combining this with $O(\mathbb{R}^{n+m}) \subseteq \mathcal{B} (\mathbb{R}^{n}) \otimes 
 \mathcal{B} (\mathbb{R}^{m})  $ implies that $\mathcal{B} (\mathbb{R}^{n+m})  \subseteq \mathcal{B} (\mathbb{R}^{n}) \otimes 
 \mathcal{B} (\mathbb{R}^{m})  $.

 For the second inclusion, consider $M = \{ A \subseteq \mathbb {R}^{m} : A \times \mathbb {R}^{n} \in \mathcal{B} (\mathbb{R}^{n+m}) \}$. Show that $M$ is a sigma algebra that contains all the open sets.  Conclude that $ A \times \mathbb {R}^{n} \in \mathcal{B} (\mathbb{R}^{n+m})$ for every $A \in \mathcal{B} (\mathbb{R}^{m})$. Similarly $ \mathbb {R}^{m} \times B \in \mathcal{B} (\mathbb{R}^{n+m})$ for every $B \in \mathcal{B} (\mathbb{R}^{n})$. Conclude that $ \mathcal{B} (\mathbb{R}^{n}) \otimes 
 \mathcal{B} (\mathbb{R}^{m})  \subseteq  \mathcal{B} (\mathbb{R}^{n+m}) $.

\subsection{Exercise 2 Solutions}

\textbf{Exercise 1}. The Bellman equation is given by 
$$V((\alpha_{1},\beta_{1}),\ldots,(\alpha_{n},\beta_{n})) = \max _{i \in \{1,\ldots,n \}  } \left ( \frac{ \alpha_{i} } {\alpha _{i} + \beta_{i} } + \beta  \left (  \frac{ \alpha_{i} } {\alpha _{i} + \beta_{i} } V (s_{S}) + \left (1-  \frac{ \alpha_{i} } {\alpha _{i} + \beta_{i} } \right )V (s_{F} ) \right ) \right ) $$
where $s_{S} $ (success) is equal to the current state  $(\alpha_{1},\beta_{1}),\ldots,(\alpha_{n},\beta_{n})$ expect that $\alpha_{i}$ is replaced by $\alpha_{i}+1$ and $s_{F} $ (failure) is equal to the current state  $(\alpha_{1},\beta_{1}),\ldots,(\alpha_{n},\beta_{n})$ expect that $\beta_{i}$ is replaced by $\beta_{i}+1$.

The state space is countable and payoffs are bounded so the Bellman Equation holds given the results in Lecture 2.

\textbf{Exercise 2}.
(i) One way to formulate this MDP is the following. Define the state space as $X = \{S, W \}$ where $W = \{Y,N \}$, $Y$ means that the house was sold and $N$ means that it wasn't. The action space is $\{0,1\}$ where $1$ means accept and $0$ means reject. Now we can define $r(i,N,1) = R(i)$,  $r(i,N,0) = 0$, $r(i,Y,1) = r(i,Y,0) = 0$. The discount factor is $\beta$.

You can easily complete the transition probabilities.

(ii)  Clearly, $V(i,Y)=0$. With abuse of notation let's write $V(i)$ instead of $V(i,N)$. Then the Bellman equation is given by
$$ V(i) = \max \{R(i) , \beta \sum_{j \in S} P_{ij} V(j) \}$$

(iii)  When the seller is more patient then the optimal policy is to sell from a higher threshold compared to the impatient seller. 
\begin{verbatim}
import numpy as np
import matplotlib.pyplot as plt

# Parameters
N = 10  # Number of states
beta = 0.99  # Discount factor

# Define the transition probability matrix P
P = np.zeros((N, N))
# Transitions for state 1
P[0, 0] = 1/3  # Stay
P[0, 1] = 2/3  # Move up
# Transitions for states 2 through N-1
for i in range(1, N-1):
    P[i, i-1] = 1/3  # Move down
    P[i, i] = 1/3    # Stay
    P[i, i+1] = 1/3  # Move up
# Transitions for state N
P[N-1, N-2] = 2/3  # Move down
P[N-1, N-1] = 1/3  # Stay

# Reward function R(i) = i (0-indexed, so add 1)
R = np.arange(1, N + 1)

# Initialize value function
V = np.zeros(N)

# Value iteration
tolerance = 1e-6
delta = float('inf')

while delta > tolerance:
    delta = 0
    V_new = np.zeros(N)
    for s in range(N):
        # Compute the value of accepting the offer
        accept = R[s]
        # Compute the value of rejecting the offer
        reject = beta * np.dot(P[s, :], V)
        # Bellman optimality equation
        V_new[s] = max(accept, reject)
        
        delta = max(delta, abs(V_new[s] - V[s]))
    V = V_new.copy()

# Print final values of the value function
print("Final values of the value function:")
for s in range(N):
    print(f"V({s+1}) = {V[s]:.2f}")

# Determine the optimal policy
policy = np.zeros(N, dtype=int)
for s in range(N):
    accept = R[s]
    reject = beta * np.dot(P[s, :], V)
    policy[s] = 1 if accept >= reject else 0  # 1 for accept, 0 for reject

# Plotting the optimal policy
plt.figure(figsize=(10, 6))
plt.subplot(2, 1, 1)  # Subplot 1 for policy
plt.plot(np.arange(1, N + 1), policy, marker='o', linestyle='-', color='b')
plt.title('Optimal Policy as a Function of State (1=Accept, 0=Reject)')
plt.xlabel('State')
plt.ylabel('Action (1=Accept, 0=Reject)')
plt.xticks(np.arange(1, N + 1))
plt.grid(True)

# Plotting the value function
plt.subplot(2, 1, 2)  # Subplot 2 for value function
plt.plot(np.arange(1, N + 1), V, marker='o', linestyle='-', color='r')
plt.title('Value Function as a Function of State')
plt.xlabel('State')
plt.ylabel('Value Function')
plt.xticks(np.arange(1, N + 1))
plt.grid(True)

plt.tight_layout()
plt.show()

\end{verbatim}

(iv) We can use an estimate and optimize approach. We use the samples to estimate the transition probabilities where we define $\hat{P}_{ij} = n_{ij}/\sum_{k=1}^{10} n_{ik} $ where $n_{ij}$ is the number of times state moved from $i$ to $j$. Now we can solve the problem with $\hat{P}_{ij}$ instead of $P_{ij}$. 

\textbf{Exercise 3}.
(i) Note that if $f \geq Tf$ then $f \geq T^{k}f \rightarrow V$ because of the monotonicity of $T$.

So  for a vector $\boldsymbol{x}$ that satisfies the constraints we have $\boldsymbol{x} \geq V$. 

 From the results in Lecture 2, the linear program has a solution where the constraints inequalities hold as equalities. 
 
 Now use the arguments above to conclude that part (i) holds.

 (ii) $|S|$ variables and at most $|S| \times |A|$ constraints.

 (iii) We can solve the linear program with only $n$ variables $w_{1},\ldots,w_{n}$, 
    $$ \min \sum _{s \in S} \sum _{k=1}^{n} w_{k} \phi_{k} (s) \text{  s.t  }  \sum _{k=1}^{n} w_{k} \phi_{k} (s_{i}) \geq r(s_{i},a) + \beta \sum_{s_{j} \in S}p(s_{i},a,s_{j})  \sum _{k=1}^{n} w_{k} \phi_{k} (s_{j}) \text{ for all } a \in \Gamma(s_{i}), \text { } s_{i} \in S. $$

\textbf{Exercise 4}.
As in Lecture 2, we can define the value function for the no discounting case by  $V(s) = \sup _{\sigma} \mathbb{E} _{\sigma} \sum_{t=1}^{T} r(s(t),a(t)) $, $s(1)=s$ where the strategies and expectations are defined similarly to Lecture 2 (with finite $T$). 

(i) No. In this case the Bellman operator is not a contraction and we don't need to use the Banach fixed-point theorem.

The proof follows from the same arguments as in Lecture 2 by showing that the value function is the fixed point of the Bellman equation (verify).

(ii) No. But we do need to assume some integrability conditions so that the sum $\mathbb{E} _{\sigma} \sum_{t=1}^{T} r(s(t),a(t)) $ is well-defined.  For example, we can assume that $r$ is bounded from below (which is also needed for the Remark above).

\textbf{Exercise 5}.
    Let $f_{0} \in B(S)$ and note that $$d(f_{k},T^{k}f_{0}) \leq \sum_{i=0}^{k-1} d(T^{k-i-1}f_{i+1},T^{k-i}f_{i}) \leq \sum _{i=0}^{k-1} \delta \beta ^{i} = \delta \frac{1 - \beta ^{k} } {1 - \beta } $$
    where $T^{0} f = f$. The first inequality follows from the triangle inequality, the second inequality follows from the fact that $T$ is $\beta$-contraction. Now we can use the fact that $T^{k} f_{0} \rightarrow V$ to conclude the result by taking $k \rightarrow \infty$. 

\textbf{Exercise 6}.
\begin{verbatim}
def sumN(numbers, amount):
    # Large number to signify unachievable amount
    max_amount = amount + 1
    V = [max_amount] * (amount + 1)
    V[0] = 0
    
    for i in range(1, amount + 1):
        for number in numbers:
            if i >= number:
                V[i] = min(V[i], V[i - number] + 1)
    
    return V[amount] if V[amount] != max_amount else -1

# Test the function
numbers = [4, 3, 5, 25]
amount = 22
print(sumN(numbers, amount))  



\end{verbatim}

   \textbf{Exercise 7}.

\begin{verbatim}
def word_change(w1, w2):
    m, n = len(w1), len(w2)
    V = [[0] * (n + 1) for _ in range(m + 1)]

    for i in range(m + 1):
        for j in range(n + 1):
            if i == 0:
               V[i][j] = j  # Insert all characters of w2
            elif j == 0:
                V[i][j] = i  # Remove all characters of w1
            elif w1[i-1] == w2[j-1]:
                V[i][j] = V[i-1][j-1]
            else:
                V[i][j] = 1 + min(V[i-1][j], V[i][j-1], V[i-1][j-1])

    return V[m][n]

# Test the function
w1 = "dan"
w2 = "fang"
print(word_change(w1, w2))  

\end{verbatim}

\subsection{Exercise 3 Solutions}

\textbf{Exercise 1}.
Using the hint and because the sum of supermodular functions is supermodular we can construct a sequence of functions $k_{n}(s,a)$ as in Lecture 3  that converges to $z(a,s) := \int f(s')p(s,a,ds')$ such that $k_{n}$ is increasing and supermodular whenever $f$ is. Taking the limit and using the fact that supermodularity and monotonicity is preserved under limits yields the result.

\textbf{Exercise 2}.
A sketch of the solution: 

(i) The state space is $S = X \times \mathcal{Y}$. The Bellman Equation is given by 
$$Tf(x,y_{k}) = \max _{\sum_{i=1}^{n} a_{i} \leq x, a_{i} \geq 0} u \left (x- \sum_{i=1}^{n} a_{i} \right ) + \beta \int  \sum _{j=1}^{m} Q_{kj} f \left ( \sum_{i=1}^{n} R_{i} a_{i} +y_{j} , y_{j} \right ) \phi(dR_{1},\ldots,dR_{n})$$
where $\phi$ is the joint distribution of $R_{i}$'s. 

The Bellman equations holds from the upper semi-continuous dynamic programming we discussed in Lecture 2 (justify..)

(ii) We need to use Theorem 1 part (ii) from Lecture 3  on $(x,a_{1},\ldots,a_{n})$. Note that $X$ is convex and the correspondence $\Gamma$ is convex. The function $u(x-\sum a_{i})$ is jointly concave in $(x,a)$ as the composition of a concave and a linear function. In addition, $\int  \sum _{j=1}^{m} Q_{kj} f \left ( \sum_{i=1}^{n} R_{i} a_{i} +y_{j} , y_{j} \right )  \phi(dR_{1},\ldots,dR_{n})$ is concave in $(a_{1},\ldots,a_{n})$ as the composition of concave and linear functions when $f$ is concave in the first argument (note that concavity is preserved under integration). Hence, we apply   Theorem 1  part (ii) from Lecture 3. 

(iii) We need to use Theorem 1 part (i) from Lecture 3. $u(x-\sum a_{i})$ is increasing in $x$ and the correspondence clearly satisfy the required monotonicity condition. Now if $f$ is increasing then $f \left ( \sum_{i=1}^{n} R_{i} a_{i} +y , y \right )$ is increasing in $y$ for all $R_{i},a_{i}$'s. Hence, $\sum _{j=1}^{m} Q_{kj}   f \left ( \sum_{i=1}^{n} R_{i} a_{i} +y_{j} , y_{j} \right )$ is increasing in $k$ because of the assumption on $Q$ and the characterization of stochastic dominance in Lecture 3. Integrate over $(R_{1},\ldots,R_{n})$ to conclude that $p$ is increasing in $y$. Now use Theorem 1 part (1) Lecture 3. 

(iv) Consider the change variable as in the hint $\sum _{i=1}^{k} a_{i} = z_{k}$. Now 
let $\Gamma(x) = \{z: 0 \leq z_{1} \leq \ldots \leq z_{n} \leq x \}$ and note that it is supermodular (why?). The Bellman equation is given by
$$Tf(x,y_{k}) = \max _{z \in \Gamma(x)} u \left (x- z_{n} \right ) + \beta \int  \sum _{j=1}^{m} Q_{kj} f \left ( \sum_{i=1}^{n} R_{i} (z_{i} - z_{i-1}) +y_{j} , y_{j} \right ) \phi(dR_{1},\ldots,dR_{n})$$ and the optimal policy $\mu$ is given by
 $\mu_{k} (x,y) = \sum _{i=1}^{k} g_{i}(x,y)$ (why)?

Now show that $p$ is supermodular in $x$ and conclude that $\mu_{1} = g_{1}$ is increasing in $x$. So by symmetry we can prove the result. 

(v)
\begin{lstlisting}[language=Python]
    import numpy as np

# Initialize parameters
num_wealth_states = 30
max_wealth = 10  # maximum wealth considered
wealth_grid = np.linspace(0, max_wealth, num_wealth_states)

num_actions = 10  # Discretize the action space into 10 possible investments per asset
action_grid = np.linspace(0, max_wealth, num_actions)

num_income_states = 2
income_states = np.array([1, 2])
transition_probabilities = np.array([[0.7, 0.3], [0.3, 0.7]])  # Transition probabilities for income states

discount_factor = 0.95
return1 = 1.05
return2 = np.array([1.0, 1.16])  # Two possible returns for the second asset

# Utility function
def utility(consumption):
    if consumption <= 0:
        return float('-inf')  # Avoid negative square root and non-positive consumption
    return np.sqrt(consumption)  # Example: square root utility

# Value function initialization
V = np.zeros((num_wealth_states, num_income_states))

# Policy storage
policy = np.full((num_wealth_states, num_income_states, 2), -1, dtype=int)  # Initialize with -1 

# Value iteration function
def value_iteration(V, iterations=1000):
    for _ in range(iterations):
        V_new = np.copy(V)
        for s in range(num_wealth_states):
            for y in range(num_income_states):
                wealth = wealth_grid[s]
                income = income_states[y]
                best_value = float('-inf')
                best_action = None
                for a1 in range(num_actions):
                    for a2 in range(num_actions):
                        investment1 = action_grid[a1]
                        investment2 = action_grid[a2]
                        if investment1 + investment2 > wealth:
                            continue  # Skip over-investment scenarios
                        consumption = wealth - investment1 - investment2
                        future_value = 0
                        for next_y, prob in enumerate(transition_probabilities[y]):
                            for r_index, r_prob in enumerate([0.5, 0.5]):
                                expected_wealth = investment1 * return1 + investment2 * return2[r_index] + income
                                next_s = np.searchsorted(wealth_grid, expected_wealth, side='right') - 1
                                next_s = max(0, min(next_s, num_wealth_states - 1))
                                future_value += prob * r_prob * V[next_s, next_y]
                        total_value = utility(consumption) + discount_factor * future_value
                        if total_value > best_value:
                            best_value = total_value
                            best_action = [a1, a2]
                if best_action is not None:
                    V_new[s, y] = best_value
                    policy[s, y] = best_action  # Update policy only if a valid action was found
        if np.max(np.abs(V_new - V)) < 1e-4:
            break
        V = V_new
    return V, policy

# Run value iteration
optimal_values, optimal_policy = value_iteration(V)

# Output results
print("Optimal Values:\n", optimal_values)
print("Optimal Policy Indices:\n", optimal_policy)

\end{lstlisting}

\textbf{Exercise 3.}
(i)    The Bellman equation for this dynamic programming problem is:
\[
V(p) = \max \left( p, 1-p, \beta \int V\left(\frac{pf_0(z)}{pf_0(z) + (1-p)f_1(z)}\right) d\mu(z) \right)
\]
where
 \(V(p)\) is the value function, representing the maximum expected reward starting from belief \(p\).
   The integral term calculates the expected value if you continue testing and the measure $\mu(z)$  has a density $pf_0(z) + (1-p)f_1(z)$ (so $\mu(dz) = [pf_0(z) + (1-p)f_1(z)]dz$). The Bellman equation holds from Lecture 2 (why?)

(ii) 
To intuitively explain why the value function \( V(p) \) is convex, consider the following scenario where nature plays a game.
Nature flips an unbiased coin.
    With probability \( \lambda \), the coin lands heads, and the belief that \( H_0 \) is correct is \( p_0 \).
   With probability \( 1 - \lambda \), the coin lands tails, and the belief that \( H_0 \) is correct is \( p_1 \).

Before knowing the outcome of the coin flip, our intermediate belief \( p \) about \( H_0 \) can be represented as a weighted average of \( p_0 \) and \( p_1 \):
\[
p = \lambda p_0 + (1 - \lambda) p_1
\]

To show that \( V(p) \) is convex, we need to show that:
\[
V(\lambda p_0 + (1 - \lambda) p_1) \leq \lambda V(p_0) + (1 - \lambda) V(p_1)
\]

Knowing the actual outcome of the coin flip (either \( p_0 \) or \( p_1 \)) provides more precise information about the true state of the world than the intermediate belief \( p \).
    With more precise information, we can make better decisions, leading to higher expected rewards.

Before the coin flip, our expected value based on the intermediate belief \( p \) is \( V(p) \).
     After the coin flip, knowing the specific outcome (either \( p_0 \) or \( p_1 \)) allows us to achieve either \( V(p_0) \) or \( V(p_1) \).
    The expected reward, knowing the coin flip outcome, is:
    \[
    \lambda V(p_0) + (1 - \lambda) V(p_1)
    \]
   This expected reward is higher than or equal to the reward based on the intermediate belief \( p \), because having more information (knowing \( p_0 \) or \( p_1 \)) improves payoffs. 

   \subsection{Exercise 4 Solutions}

\textbf{Exercise 1}
  Define $f(x) = \Vert x \Vert ^{2}$ and $Y_{t} = W(x_{t}) - x_{t} + Z_{t}$ so $x_{t+1} = x_{t} + \gamma_{t}Y_{t}$ as in Lecture 4. 

  Note that $\nabla f(x) = x$ is Lipschitz continuous and from the convergence theorem $\nabla f(x_{t}) = x_{t}$ will converge to $0$ if the gradient direction and mean square errors condition hold. 
  
  For every $x$ we have 
  $$ W(x) \cdot x \leq \Vert W(x) \Vert \cdot \Vert x \Vert \leq \beta  \Vert x \Vert ^{2}$$
  where the first inequality follows from Cauchy-Schwarz inequality and the second from the exercise assumption.  Now subtract $\Vert x \Vert ^{2}$ from both sides to get 
  $$ x \cdot (W(x) - x) \leq - (1-\beta)\Vert x \Vert ^{2}. $$
  Let $c= (1-\beta)$ and note that $\mathbb{E}[ Y_{t}  | \mathcal{F}_{t-1} ] = W(x_{t}) - x_{t}$ to conclude from the last inequality that 
  $$- \nabla f(x_{t}) \cdot  \mathbb{E}[ Y_{t}  | \mathcal{F}_{t-1} ]  \geq c \Vert \nabla f(x_{t}) \Vert ^{2} $$
  so the gradient direction condition holds. Now verify that the mean square condition holds to conclude.

  \textbf{Exercise 3}. 
    (i) Show that the conditions of the noisy gradient algorithm applies. In particular, $\mathbb{E}[ Z_{t}  | \mathcal{F}_{t-1} ] = 0$ (why?) and the mean square condition holds. 

    (ii) Unique minimizer follows from strict convexity of $f$ as the sum of strictly convex functions. Now use the convergence theorem to conclude that $x_{t}$ converges to the unique minimizer $x^{*}$.

    (iii) \begin{lstlisting}
import numpy as np
import matplotlib.pyplot as plt

# Generate synthetic data for demonstration
np.random.seed(0)
N = 1000  # number of data points
d = 2    # number of features
X = np.random.randn(N, d)
true_theta = np.array([2, -3.5])
y = X @ true_theta + np.random.randn(N) 
learning_rate = 0.01
n_iterations = 500

# loss function
def compute_loss(theta, X, y):
    return np.mean((y - X @ theta) ** 2)

# Stochastic Gradient Descent 
def stochastic_gradient_descent(X, y, learning_rate=0.01, n_iterations=500):
    N, d = X.shape
    theta = np.random.randn(d)  # Initial guess 
    loss_values = []

    for t in range(n_iterations):
        i = np.random.randint(N)  # Randomly select an index 
        xi = X[i]
        yi = y[i]
        gradient = -2 * xi * (yi - xi @ theta)  # Gradient of L(theta, z_i)
        theta = theta - learning_rate * gradient  # Update step
        loss_values.append(compute_loss(theta, X, y))

    return theta, loss_values

# Run stochastic gradient descent
theta_min, loss_values = stochastic_gradient_descent(X, y, learning_rate, n_iterations)

# Plot the convergence of the loss function
plt.plot(loss_values)
plt.xlabel('Iteration')
plt.ylabel('Loss')
plt.title('Convergence of Stochastic Gradient Descent')
plt.show()

print(f'The estimated parameters found by SGD: {theta_min}')
    \end{lstlisting}

    (iv) In stochastic gradient descent (SGD), updating the parameters using only one sample per iteration can introduce a high variance in the updates. This is because each update direction is highly dependent on a single randomly chosen data point, which may not represent the overall data distribution well. This high variance can lead to instability of the algorithm. 

To mitigate the high variance issue, we can use a  \textit{mini-batch} stochastic gradient descent. Instead of using a single sample, we use a random subset of the data, called a \textit{mini-batch}, to compute the gradient at each iteration. This approach reduces the variance of the updates while still being computationally efficient compared to classical gradient descent.

Formally, let \( B_t \) be a mini-batch of data samples at iteration \( t \). The parameter update rule becomes:
\begin{equation*}
    x_{t+1} = x_t - \gamma_t \frac{1}{|B_t|} \sum_{z_i \in B_t} \nabla L(x_t, z_i),
\end{equation*}
where \( |B_t| \) is the size of the mini-batch \( B_t \).

In practice, it can be beneficial to start with small mini-batches and gradually increase their size over time. This can lead to fast initial learning and stable convergence as increasing the batch size reduces the variance of the updates.

One way to implement this is to use a schedule for the mini-batch size. For example, start with a small batch size \( B_0 \) and increase it according to a predefined schedule, such as linearly or exponentially, as the number of iterations increases.

\textbf{Exercise 4.}
By starting each episode with a random initial state, the agent is more likely to explore different parts of the state space. This ensures better exploration of the environment, which can lead to better learning. One the other hand, in a single long episode, the agent might get stuck exploring only a subset of the state space, especially if it falls into loops. 
\begin{lstlisting}
import numpy as np
import random

# Initialize parameters
grid_size = 8
num_states = grid_size * grid_size
num_actions = 4  # up, down, left, right
discount_factor = 0.99
learning_rate = 0.1
epsilon = 0.1  # exploration rate
episodes = 1000  # number of episodes for training
max_steps = 200  # maximum steps per episode

# Action mapping
actions = {
    0: (-1, 0),  # up
    1: (1, 0),   # down
    2: (0, -1),  # left
    3: (0, 1)    # right
}

# Define rewards and obstacles
target_position = (3, 7)
obstacles = [(1, 2), (1, 3), (3, 1), (3, 2), (3, 6), (6, 6)]
reward_matrix = np.full((grid_size, grid_size), -0.1)
reward_matrix[target_position] = 10
for obs in obstacles:
    reward_matrix[obs] = -1

# Initialize Q-table
Q = np.zeros((num_states, num_actions))

def get_state(row, col):
    return row * grid_size + col

def get_position(state):
    return state // grid_size, state % grid_size

def is_valid_position(row, col):
    return 0 <= row < grid_size and 0 <= col < grid_size

for episode in range(episodes):
    state = random.randint(0, num_states - 1)
    row, col = get_position(state)
    for step in range(max_steps):
        if random.uniform(0, 1) < epsilon:
            action = random.randint(0, num_actions - 1)
        else:
            action = np.argmax(Q[state, :])
        
        new_row, new_col = row + actions[action][0], col + actions[action][1]
        if not is_valid_position(new_row, new_col):
            new_row, new_col = row, col
        
        new_state = get_state(new_row, new_col)
        reward = reward_matrix[new_row, new_col]
        
        Q[state, action] = Q[state, action] + learning_rate * (
            reward + discount_factor * np.max(Q[new_state, :]) - Q[state, action]
        )
        
        state, row, col = new_state, new_row, new_col

# Print the optimal policy
policy = np.argmax(Q, axis=1).reshape((grid_size, grid_size))
policy_arrows = np.full((grid_size, grid_size), ' ')
for i in range(grid_size):
    for j in range(grid_size):
        if (i, j) == target_position:
            policy_arrows[i, j] = 'G'
        elif (i, j) in obstacles:
            policy_arrows[i, j] = 'X'
        else:
            action = policy[i, j]
            if action == 0:
                policy_arrows[i, j] = 'U'
            elif action == 1:
                policy_arrows[i, j] = 'D'
            elif action == 2:
                policy_arrows[i, j] = 'L'
            elif action == 3:
                policy_arrows[i, j] = 'R'

print("Optimal policy (G: Goal, X: Obstacle):")
print(policy_arrows)


\end{lstlisting}

\textbf{Exercise 5.}

\begin{lstlisting}
    import enum
import random
import numpy as np
from typing import List, Tuple

# Reward system for actions
class RewardValues(enum.IntEnum):
    match_11 = 3
    match_12 = 2
    match_21 = 4
    match_22 = 3
    match_both = 10
    no_match = 0

# Mapping actions to rewards
ACTION_REWARD_MAP = {
    0: RewardValues.match_11,
    1: RewardValues.match_12,
    2: RewardValues.match_21,
    3: RewardValues.match_22,
    4: RewardValues.match_both,
    5: RewardValues.no_match
}

# Probabilities of drivers and requests reappearing
REAPPEARANCE_PROBABILITIES = [0.5, 0.6, 0.45, 0.7]

# Dictionary mapping action names to indices
ACTION_IDX = {'match_11': 0, 'match_12': 1, 'match_21': 2, 'match_22': 3, 'match_both': 4, 'no_match': 5}

# Total number of possible states (2^4 = 16)
STATE_COUNT = 2 ** 4

# Q-learning algorithm
def q_learning_process() -> np.ndarray:
    q_values = np.zeros((STATE_COUNT, len(ACTION_IDX)))
    learning_rate = 0.1
    discount_factor = 0.99
    exploration_rate = 0.1
    steps = 10000  # Adjusted for quicker execution

    for _ in range(steps):
        state_idx = random.randint(0, STATE_COUNT - 1)
        state = decode_state(state_idx)
        possible_actions = get_available_actions(state)

        if random.uniform(0, 1) < exploration_rate:
            chosen_action = random.choice(possible_actions)
        else:
            q_for_actions = q_values[state_idx, possible_actions]
            chosen_action = possible_actions[np.argmax(q_for_actions)]

        next_state_idx, reward_val = perform_action(state_idx, chosen_action)
        q_values[state_idx, chosen_action] += learning_rate * (
            reward_val + discount_factor * np.max(q_values[next_state_idx]) - q_values[state_idx, chosen_action]
        )

    return q_values

# Function to print the optimal policy
def display_optimal_policy(q_values: np.ndarray) -> None:
    print("Computed optimal policy:")
    for i in range(STATE_COUNT):
        current_state = decode_state(i)
        actions_available = get_available_actions(current_state)
        best_action_idx = np.argmax(q_values[i, actions_available])
        best_action = actions_available[best_action_idx]
        action_name = [key for key, val in ACTION_IDX.items() if val == best_action][0]
        print(f"State {current_state}: {action_name}")

# Function to decode state index into state representation (d1, d2, r1, r2)
def decode_state(index: int) -> List[int]:
    driver1 = (index & 8) >> 3
    driver2 = (index & 4) >> 2
    request1 = (index & 2) >> 1
    request2 = index & 1
    return [driver1, driver2, request1, request2]

# Determine valid actions for a given state
def get_available_actions(current_state: List[int]) -> List[int]:
    d1, d2, r1, r2 = current_state
    actions = [5]  # 'no_match' is always possible
    if d1 and r1:
        actions.append(0)
    if d1 and r2:
        actions.append(1)
    if d2 and r1:
        actions.append(2)
    if d2 and r2:
        actions.append(3)
    if d1 and d2 and r1 and r2:
        actions.append(4)
    return actions

# Perform the action and determine the next state and reward
def perform_action(current_state_idx: int, action: int) -> Tuple[int, int]:
    state_rep = decode_state(current_state_idx)
    updated_state = state_rep[:]

    if action == 0:
        updated_state = [0, state_rep[1], 0, state_rep[3]]
    elif action == 1:
        updated_state = [0, state_rep[1], state_rep[2], 0]
    elif action == 2:
        updated_state = [state_rep[0], 0, 0, state_rep[3]]
    elif action == 3:
        updated_state = [state_rep[0], 0, state_rep[2], 0]
    elif action == 4:
        updated_state = [0, 0, 0, 0]
    
    # Simulate reappearance of drivers and requests
    updated_state = [
        1 if updated_state[0] == 0 and random.random() < REAPPEARANCE_PROBABILITIES[0] else updated_state[0],
        1 if updated_state[1] == 0 and random.random() < REAPPEARANCE_PROBABILITIES[1] else updated_state[1],
        1 if updated_state[2] == 0 and random.random() < REAPPEARANCE_PROBABILITIES[2] else updated_state[2],
        1 if updated_state[3] == 0 and random.random() < REAPPEARANCE_PROBABILITIES[3] else updated_state[3]
    ]
    
    next_state_idx = encode_state(updated_state)
    reward = ACTION_REWARD_MAP[action]
    
    return next_state_idx, reward

# Encode state back into index form
def encode_state(state: List[int]) -> int:
    driver1, driver2, req1, req2 = state
    return (driver1 << 3) | (driver2 << 2) | (req1 << 1) | req2

# Run the Q-learning process and display the optimal policy
q_matrix = q_learning_process()
display_optimal_policy(q_matrix)

\end{lstlisting}

 \end{document}